\documentclass[11pt,reqno,a4paper]{amsart}

\usepackage[T1]{fontenc}
\usepackage[utf8]{inputenc}
\usepackage{lmodern}
\IfFileExists{fbb.sty}{\usepackage{fbb}}{}
\usepackage{microtype}

\usepackage{amsmath,amssymb,amsfonts,amsthm,mathtools}
\usepackage{xcolor}
\usepackage{url}
\usepackage[noadjust]{cite}

\usepackage[top=2.5cm,bottom=2.4cm,left=2.8cm,right=2.8cm,headsep=0.2in,headheight=14pt]{geometry}
\usepackage{fancyhdr}
\usepackage[colorlinks=true,
  linkcolor=blue!50!green,
  citecolor=blue!50!green,
  urlcolor=blue!50!green]{hyperref}
\hypersetup{
  pdfauthor={Mayukh Mukherjee, Soumyadeb Samanta and Soumyadip Thandar},
  pdftitle={Linear-growth harmonic functions for nonsymmetric random walks on groups of polynomial growth}
}
\usepackage[nameinlink,capitalize,noabbrev]{cleveref}
\usepackage{aliascnt}

\theoremstyle{plain}
\newtheorem{theorem}{Theorem}[section]
\crefname{theorem}{theorem}{theorems}
\Crefname{theorem}{Theorem}{Theorems}

\newaliascnt{lemma}{theorem}
\newtheorem{lemma}[lemma]{Lemma}
\aliascntresetthe{lemma}
\crefname{lemma}{lemma}{lemmas}
\Crefname{lemma}{Lemma}{Lemmas}

\newaliascnt{proposition}{theorem}
\newtheorem{proposition}[proposition]{Proposition}
\aliascntresetthe{proposition}
\crefname{proposition}{proposition}{propositions}
\Crefname{proposition}{Proposition}{Propositions}

\newaliascnt{corollary}{theorem}
\newtheorem{corollary}[corollary]{Corollary}
\aliascntresetthe{corollary}
\crefname{corollary}{corollary}{corollaries}
\Crefname{corollary}{Corollary}{Corollaries}

\theoremstyle{definition}
\newaliascnt{definition}{theorem}
\newtheorem{definition}[definition]{Definition}
\aliascntresetthe{definition}
\crefname{definition}{definition}{definitions}
\Crefname{definition}{Definition}{Definitions}

\newtheorem{defn}[definition]{Definition}
\crefname{defn}{definition}{definitions}
\Crefname{defn}{Definition}{Definitions}

\newaliascnt{example}{theorem}
\newtheorem{example}[example]{Example}
\aliascntresetthe{example}
\crefname{example}{example}{examples}
\Crefname{example}{Example}{Examples}

\newtheorem{exmp}[example]{Example}
\crefname{exmp}{example}{examples}
\Crefname{exmp}{Example}{Examples}

\theoremstyle{remark}
\newaliascnt{remark}{theorem}
\newtheorem{remark}[remark]{Remark}
\aliascntresetthe{remark}
\crefname{remark}{remark}{remarks}
\Crefname{remark}{Remark}{Remarks}

\theoremstyle{plain}
\newtheorem{mainthm}{Theorem}

\crefname{mainthm}{Theorem}{Theorems}
\Crefname{mainthm}{Theorem}{Theorems}

\crefname{section}{section}{sections}
\Crefname{section}{Section}{Sections}
\crefname{appendix}{Appendix}{Appendices}
\Crefname{appendix}{Appendix}{Appendices}
\crefname{subsection}{subsection}{subsections}
\Crefname{subsection}{Subsection}{Subsections}
\crefname{subsubsection}{subsection}{subsections}
\Crefname{subsubsection}{Subsection}{Subsections}
\crefname{equation}{equation}{equations}
\Crefname{equation}{Equation}{Equations}

\newcommand{\EE}{\mathbb E}
\newcommand{\PP}{\mathbb P}
\newcommand{\RR}{\mathbb R}
\newcommand{\CC}{\mathbb C}
\newcommand{\ZZ}{\mathbb Z}
\newcommand{\NN}{\mathbb N}
\newcommand{\Hom}{\operatorname{Hom}}
\newcommand{\eps}{\varepsilon}
\newcommand{\supp}{\operatorname{supp}}
\newcommand{\BHF}{\mathrm{BHF}}
\newcommand{\HF}{\mathrm{HF}}

\renewcommand{\Re}{\operatorname{Re}}
\newcommand{\LHF}{\operatorname{LHF}}
\newcommand{\SAS}{\operatorname{SAS}}

\begin{document}

\title[Linear-growth harmonic functions]{Linear-growth harmonic functions for nonsymmetric random walks on groups of polynomial growth}
\author{Mayukh Mukherjee}

\address{MM: Department of Mathematics, Indian Institute of Technology Bombay, Powai, Mumbai. 400 076. INDIA.}
\email{mathmukherjee@gmail.com}

\author{Soumyadeb Samanta}
\address{SS: Department of Mathematics, Indian Institute of Technology Bombay, Powai, Mumbai. 400 076. INDIA.}
\email{soumyadeb@iitb.ac.in}

\author{Soumyadip Thandar}
\address{ST: Theoretical Statistics and Mathematics unit, Indian Statistical Institute, Kolkata 700108, INDIA.}
\email{stsoumyadip@gmail.com}


\begin{abstract}
Let $G$ be a finitely generated group of polynomial growth and let $\mu$ be an adapted, Abelian-centered probability measure with a finite exponential moment (not necessarily symmetric or finitely supported). We prove 
that every $\mu$-harmonic function of at most linear growth is globally Lipschitz. In particular, if $G$ is nilpotent, these functions are exactly the affine characters:
$$
\HF_1(G,\mu)=\LHF(G,\mu)=P^1(G).
$$
The analytic machinery required for this is a convolution-gradient estimate that is uniform over families of Abelian-centered, finitely supported probability measures satisfying fixed ellipticity and exponential-moment bounds. A quantitative induction-restriction theorem transfers 
$\LHF$ across finite-index subgroups. We also determine the normed structure of these spaces. In particular, on a nilpotent group, the Lipschitz seminorm of an affine character is exactly the dual stable norm of the word metric. Under the same nonsymmetric hypotheses, linear-growth harmonic functions modulo constants identify canonically with the virtual first cohomology. This identification is an isometry for the asymptotic Lipschitz seminorm. For the ordinary Lipschitz seminorm it is a contraction with bounded inverse; the inverse bound seems to depend on $\mu$, as shown by an infinite dihedral group example.
\end{abstract}

\maketitle

\section{Introduction}
Let $G$ be a finitely generated group and $\mu$ be a probability measure on $G$. We study the $\mu$-harmonic functions having at most linear growth in the setting of virtually nilpotent groups, where $\mu$ is adapted, Abelian-centered, and has a finite exponential moment (some results, like the affine-rigidity result below, require less). Under our hypotheses, we prove that the space of harmonic functions of at most linear growth is finite-dimensional, with dimension one plus the virtual first Betti number of $G$ (first Betti number of any finite-index torsion-free nilpotent subgroup). In this context, we recall the landmark result of Kleiner \cite{Kl10} which shows that on groups of polynomial growth, the space of polynomially growing harmonic functions of a fixed degree is finite dimensional, albeit with stronger assumptions on $\mu$ than what we assume in our paper (see also earlier work of Colding-Minicozzi \cite{CM97}). We further investigate which degree-one functions can be harmonic (this is not automatic due to lack of symmetry of $\mu$), and when linear growth functions necessarily become Lipschitz. We also investigate the (semi)norms carried by the space of Lipschitz $\mu$-harmonic functions.

Write $G_{\mathrm{ab}}=G/[G,G]$ and
$$
\mathbf m_{\mathrm{ab}}(\mu)=\sum_{g\in G}\mu(g)[g]
   \in G_{\mathrm{ab}}\otimes_{\ZZ}\RR.
$$
For an affine character $f(x)=c+\varphi([x])$ with $\varphi \in \Hom(G_{\mathrm{ab}},\RR)$, we notice that $f$ is $\mu$-harmonic precisely when $\varphi$ annihilates $\mathbf m_{\mathrm{ab}}(\mu)=0$. The condition $\mathbf m_{\mathrm{ab}}(\mu)=0$ is therefore the natural centering assumption and is automatic for symmetric measures with finite first moment, but is substantially weaker than symmetry.

Alexopoulos \cite{Al02} studies polynomial-growth harmonic functions for finitely supported random walks. Meyerovitch--Perl--Tointon--Yadin \cite{MPTY17} describe polynomial-growth harmonic functions for symmetric, adapted, finitely supported measures in terms of polynomials on a finite-index nilpotent subgroup. In degree one these polynomials are affine characters. The affine classification of Lipschitz harmonic functions in \Cref{thm:A} requires only absolute convergence of the Abelian drift. Under Abelian centering and a finite exponential moment of the measure, the classification of all linear-growth harmonic functions requires neither symmetry nor finite support.

\begin{mainthm}[Affine rigidity and linear-growth classification]\label{thm:A}
Let $G$ be a finitely generated nilpotent group and let $\mu$ be an adapted probability measure whose Abelian drift
$$
\mathbf m_{\mathrm{ab}}(\mu)=\sum_{g\in G}\mu(g)[g]
$$
converges absolutely in $G_{\mathrm{ab}}\otimes_{\ZZ}\RR$ (in particular, a finite first moment suffices). Then
$$
\LHF(G,\mu)=\{x\mapsto c+\varphi([x]):
c\in\CC, \varphi\in\Hom(G_{\mathrm{ab}},\CC),\
\varphi(\mathbf m_{\mathrm{ab}}(\mu))=0\}.
$$
If $\mu$ is Abelian-centered and has a finite exponential moment, then
$$
\HF_1(G,\mu)=\LHF(G,\mu)=P^1(G).
$$
More generally, let $G$ be a finitely generated group of polynomial growth and let $\mu$ be an adapted, Abelian-centered probability measure with a finite exponential moment. Then
$$
\HF_1(G,\mu)=\LHF(G,\mu),
$$
and this space has complex dimension one plus the virtual first Betti number of $G$. None of these assertions require symmetry or finite support.
\end{mainthm}

The classification of $\LHF$ uses the Choquet--Deny property together with distortion in the lower central series (\Cref{lem:deep-power-distortion}, \Cref{prop:choquet-deny-induction}). 

For the equality $\HF_1 = \LHF$, non-symmetry of $\mu$ is a technical challenge. For a symmetric measure $\mu$ on $G$, the convolution operator $P_\mu(f):=f*\mu$ is a self-adjoint operator on $\ell^2(G)$, and truncation in a symmetric word ball preserves symmetry after the omitted mass is placed at the identity. Neither advantage is available in general for a nonsymmetric Abelian-centered measure. In particular, the truncated measure $\mu\mathbf 1_{B_R}$ usually has a small but nonzero Abelian drift, equal to the negative of the drift in the discarded tail. Under repeated convolution this error accumulates with time, so a fixed truncation does not provide centered estimates at arbitrarily large scales.

We correct the truncated drift by modifying the masses on a fixed finite subset of $\supp(\mu)$. Adaptedness and Abelian centering provide such a subset whose images in $G_{\mathrm{ab}}\otimes_\ZZ\RR$ have the origin in the interior of their convex hull (\Cref{lem:finite-recentering-set}). The correction terms cancel out the drift in the discarded tail and preserve positivity of the truncated measure for large truncation radii (see proof of \Cref{thm:unconditional-HF1-closure}). Since the support of the corrected measure still depends on the radius, the required convolution estimates must be uniform over the resulting family. This is the main content of the following family-uniform version of Dungey's convolution-gradient method \cite{Dungey2008}, which is the main technical contribution of our paper. Apart from providing the missing ingredient required to complete the proof of \Cref{thm:A}, we suspect \Cref{thm:B} could be of independent interest. 

\begin{mainthm}[Family-uniform convolution gradients]\label{thm:B}
Let $G$ be a finitely generated nilpotent group, let $D$ be its polynomial-growth degree (with $D=0$ if $G$ is finite), and let $W$ be a finite symmetric generating set. Suppose $(K_R)_{R\ge R_0}$ is a family of Abelian-centered, finitely supported probability measures such that, for some $\eta,\beta,M>0$,
$$
K_R(e)\ge\eta,\qquad K_R(w)\ge\eta\quad(w\in W),\qquad
\sup_{R\ge R_0}\sum_g e^{\beta|g|_W}K_R(g)\le M.
$$
There exist $a,C>0$, depending only on $(G,W,\eta,\beta,M)$, such that for every $R\ge R_0$, $h\in W$, and $q\ge1$,
$$
\bigl\|e^{a|\cdot|_W/\sqrt q}K_R^{(q)}\bigr\|_2
 \le Cq^{-D/4},
\qquad
\bigl\|e^{a|\cdot|_W/\sqrt q}\partial_hK_R^{(q)}\bigr\|_2
 \le Cq^{-1/2-D/4}.
$$
Consequently,
$$
\|\partial_hK_R^{(q)}\|_1\le Cq^{-1/2},\qquad
\sum_g |g|_W|\partial_hK_R^{(q)}(g)|\le C.
$$
Here, $\partial_h\nu(y)=\nu(h^{-1}y)-\nu(y)$.
\end{mainthm}

The constants are independent of the supports of the measures. Applied to the recentered truncations with radius logarithmic in $q$, these bounds control both the total-variation error and its first moment, yielding the global Lipschitz bound in \Cref{thm:A}.

The affine description also has an exact metric form. If $S$ is a finite symmetric generating set of a nilpotent group, let $\ell_S$ be the stable norm of the word metric on $G_{\mathrm{ab}}\otimes_\ZZ\RR$, and let $\ell_S^*$ be its dual norm. For a general group of polynomial growth, passing from a finite-index subgroup to the whole group can reduce first cohomology, since  every homomorphism on the subgroup might not extend to the whole group. For $\mathbb F\in\{\RR,\CC\}$, the appropriate invariant is the virtual first cohomology
$$
\mathrm{VH}^1(G;\mathbb F):=\varinjlim_N H^1(N;\mathbb F),
$$
where $N$ ranges over the finite-index torsion-free nilpotent subgroups and the transition maps are restrictions. This distinction already occurs for the infinite dihedral group
\[
D_\infty=\langle r,t\mid t^2=e, trt=r^{-1}\rangle.
\]
Indeed, with trivial real coefficients,
\[
H^1(D_\infty;\RR)=\Hom(D_\infty,\RR)=0.
\]
To see this, let $\phi:D_\infty\to\RR$ be a homomorphism. The relation $t^2=e$ gives $\phi(t)=0$, while $trt=r^{-1}$ then gives $\phi(r)=-\phi(r)$ and hence $\phi(r)=0$. On the other hand, every finite-index torsion-free subgroup of $D_\infty$ is of the form $\langle r^m\rangle\cong\ZZ$. Its real first cohomology is one-dimensional, and the restriction maps between such subgroups are isomorphisms. Therefore
\[
\mathrm{VH}^1(D_\infty;\RR)\cong\RR.
\]

\begin{mainthm}[Dual stable norm and virtual first cohomology]\label{thm:C}
For a finitely generated nilpotent group $G$, let $S$ be a finite symmetric generating set and let $\mu$ be an adapted, Abelian-centered probability measure with a finite first moment. For $f(x)=c+\varphi([x])$ the gradient map gives an isometry
$$
(\LHF(G,\mu)/\CC,\|\nabla_S\cdot\|_\infty)
\cong (\Hom(G_{\mathrm{ab}},\CC),\ell_S^*),
\qquad \|\nabla_Sf\|_\infty=\ell_S^*(\varphi).
$$
For a finitely generated group $G$ of polynomial growth, let $S_G$ be a finite symmetric generating set and let $\mu$ be an adapted, smooth and Abelian-centered probability measure on $G$. Then the canonical map
$$
\mathcal A_{G,\mu}:\LHF(G,\mu)/\CC
\longrightarrow\mathrm{VH}^1(G;\CC)
$$
is a linear isomorphism, bi-Lipschitz from the ordinary Lipschitz seminorm to the ambient dual stable norm determined by $S_G$, and satisfies
$$
\bigl\|\mathcal A_{G,\mu}[f]\bigr\|_{\ell_G^\ast}
=\|f\|_{\mathrm{as}}\le\|\nabla_{S_G}f\|_\infty
$$
(see \Cref{def:ambient-stable-vH1,def:asymptotic-lipschitz} below). Thus it is an isometry for the asymptotic Lipschitz seminorm and a contraction for the ordinary Lipschitz seminorm. 
\end{mainthm}
The reverse Lipschitz constant cannot be chosen independently of $\mu$, as shown by  the infinite dihedral group example in \Cref{exmp:dihedral-inflation}.  

The passage from nilpotent groups to virtually nilpotent groups is provided by a quantitative finite-index theorem.

\begin{mainthm}[First-return induction and restriction]\label{thm:D}
Let $G$ be finitely generated, let $H\le G$ have finite index, and fix finite symmetric generating sets $S_G$ and $S_H$. Let $\mu$ be an adapted and smooth probability measure on $G$. If $\mu_H$ is the first-return law on $H$, restriction is a linear isomorphism
$$
\mathrm{Res}_H^G:\LHF(G,\mu)\longrightarrow\LHF(H,\mu_H),
$$
whose inverse is $\mathrm{Ind}_H^G\widetilde f(x)=\EE_x[\widetilde f(X_\tau)]$, where $\tau=\inf\{n\ge0:X_n\in H\}$. Moreover,
$$
\|\nabla_{S_G}\mathrm{Ind}_H^G\widetilde f\|_\infty
 \le C_*\|\nabla_{S_H}\widetilde f\|_\infty,
\qquad
\|\nabla_{S_H}\mathrm{Res}_H^Gf\|_\infty
 \le C_{H,G}\|\nabla_{S_G}f\|_\infty,
$$
with explicit constants given in \Cref{bij:LHF}. These maps are functorial under chains of finite-index subgroups, $H$-equivariant for left translation, and natural under conjugation.
\end{mainthm}

As a consequence, if $G$ is a finitely generated group of polynomial growth and $\mu$ is adapted and smooth, every Lipschitz $\mu$-harmonic function of sublinear growth is constant with no symmetry or centering assumption (\Cref{cor:sublinear-constant}). This is complementary to the work of Sinclair \cite{Sinclair15} (symmetric, finitely supported, large-scale Lipschitz) and to the stationary-environment results of Benjamini--Duminil-Copin--Kozma--Yadin \cite{BDKY15}.

\emph{Organization.} \Cref{sec:preliminaries} deals with the preliminaries. \Cref{sec:LHF-affine} proves affine rigidity, classifies linear growth harmonic functions on virtually nilpotent groups, and introduces virtual first cohomology. \Cref{sec:stable-norms} introduces the stable norms and identifies the Lipschitz seminorm with the dual stable norm. \Cref{sec:finite-index} proves quantitative induction-restriction of Lipschitz harmonic functions across finite index and discusses the relationship between the different seminorms we come across in this paper. Finally, \Cref{sec:uniform-dungey} proves the family-uniform convolution estimates used in the linear-growth classification.
\section{Preliminaries}\label{sec:preliminaries}

\subsection{Harmonic functions and moment assumptions}\label{subsec:harmonic-and-moments}
We retain the conventions of the introduction; the generating set $S$ enters only through the word metric $|\cdot|$, and $\mu$ need not be finitely supported. We refer to \cite{Yadin2024} for background and further references.

\begin{defn}[Symmetric measure]\label{def:symmetric}
The probability measure $\mu$ is \emph{symmetric} if $\mu(g)=\mu(g^{-1})$ for all $g\in G$.
\end{defn}

\begin{defn}[Adapted/non-degenerate measure]\label{defn:non_degenerate_measure}
The probability measure $\mu$ is \emph{adapted} (or non-degenerate) if the semigroup generated by $\supp(\mu)$ equals $G$.
\end{defn}

\begin{defn}[Smooth measure]\label{def:smooth_measure}
A measure $\mu$ on $G$ is \emph{smooth} if there exists $\zeta>0$ such that
\begin{equation}\label{eq:exp_moment}
\Psi(\zeta):=\sum_{x\in G}\mu(x) e^{\zeta|x|}<\infty.
\end{equation}
We say $\mu$ has \emph{superexponential moments} if \eqref{eq:exp_moment} holds for all $\zeta>0$.
\end{defn}

\begin{defn}[SAS measure]\label{def:SAS}
We will use the acronym \emph{SAS} for measures that are \emph{symmetric, adapted, and smooth}.
\end{defn}

In particular, smoothness implies that a $\mu$-distributed increment has an exponential tail and finite first moment in the word metric. We work throughout with the (right) $\mu$-random walk convention
$$
X_{t+1}=X_t \xi_{t+1}\qquad(t\ge0),
$$
where $(\xi_t)_{t\ge1}$ are i.i.d. with law $\mu$.

\begin{defn}[Harmonic function]\label{def:harmonic}
Let $G$ be a group and $\mu$ a probability measure on $G$. A function $f:G\to\CC$ is \emph{$\mu$-harmonic at $k\in G$} if
\begin{equation}\label{eq:def_harmonic}
f(k)=\sum_{g\in G}\mu(g) f(kg),
\end{equation}
and the series converges absolutely. If \eqref{eq:def_harmonic} holds for all $k$, we say $f$ is $\mu$-harmonic on $G$. We denote by $\BHF(G,\mu)$ the vector space of bounded $\mu$-harmonic functions on $G$.
\end{defn}

\begin{remark}\label{rem:harmonic_basic}
(a) $G$ acts on functions by left translation $(g.f)(k):=f(g^{-1}k)$; this preserves $\mu$-harmonicity. \newline (b) If $f$ is $\mu$-harmonic, then $f$ is $\mu^{*n}$-harmonic for every $n\ge1$ (by induction). \newline (c) With our right-walk convention \eqref{eq:def_harmonic}, the (equivalent) left-walk harmonicity would read $f(k)=\sum_g\mu(g)f(gk)$. For symmetric $\mu$ the map $f\mapsto f(\cdot^{-1})$ is an isomorphism between the two spaces; on Abelian groups the two notions coincide.
\end{remark}

\subsection{Polynomials}\label{subsec:polys}
We recall the derivative definition of polynomials (see \cite{Leib02}).

\begin{defn}[Polynomial via discrete derivatives]\label{def:poly_using_derivative}
For $u\in G$ define the left and right differences by
$$
\partial_u f(x):=f(ux)-f(x),\qquad \partial^{u} f(x):=f(xu)-f(x).
$$
Let $H\subset G$. A function $f:G\to\CC$ is a \emph{polynomial of degree $\le k$ with respect to $H$} if
$$
\partial_{u_1}\cdots\partial_{u_{k+1}}f\equiv 0\quad\text{for all }u_1,\dots,u_{k+1}\in H.
$$
If $H=G$ we say $f$ is a (global) polynomial of degree $\le k$ and write $f\in P^k(G)$. By convention $P^k(G)=\{0\}$ for $k<0$.
\end{defn}

\begin{remark}\label{rmk:polynomial_defn_derivative}
(1) Using left or right differences yields the same class of polynomials \cite[Corollary 2.13]{Leib02}. \newline (2) If $S_1,S_2$ are finite generating sets, then $f$ is a polynomial of degree $\le k$ (with respect to $S_1$) iff the same holds with respect to $S_2$ \cite[Lemma 2.2]{Leib02}. Thus, $f\in P^k(G)$ iff $\partial_{u_1}\cdots\partial_{u_{k+1}}f\equiv 0$ for all $u_i$ in any fixed finite generating set $S$.
\end{remark}

\subsection{Lipschitz harmonic functions and harmonic functions of polynomial growth}\label{subsec:Lip_and_growth}
For a function $f:G\to\CC$ and symmetric generating set $S$, define the (right) gradient by $(\nabla f(x))_s:=\partial^{s}f(x)=f(xs)-f(x)$ and the seminorm
$$
\|\nabla_S f\|_\infty:=\sup_{s\in S}\sup_{x\in G}|\partial^{s}f(x)|.
$$
We call $f$ \emph{Lipschitz} if $\|\nabla_S f\|_\infty<\infty$ (independent of the choice of $S$ by \Cref{lem:gradcomp}). Denote by $\LHF(G,\mu)$ the space of $\mu$-harmonic, Lipschitz functions on $G$.

\begin{lemma}[Generator comparability]\label{lem:gradcomp}
If $S_1,S_2$ are finite symmetric generating sets, then there is $C\ge1$ (depending only on $S_1,S_2$) such that
$$
\|\nabla_{S_1} f\|_\infty \le C \|\nabla_{S_2} f\|_\infty\qquad\text{for all } f:G\to\CC.
$$
\end{lemma}

\begin{proof}
Let $C\ge1$ be such that every $s\in S_1$ is a word of $S_2$-length at most $C$; write $s=t_1\cdots t_\ell$ with $\ell\le C$. Then
$$
|\partial^{s}f(x)|
\le \sum_{j=1}^\ell |\partial^{t_j}f(x t_1\cdots t_{j-1})|
\le \ell\|\nabla_{S_2}f\|_\infty\le C\|\nabla_{S_2}f\|_\infty,
$$
and taking the supremum over $x\in G$ and the maximum over $s\in S_1$ gives the claim.
\end{proof}

\noindent
For $t\ge0$ and finite symmetric $S$, define the $t$-growth seminorm
\begin{equation}\label{def:poly_growing_harmonic}
\|f\|_{S,t}:=\limsup_{r\to\infty} r^{-t}\max_{|x|\le r}|f(x)|.
\end{equation}
Different $S$ give equivalent seminorms; we write simply $\|f\|_t$. Set
$$
\HF_t(G,\mu):=\big\{f:G\to\CC: f \text{ is $\mu$-harmonic and } \|f\|_t<\infty\big\}.
$$

\section{Lipschitz harmonic functions and affine structure}\label{sec:LHF-affine}

In this section, we relate degree $1$ polynomials, Lipschitz harmonic functions, and affine characters. 

\medskip
\noindent\textbf{Abelian drift and centering.} Let $\pi_{\mathrm{ab}}:G\to G_{\mathrm{ab}}:=G/[G,G]$ be the canonical projection onto the Abelianization, and for notational convenience write $[g]$ for the class $\pi_{\mathrm{ab}}(g)$ in $G_{\mathrm{ab}}$. Fix a finite symmetric generating set $S$ of $G$. Set $V:=G_{\mathrm{ab}}\otimes_{\ZZ}\RR$, which is a finite-dimensional real vector space. Define the \emph{Abelian drift} of $\mu$ by
\begin{equation}\label{eq:ab_drift}
\mathbf m_{\mathrm{ab}}(\mu) := \sum_{g\in G}\mu(g)[g] \in V.
\end{equation}
The series in \eqref{eq:ab_drift} converges absolutely in $V$ provided $\mu$ has finite first moment.

Indeed, equip $V$ with any norm $\|\cdot\|$ and let $|\cdot|$ denote the word length on $G$ with respect to $S$. Writing $g=s_1\cdots s_{|g|}$ with $s_i\in S$ and using additivity in $V$ gives
$$
\|[g]\| = \Big\|\sum_{i=1}^{|g|}[s_i]\Big\| \le \sum_{i=1}^{|g|}\|[s_i]\| \le C|g|\qquad\text{with }C:=\max_{s\in S}\|[s]\|.
$$
Since $\sum_g \mu(g) |g|<\infty$ by the finite first-moment assumption, we obtain $\sum_g \mu(g) \|[g]\|<\infty$, and hence absolutely convergent in $V$. We say that $\mu$ is \emph{Abelian-centered} if $\mathbf m_{\mathrm{ab}}(\mu)=0$, or equivalently
\begin{equation}\label{eq:centering_equiv}
\sum_{g\in G}\mu(g) \phi\big([g]\big) = 0\qquad\text{for all }\phi\in\Hom(G_{\mathrm{ab}},\RR).
\end{equation}
If $\mu$ is symmetric (i.e., $\mu(g)=\mu(g^{-1})$) and has a finite first moment, then $\mu$ is Abelian-centered. Indeed, since $\phi([g^{-1}])=-\phi([g])$ for any homomorphism $\phi$, the symmetry of $\mu$ leads to pairwise cancellation in the absolutely convergent series \eqref{eq:centering_equiv}.

\begin{lemma}[Degree $1$ $\Rightarrow$ affine character]\label{lem:deg1_affine}
If $f\in P^1(G)$, then for each $u\in G$ the right difference $\partial^{u}f$ is constant on $G$. Consequently,
$$
\phi:G\to\CC,\qquad \phi(u):=\partial^{u}f(e),
$$
is a homomorphism, hence factors through $G_{\mathrm{ab}}$. In particular there exist $c\in\CC$ and a homomorphism $\bar\phi:G_{\mathrm{ab}}\to\CC$ such that
$$
f(x)=c+\bar\phi([x])\qquad(\forall x\in G).
$$
\end{lemma}

\begin{proof}
Since $f\in P^1(G)$, every second right-difference vanishes:
$$
\partial^{v}\partial^{u} f \equiv 0 \qquad(\forall u,v\in G).
$$
Thus, for each fixed $u$, the function $\partial^{u}f$ has degree $\le 0$ and is therefore constant; write
$$
\phi(u) := \partial^{u}f(e)\quad\text{($=\partial^{u}f(x)$ for all $x\in G$)}.
$$
Next, for any $u,v\in G$ we have
\begin{align*}
\phi(uv)&=\partial^{uv}f(e) = f(uv)-f(e) \\
&= \underbrace{f(uv)-f(u)}_{=\partial^{v}f(u)} + \underbrace{f(u)-f(e)}_{=\partial^{u}f(e)}.
\end{align*}
Since $\partial^v f$ is constant, $\partial^v f(u) = \partial^v f(e) = \phi(v)$. By definition, $\partial^u f(e) = \phi(u)$. Thus,
$$
\phi(uv) = \phi(u)+\phi(v).
$$
Hence, $\phi:G\to(\CC,+)$ is a homomorphism. As $(\CC,+)$ is Abelian, $\phi$ factors through $G_{\mathrm{ab}}$ by the universal property of Abelianization, yielding $\bar\phi:G_{\mathrm{ab}}\to\CC$.

Finally, $f(x)-f(e)=\partial^{x}f(e)=\phi(x)$ for all $x\in G$. Hence, $f(x)=c+\bar\phi([x])$ with $c:=f(e)$.
\end{proof}

\begin{lemma}[Length distortion of powers of $G_s$ elements]\label{lem:deep-power-distortion}
Let $G$ be a finitely generated group with a fixed finite symmetric generating set, and let $s\ge1$ be such that the $s$-th lower central subgroup $G_s$ is central in $G$ (equivalently, $G$ is nilpotent of class at most $s$). Then for every $z\in G_s$,
$$
|z^n|=O(n^{1/s})\qquad(n\to\infty),
$$
with implied constant depending only on $z$ and $s$.
\end{lemma}
\begin{proof}
If $s=1$, then $|z^n|\le n|z|$, so the conclusion is immediate. Hence assume $s\ge2$.
Set
$$
\mathcal C_s(S):=\{[u_1,\dots,u_s]:u_i\in S\},
$$
where we use the notation $[u_1, u_2,\dots, u_k]=[[\dots[[u_1, u_2],u_3]\dots, u_k]$. We first show that
$$
G_s=\langle \mathcal C_s(S)\rangle.
$$
Since $G_s\subset Z(G)$ the commutator map is multiplicative in each variable; for example,
$$
[ab,v_2,\dots,v_s]=[a,v_2,\dots,v_s][b,v_2,\dots,v_s],
$$
and similarly in every slot. Writing each $g_i$ as a word in the symmetric generating set $S$ and expanding one slot at a time shows that every $[g_1,\dots,g_s]$ is a product of elements of $\mathcal C_s(S)$. This proves the claim. Moreover, iterating this gives
$$
[u_1^{m_1},\dots,u_s^{m_s}]
=
[u_1,\dots,u_s]^{m_1\cdots m_s}
\qquad (m_1,\dots,m_s\in\NN).
$$
Also, repeated use of $|[a,b]|\le 2|a|+2|b|$ yields a constant $C_s>0$ depending only on $s$ such that for every $g_1,g_2,\dots, g_s \in G$,
$$
|[g_1,\dots,g_s]|\le C_s\sum_{j=1}^s |g_j|.
$$
Hence for $c=[u_1,\dots,u_s]$ (where $u_1,u_2,\dots u_s \in S$) and every $m\in\NN$,
$$
|c^{m^s}|
=
|[u_1^m,\dots,u_s^m]|
\le C_s\sum_{j=1}^s |u_j^m|
\le C_s sm.
$$

Now choose such commutators $c_1,\dots,c_r$ generating $G_s$, and write
$$
z=c_1^{a_1}\cdots c_r^{a_r}.
$$
After replacing some $c_i$ by $c_i^{-1}$ if necessary, we may assume $a_i\ge 0$ for all $i$. Since $G_s$ is central, we have
$$
z^{m^s}=c_1^{a_1m^s}\cdots c_r^{a_rm^s}.
$$
For each $i$, let $c_i=[u_{i1},u_{i2},\dots,u_{is}]$. Then,
$$
|c_i^{a_i m^s}|
=
|[u_{i1}^{a_i m},u_{i2}^m,\dots,u_{is}^m]|
\le C_i m
$$
for a constant $C_i$ depending only on $C_s$ and $a_i$. Therefore
$$
|z^{m^s}|
\le \sum_{i=1}^r |c_i^{a_i m^s}|
\le C_z m
$$
for some constant $C_z$ depending only on $z$ and $s$.

Finally, write
$$
n=\sum_{k=0}^K \varepsilon_k 2^{ks},
\qquad 0\le \varepsilon_k<2^s,
\qquad 2^{Ks}\le n<2^{(K+1)s}.
$$
Using subadditivity of word length,
$$
|z^n|
=
\Big|\prod_{k=0}^K z^{\varepsilon_k 2^{ks}}\Big|
\le \sum_{k=0}^K |z^{\varepsilon_k 2^{ks}}|
\le \sum_{k=0}^K \varepsilon_k|z^{2^{ks}}|
\le C_z 2^s \sum_{k=0}^K 2^k
\le C_z2^{s+1}2^K
\le D n^{1/s}.
$$ 
where $D=C_z2^{s+1}$. Thus $|z^n|=O(n^{1/s})$.
\end{proof}

\begin{proposition}[$\LHF\subseteq P^1$ for adapted measures]\label{prop:choquet-deny-induction}
Let $G$ be a finitely generated nilpotent group and let $\mu$ be an adapted probability measure on $G$. Then every Lipschitz $\mu$-harmonic function on $G$ is a polynomial of degree at most $1$:
$$
\LHF(G,\mu)\subseteq P^1(G).
$$
\end{proposition}

\begin{proof}
Let $f\in\LHF(G,\mu)$, and let $s$ denote the nilpotency class of $G$. We argue by induction on $s$.

\smallskip
\noindent\emph{Base case $s=1$ (Abelian).} For each $c\in G$ the left difference $D_c f(x):=f(cx)-f(x)$ is $\mu$-harmonic. Since $G$ is Abelian, $d(cx,x)=|c|$, so $|D_c f(x)|\le\|\nabla_S f\|_\infty|c|$, and $D_c f$ is bounded. By \cite[Theorem 1]{FHTV19}, $D_c f$ is constant; write $\alpha(c):=D_c f(e)$. The identity $f(cx)=f(x)+\alpha(c)$ for all $c,x$ gives $\alpha(cd)=\alpha(c)+\alpha(d)$, so $\alpha$ is a homomorphism and $f\in P^1(G)$.

\smallskip
\noindent\emph{Inductive step ($s\ge 2$).} Let $G_s$ be the last nontrivial term of the lower central series. Then $G_s\subseteq Z(G)\cap [G,G]$. For each $z\in G_s$, the left difference $D_z f(x):=f(zx)-f(x)$ is $\mu$-harmonic. Since $z\in Z(G)$, $d(zx,x)=|x^{-1}z^{-1}x|=|z^{-1}|=|z|$, so
$$
|D_z f(x)|\le\|\nabla_S f\|_\infty|z|,
$$
which means $D_z f$ is bounded. By the Choquet-Deny property \cite{FHTV19}, $D_z f\equiv\alpha(z)$ for some constant $\alpha(z)\in\CC$.

Now for every $x\in G$ and $n\in\NN$,
$$
f(z^n x)=f(x)+n\alpha(z).
$$
Since $f$ has linear growth ($|f(w)|\le C(1+|w|)$) and $z\in G_s$ with $s\ge 2$, \Cref{lem:deep-power-distortion} gives $|z^n|=O(n^{1/s})$.

Hence
$$
|n\alpha(z)|\le C\big(1+|z^n x|+|x|\big)
\le C_x+C' n^{1/s}.
$$
Dividing by $n$ and letting $n\to\infty$ forces $\alpha(z)=0$. Therefore $f(zx)=f(x)$ for all $z\in G_s$, so $f$ descends to a function $\bar f$ on $G/G_s$ such that $f=\bar f \circ \pi$, where $\pi:G\to G/G_s$ is the canonical projection map.

The quotient $G/G_s$ is nilpotent of class at most $s-1$, the push-forward $\pi_*\mu$ is adapted on $G/G_s$ and $\bar f$ is Lipschitz $(\pi_*\mu)$-harmonic (since $\|\nabla_{\pi(S)}\bar f\|_\infty\le\|\nabla_S f\|_\infty$). By the induction hypothesis, $\bar f\in P^1(G/G_s)$ which also means $f\in P^1(G)$.
\end{proof}

\begin{theorem}[Affine rigidity for $\LHF$]\label{thm:HF_1-is-linear-poly}
Let $G$ be a finitely generated nilpotent group, and let $\mu$ be an adapted probability measure on $G$ (no symmetry assumed) such that the Abelian drift
$$
\mathbf m_{\mathrm{ab}}(\mu)=\sum_{g\in G}\mu(g)[g]
$$
converges absolutely in $G_{\mathrm{ab}}\otimes_\ZZ\RR$ (for example, it suffices that $\mu$ have finite first moment). Then
$$
\LHF(G,\mu)=\Big\{x\mapsto c+\bar\phi([x]) : c\in\CC,
\bar\phi\in\Hom(G_{\mathrm{ab}},\CC), \bar\phi\big(\mathbf m_{\mathrm{ab}}(\mu)\big)=0\Big\}.
$$
In particular, if $\mu$ is Abelian-centered then $\LHF(G,\mu)=P^1(G)$, and every Lipschitz $\mu$-harmonic function is affine.
\end{theorem}

\begin{proof}
Let $f(x)=c+\bar\phi([x])$ with $\bar\phi\big(\mathbf m_{\mathrm{ab}}(\mu)\big)=0$. For any $k\in G$,
$$
\sum_{g}\mu(g) f(kg)
= \sum_{g}\mu(g) \big(c+\bar\phi([k])+\bar\phi([g])\big)
= f(k) + \bar\phi\big(\mathbf m_{\mathrm{ab}}(\mu)\big)
= f(k).
$$
Thus $f$ is $\mu$-harmonic. Moreover, $\partial^{s}f(x)=\bar\phi([s])$ for all $x\in G$, so $f$ is Lipschitz with $\|\nabla_S f\|_\infty=\max_{s\in S}|\bar\phi([s])|$.

\medskip
Conversely, let $f\in \LHF(G,\mu)$. By \Cref{prop:choquet-deny-induction} and  \Cref{lem:deg1_affine}, $f(x)=c+\bar\phi([x])$ for some $\bar\phi\in\Hom(G_{\mathrm{ab}},\CC)$. Since $f$ is $\mu$-harmonic, we have
$$
0=\sum_g\mu(g)\bigl[f(kg)-f(k)\bigr]
=\sum_g\mu(g)\bar\phi([g])
=\bar\phi\big(\mathbf m_{\mathrm{ab}}(\mu)\big),
$$
so $\bar\phi$ annihilates the Abelian drift.
\end{proof}

\begin{corollary}[Norm identification]\label{cor:norm_ident_vnilp}
Let $G$ be finitely generated and nilpotent, $\mu$ adapted and smooth, and $S$ a finite symmetric generating set. Then the map $(\bar\phi,c)\mapsto f(x):=c+\bar\phi([x])$ is an isomorphism
$$
\big\{(\bar\phi,c)\in\Hom(G_{\mathrm{ab}},\CC)\times\CC:
  \bar\phi\big(\mathbf m_{\mathrm{ab}}(\mu)\big)=0\big\}\to\LHF(G,\mu),
$$
with $\|\nabla_S f\|_\infty=\max_{s\in S}|\bar\phi([s])|$. In particular, if $\mu$ is Abelian-centered then 
$$
\LHF(G,\mu)\cong\Hom(G_{\mathrm{ab}},\CC)\oplus\CC.
$$
\end{corollary}

\begin{proof}
This is a restatement of \Cref{thm:HF_1-is-linear-poly}: the displayed set is $\LHF(G,\mu)$, the map $(\bar\phi,c)\mapsto[x\mapsto c+\bar\phi([x])]$ is the asserted isomorphism, and the seminorm identity $\|\nabla_S f\|_\infty=\max_{s\in S}|\bar\phi([s])|$ is the computation $\partial^s f=\bar\phi([s])$ carried out there. When $\mu$ is Abelian-centered the constraint $\bar\phi(\mathbf m_{\mathrm{ab}}(\mu))=0$ is vacuous.
\end{proof}

\begin{theorem}[Isometric identification]\label{thm:coho-rigidity}
Let $G$ be finitely generated and nilpotent, and let $\mu$ be adapted and have a finite first moment. Assume that $\mu$ is Abelian-centered: $\mathbf m_{\mathrm{ab}}(\mu)=0$. Then the map
$$
\Theta: \LHF(G,\mu)/\CC \longrightarrow \Hom(G_{\mathrm{ab}},\CC),\qquad
\Theta([f])([s]_G) := \partial^{s}f(e),
$$
is a well-defined linear \emph{isometric} isomorphism, with inverse $\varphi\mapsto [x\mapsto \varphi([x]_G)]$.
\end{theorem}

\begin{proof}
By \Cref{thm:HF_1-is-linear-poly}, $\LHF(G,\mu)=\{c+\varphi([x]_G) :\varphi\in\Hom(G_{\mathrm{ab}},\CC), \varphi(\mathbf m_{\mathrm{ab}}(\mu))=0\}$. Under centering, this is exactly $\{c+\varphi([x]_G): c\in\CC, \varphi\in\Hom(G_{\mathrm{ab}},\CC)\}$. For such $f$, the gradient is constant in $x$:
$$
\partial^{s} f(x)=f(xs)-f(x)=\varphi([xs]_G)-\varphi([x]_G)=\varphi([s]_G).
$$
Thus, $\Theta([f])([s]_G)=\partial^{s}f(e)=\varphi([s]_G)$, so $\Theta([f])=\varphi$. The norm identity
$$
\|\nabla_{S}f\|_\infty=\max_{s\in S}|\varphi([s]_G)|=\|\varphi\|_S
$$
holds by \Cref{cor:norm_ident_vnilp}. This gives a well-defined isometric bijection with the stated inverse.
\end{proof}

\begin{example}[The discrete Heisenberg group]\label{ex:heisenberg}
Consider $G = H_3(\ZZ) = \langle X, Y, Z \mid [X,Y]=Z, [X,Z]=[Y,Z]=e \rangle$. The Abelianization $G_{\mathrm{ab}} \cong \ZZ^2$ is generated by the images $[X]$ and $[Y]$.

Define a non-symmetric probability measure $\mu$ on $G$ by
$$
\mu(X) = \tfrac{1}{2},\quad \mu(X^{-2}) = \tfrac{1}{4},\quad \mu(Y) = \tfrac{1}{8},\quad \mu(Y^{-1}) = \tfrac{1}{8}.
$$
The support generates $G$, so $\mu$ is adapted. Its Abelian drift is
$$
\mathbf{m}_{\mathrm{ab}}(\mu)
= \big(\tfrac{1}{2}\cdot 1 + \tfrac{1}{4}\cdot (-2)\big)[X]
  + \big(\tfrac{1}{8}\cdot 1 + \tfrac{1}{8}\cdot (-1)\big)[Y]
= 0,
$$
so $\mu$ is Abelian-centered despite being non-symmetric.

By \Cref{cor:norm_ident_vnilp}, every $f\in\LHF(G,\mu)$ has the form $f(g)=c+\bar\phi([g])$. In coordinates $g=X^xY^yZ^z$, the class $[g]\in G_{\mathrm{ab}}$ depends only on $x$ and $y$. The two independent non-constant Lipschitz harmonic functions are the coordinate projections
$$
f_1(X^xY^yZ^z)=x,\qquad f_2(X^xY^yZ^z)=y.
$$
No $z$-coordinate appears, by the preceding classification: every Lipschitz harmonic function is an affine character on $G_{\mathrm{ab}}$, and $Z=[X,Y]$ maps to $0$ in $G_{\mathrm{ab}}$. Hence every such function has the form
$$
f(X^xY^yZ^z)=c+ax+by.
$$
For $S=\{X^{\pm1},Y^{\pm1}\}$, both have $\|\nabla_S f_i\|_\infty=1$.
\end{example}

\begin{lemma}[Centering across finite-index]\label{lem:centered-transfer}
Let $G$ be a finitely generated virtually nilpotent group, $N\trianglelefteq G$ be a finite-index nilpotent normal subgroup of $G$ and $Q:=G/N$. Choose a set of right coset representatives $\{e=x_1, x_2, \dots, x_{|Q|}\}$ for $N$ in $G$. Thus, $G$ is partitioned as the disjoint union of right cosets: 
$$
\displaystyle G = \bigsqcup_{i=1}^{|Q|} N x_i.
$$
Define $s:Q\to G$ by $s(Nx_i):=x_i$. Then $Q$ acts on $N_{\mathrm{ab}}\otimes_\ZZ\RR$ by conjugation via $\rho:Q\to\mathrm{GL}(N_{\mathrm{ab}}\otimes_\ZZ\RR)$:
$$
\rho(q)[n]_{N_{\mathrm{ab}}}=[s(q)ns(q)^{-1}]_{N_{\mathrm{ab}}}\quad (q\in Q).
$$
Let $\mu$ be an adapted probability measure on $G$ with finite first moment and $\mu_N$ be the induced hitting measure on $N$. If $\mu$ is Abelian-centered, that is $\mathbf m_{\mathrm{ab}}(\mu)=0$ in $G_{\mathrm{ab}}\otimes_\ZZ\RR$, then $\mu_N$ is Abelian-centered:
$$
\mathbf m_{\mathrm{ab}}(\mu_N)=0 \qquad\text{in } N_{\mathrm{ab}}\otimes_\ZZ\RR.
$$
\end{lemma}

\begin{proof}
By Maschke's theorem for finite groups $N_{\mathrm{ab}}\otimes_\ZZ\RR=(N_{\mathrm{ab}}\otimes_\ZZ\RR)^{Q}\oplus W$, where $(N_{\mathrm{ab}}\otimes_\ZZ\RR)^{Q}$ denotes the set of vectors in $(N_{\mathrm{ab}}\otimes_\ZZ\RR)$ left invariant under the action of $Q$ and $W$ is the $Q$-invariant complement of $(N_{\mathrm{ab}}\otimes_\ZZ\RR)^{Q}$. Set $\delta:=\mathbf m_{\mathrm{ab}}(\mu_N)=\EE_e\big[[X_\tau]_{N_{\mathrm{ab}}}\big]$, where $(X_t)$ is the right $\mu$-walk from $e$ and $\tau$ its first return time to $N$. Here, $\EE_e[\tau]<\infty$ since $[G:N]<\infty$. We show $\delta$ lies in $N_{\mathrm{ab}}\otimes_\ZZ\RR=(N_{\mathrm{ab}}\otimes_\ZZ\RR)^{Q}\cap W$, whence $\delta=0$.

\emph{$\delta\in W$.} Let $\phi\in\Hom(N_{\mathrm{ab}},\RR)$ be a $Q$-invariant functional on $N_{\mathrm{ab}}\otimes_\ZZ\RR$ and by a slight abuse of notation, we again denote by $\phi$ its unique representative in $\Hom(N,\RR)$. Then $\phi$ is the restriction of a homomorphism $\Phi:G\to\RR$. Indeed, choose a set of right coset representatives $\{e=x_1, x_2, \dots, x_{|Q|}\}$ for $N$ in $G$. Thus, $G$ is partitioned as the disjoint union of right cosets:
$$ 
G = \bigsqcup_{i=1}^{|Q|} N x_i. 
$$

For any element $g \in G$, right-multiplication by $g$ permutes the right cosets. This induces a permutation $\sigma_g$ on the set $\{1, \dots, |Q|\}$ such that $N x_i g = N x_{\sigma_g(i)}$. Consequently, there exists a unique element $n_i \in N$ such that
$$ 
x_i g = n_i x_{\sigma_g(i)} \implies n_i = x_i g x_{\sigma_g(i)}^{-1} \in N. 
$$
We define $F: G \to \mathbb{R}$ by 
$$ 
F(g) := \sum_{i=1}^{|Q|} \phi\left( x_i g x_{\sigma_g(i)}^{-1} \right) \quad (g\in G). 
$$
It is a routine exercise to check that $F$ is a group homomorphism and $F(n) = \sum_{i=1}^k \phi\left( x_i n x_i^{-1} \right)=|Q|\phi(n)$ for $n\in N$. We define $\Phi: G \to \mathbb{R}$ by 
$$ 
\Phi(g) := \frac{1}{|Q|} F(g) \quad (g\in G). 
$$
Then $\Phi$ is a homomorphism whose restriction to $N$ is $\phi$. As $\Phi$ is a homomorphism, $\Phi(X_t)$ has i.i.d. increments $\Phi(\xi_i)$ of mean $\sum_{g\in G}\Phi(g)\mu(g)=0$ (as $\mu$ is Abelian-centered), so it is a mean-zero martingale. Moreover, $|\Phi(X_{\tau\wedge t})|\le C\sum_{i\le\tau}|\xi_i|$ for some $C>0$ depending only on $\Phi$, so $(\Phi(X_{\tau\wedge t}))$ is integrable by Wald's identity (as $\EE_e[\tau]<\infty$). Optional stopping theorem gives $\phi(\delta)=\EE_e[\Phi(X_\tau)]=\Phi(e)=0$. Since $\phi$ is an arbitrary $Q$-invariant functional on $N_{\mathrm{ab}}\otimes_\ZZ\RR$, this forces $\delta\in W$.

\emph{$\delta\in(N_{\mathrm{ab}}\otimes_\ZZ\RR)^{Q}$.} The projected random walk $(\bar X_t)$ on $Q$ is an irreducible random walk on the finite group $Q$ driven by the projected measure $\bar\mu(Ng)=\sum_{n\in N}\mu(ng)$ on $Q$, so its stationary measure is the uniform measure $\pi_{\mathrm{unif}}$ on $Q$. Let $V_t(q)=\sum_{j=0}^{t-1}\mathbf{1}_{\{\bar X_j=q\}}$ and $v(q)=\EE_{e_Q}[V_T(q)]$ $(q\in Q)$. By standard theory of Markov chains, and uniqueness of stationary measure, we have $\EE_{e_Q}[V_T(q)]=\pi_{\mathrm{unif}}(q)/\pi_{\mathrm{unif}}(e_Q)=1$ for all $q\in Q$. Fix a set-section $s:Q\to G$ with $s(e_Q)=e$, and define $\langle g\rangle:=[gs(\bar g)^{-1}]_{N_{\mathrm{ab}}}\in N_{\mathrm{ab}}\otimes_\ZZ\RR$ for $g\in G$, where $[\cdot]_{N_{\mathrm{ab}}}$ is the projection $N\to N_{\mathrm{ab}}\otimes_\ZZ\RR$ and $\bar g$ denotes the image of $g$ in $Q$. Note $\langle g\rangle=[g]_{N_{\mathrm{ab}}}$ when $g\in N$, so $\delta=\EE_e\langle X_\tau\rangle$ as $X_\tau\in N$. Writing $n_t:=X_ts(\bar X_t)^{-1}\in N$ and using additivity of $[\cdot]_{N_{\mathrm{ab}}}$ on $N$,
$$
\langle X_{t+1}\rangle-\langle X_t\rangle
=\big[s(\bar X_t)\xi_{t+1}s(\bar X_{t+1})^{-1}\big]_{N_{\mathrm{ab}}},
$$
the element $s(\bar X_t)\xi_{t+1}s(\bar X_{t+1})^{-1}$ lying in $N$ because its image in $Q$ is $\bar X_t\bar\xi_{t+1}(\bar X_t\bar\xi_{t+1})^{-1}=e_Q$. Define
$$
\sigma(a,b):=s(a)s(b)s(ab)^{-1}\in N\qquad(a,b\in Q),
$$
which lies in $N$ since $\overline{\sigma(a,b)}=ab(ab)^{-1}=e_Q$. Using additivity together with $\rho(q)[m]_{N_{\mathrm{ab}}}=[s(q)ms(q)^{-1}]_{N_{\mathrm{ab}}}$ for $m\in N$, the increment splits as
$$
\langle X_{t+1}\rangle-\langle X_t\rangle
=\rho(\bar X_t)\big[\xi_{t+1}s(\bar\xi_{t+1})^{-1}\big]_{N_{\mathrm{ab}}}
+\big[\sigma(\bar X_t,\bar\xi_{t+1})\big]_{N_{\mathrm{ab}}}.
$$
For $q\in Q$, define $F(q):=\rho(q)\bar w+\EE_\mu\big[[\sigma(q,\bar\xi)]_{Nab}\big]$, where
$$
\bar w:=\EE_\mu\big[[\xi s(\bar\xi)^{-1}]_{N_{\mathrm{ab}}}\big].
$$
Then, $\delta=\EE_e\left[ \sum_{t=0}^{\tau-1}F(\bar X_t)\right]=\EE_e\left[ \sum_{q\in Q}V_T(q)F(q)\right]=\sum_{q\in Q}F(q)E_{e}[V_T(q)]=\sum_{q\in Q}F(q)$. Therefore,
$$
\delta=\sum_{q\in Q}\rho(q)\bar w+\sum_{q\in Q}\EE_\mu\big[\sigma(q,\bar\xi)\big]_{N_{\mathrm{ab}}}
=|Q|P_{\mathrm{inv}}(\bar w)+\EE_\mu\big[T(\bar\xi)\big],
\qquad
\begin{aligned}
T(b)&:=\textstyle\sum_{a\in Q}[\sigma(a,b)]_{N_{\mathrm{ab}}},\\
P_{\mathrm{inv}}&:=|Q|^{-1}\textstyle\sum_{q}\rho(q).
\end{aligned}
$$
Note that the first term lies in $(N_{\mathrm{ab}}\otimes_\ZZ\RR)^{Q}$. For the second, Abelianize the identity $s(a)\sigma(b,c)s(a)^{-1}\sigma(a,bc)=\sigma(a,b)\sigma(ab,c)$ to
$$
\rho(a)[\sigma(b,c)]_{N_{\mathrm{ab}}}+[\sigma(a,bc)]_{N_{\mathrm{ab}}}
=[\sigma(a,b)]_{N_{\mathrm{ab}}}+[\sigma(ab,c)]_{N_{\mathrm{ab}}},
$$
and sum over $a\in Q$. Reindexing $a\mapsto ab$ in the last term gives
$$
|Q|P_{\mathrm{inv}}[\sigma(b,c)]_{N_{\mathrm{ab}}}+T(bc)=T(b)+T(c).
$$
Composing with the projection $P_W$ onto the complement $W$ annihilates the $P_{\mathrm{inv}}$-term, so $P_W\circ T:Q\to W$ is a homomorphism from the finite group $Q$ into the vector space $W$, hence identically $0$. Thus $T(b)\in(N_{\mathrm{ab}}\otimes_\ZZ\RR)^{Q}$ for every $b$, and so does $\EE_\mu[T(\bar\xi)]$. Therefore $\delta\in(N_{\mathrm{ab}}\otimes_\ZZ\RR)^{Q}$.

Combining the above, $\delta\in(N_{\mathrm{ab}}\otimes_\ZZ\RR)^{Q}\cap W=0$.
\end{proof}

\begin{theorem}[Dimension of $\LHF$]\label{thm:dim-LHF}
Let $G$ be a finitely generated virtually nilpotent group, and let $\mu$ be an adapted and smooth probability measure on $G$. Let $R$ denote the \emph{virtual first Betti number} of $G$, i.e. the common value
$$
R := \dim_\RR\big(N_{\mathrm{ab}}\otimes_\ZZ\RR\big)
$$
for any finite-index nilpotent subgroup $N\le G$. Let $N$ be any such subgroup and let $\mu_N$ be the induced hitting measure on $N$. Define
$$
\delta(\mu_N) :=
\begin{cases}
0 &\text{if } \mathbf{m}_{\mathrm{ab}}(\mu_N)=0,\\
1 &\text{if } \mathbf{m}_{\mathrm{ab}}(\mu_N)\neq 0.
\end{cases}
$$
Then
$$
\dim_\CC \LHF(G,\mu) = R + 1 - \delta(\mu_N).
$$
When $\mu$ is Abelian-centered, $\delta(\mu_N)=0$ and we have $\dim_\CC \LHF(G,\mu) = R + 1$.
\end{theorem}

\begin{proof}
By \Cref{bij:LHF}, $\LHF(G,\mu)\cong\LHF(N,\mu_N)$. 
By \Cref{thm:HF_1-is-linear-poly},
$
\LHF(N,\mu_N)=
\big\{x\mapsto c+\bar\phi([x]_N):\bar\phi\big(\mathbf m_{\mathrm{ab}}(\mu_N)\big)=0\big\}.
$
Now $\Hom(N_{\mathrm{ab}},\CC)$ has dimension $R$ over $\CC$. Write $v_N:=\mathbf{m}_{\mathrm{ab}}(\mu_N)\in N_{\mathrm{ab}}\otimes_\ZZ\RR$. If $v_N=0$, $\LHF(N,\mu_N)=\CC\oplus\Hom(N_{\mathrm{ab}},\CC)$ so $\dim_\CC\LHF(G,\mu)=R+1$. If $v_N\ne 0$, the single linear constraint $\bar\phi(v_N)=0$ lowers the dimension by one, giving $\dim_\CC\LHF(G,\mu)=R$. This proves $\dim_\CC\LHF(G,\mu)=R+1-\delta(\mu_N)$. When $\mu$ is Abelian-centered, apply \Cref{lem:centered-transfer} to $N_0:=\bigcap_{g\in G}gNg^{-1}$, a finite-index normal nilpotent subgroup. This gives $\mathbf m_{\mathrm{ab}}(\mu_{N_0})=0$, so the formula applied to $N_0$ yields $\dim_\CC\LHF(G,\mu)=R+1$. Since the same formula holds for the chosen $N$, this forces $\delta(\mu_N)=0$. The dimension drop $\delta(\mu_N)=1$ therefore occurs only for measures that are \emph{not} Abelian-centered on $G$.
\end{proof}

\begin{remark}\label{rem:MPTY_relation}
For general groups of polynomial growth and symmetric, adapted and smooth measures, Meyerovitch-Perl-Tointon-Yadin \cite{MPTY17} describe $\HF_k(G,\mu)$ in terms of harmonic polynomials on a finite-index nilpotent subgroup. In particular, for $k=1$ they demonstrate the maximal dimension $R+1$ in the symmetric case. \Cref{thm:dim-LHF} refines this picture by identifying the dimension drop at degree $1$ with the \emph{induced drift} $\mathbf{m}_{\mathrm{ab}}(\mu_N)$ on the finite-index nilpotent subgroup.
\end{remark}

\begin{lemma}[Interior centering from adapted measures]\label{lem:finite-recentering-set}
Let $G$ be a finitely generated group, $\mu$ an adapted and Abelian-centered probability measure on $G$, and $V:=G_{\mathrm{ab}}\otimes_\ZZ\RR$. Assume $\dim V\ge 1$. Let $\Sigma:=\{[g]:g\in\supp(\mu)\}\subset V$ and $\Sigma^\times:=\Sigma\setminus\{0\}$. Then $0\in\operatorname{int}(\operatorname{conv}(\Sigma^\times))$. Consequently, there exists a finite subset $T_0\subset\supp(\mu)\setminus\{e\}$ such that $0\in\operatorname{int}(\operatorname{conv}\{[t]:t\in T_0\})$.
\end{lemma}

\begin{proof}
Write $\Lambda:=G_{\mathrm{ab}}$. Because $\supp(\mu)$ semigroup-generates $G$, the set $\Sigma$ semigroup-generates $\Lambda$. Removing $0$ does not change the generated semigroup, so $\Sigma^\times$ also semigroup-generates $\Lambda$.

If $0\notin\operatorname{int}(\operatorname{conv}(\Sigma^\times))$, the Hahn-Banach separation theorem gives a nonzero real linear functional $\ell: V\to\RR$ with $\ell(\sigma)\ge 0$ for all $\sigma\in\Sigma^\times$. Hence $\ell\ge 0$ on every finite semigroup sum of elements of $\Sigma^\times$, and therefore on all of $\Lambda$. Since also $-\lambda\in\Lambda$ whenever $\lambda\in\Lambda$, this forces $\ell(\lambda)=0$ for every $\lambda\in\Lambda$, hence $\ell=0$ on $V$, a contradiction. Thus $0\in\operatorname{int}(\operatorname{conv}(\Sigma^\times))$.

To extract a finite subset, let $\mathbb S(V)$ be the unit sphere in $V$. Since $V$ is linearly isomorphic to $\RR^d$ for some $d\in\NN$, we transport the inner product $\langle \cdot,\cdot\rangle$ of $\RR^d$ to $V$. For each $\sigma\in\Sigma^\times$, define the open set $U_\sigma:=\{u\in\mathbb S(V):\langle u,\sigma\rangle>0\}$. Since $0\in\operatorname{int}(\operatorname{conv}(\Sigma^\times))$, $(U_\sigma)_{\sigma\in\Sigma^\times}$ cover $\mathbb S(V)$. By compactness, choose $\sigma_1,\dots,\sigma_m\in\Sigma^\times$ whose hemispheres still cover $\mathbb S(V)$. Then $0\in\operatorname{int}(\operatorname{conv}\{\sigma_1,\dots,\sigma_m\})$. Choose $t_i\in\supp(\mu)\setminus\{e\}$ with $[t_i]=\sigma_i$, and set $T_0:=\{t_1,\dots,t_m\}$.
\end{proof}

\begin{lemma}[Affine correction coefficients]\label{lem:small-affine-correction}
Let $T\subset V$ be finite with $0\in\operatorname{int}(\operatorname{conv}(T))$. Fix a norm $\| \cdot\|$ on $V$. Then there exists $C_0>0$ such that for every $v\in V$ there exist real coefficients $(c_t(v))_{t\in T}$ satisfying
$$
\sum_{t\in T}c_t(v)t=v,\qquad
\sum_{t\in T}c_t(v)=0,\qquad
\max_{t\in T}|c_t(v)|\le C_0\|v\|.
$$
\end{lemma}

\begin{proof}
Let $E:=\{a=(a_t)_{t\in T}\in\RR^T:\sum_{t\in T}a_t=0\}$ and define $L:E\to V$ by $L(a):=\sum_{t\in T}a_tt$. We show $L$ is surjective. Fix $v\in V\setminus\{0\}$. Since $0\in\operatorname{int}(\operatorname{conv}(T))$, there exists $\eps>0$ such that $\pm\eps v\in\operatorname{conv}(T)$. Hence there exist coefficients $a_t^+, a_t^-\ge 0$ with $\sum_t a_t^+=1=\sum_t a_t^-$ and $\eps v=\sum_{t\in T}a_t^+t$, $-\eps v=\sum_{t\in T}a_t^-t$. Set $c_t:=(a_t^+-a_t^-)/(2\eps)$. Then $c=(c_t)\in E$ and $L(c)=v$.

Choose a linear right inverse $R:V\to E$ of $L$ and set $c_t(v):=(Rv)_t$. Since $R$ is linear on the finite-dimensional space $V$, there exists $C_0>0$ with $\max_{t\in T}|c_t(v)|=\|Rv\|_{\ell^\infty(T)}\le C_0\|v\|$ for all $v\in V$.
\end{proof}

We now state the uniform Dungey estimates whose proof is given at the end of \Cref{sec:uniform-dungey}.

\begin{proposition}[Uniform diagonal Dungey estimates]\label{lem:uniform-dungey-family}
Let $G$ be a finitely generated nilpotent group and let $W\subseteq G$ be a fixed finite symmetric generating set containing $e$. Let $(K_R)_{R\ge R_0}$ be a family of centered, finitely supported probability measures on $G$, and assume there exist constants $\eta>0$, $\beta>0$, and $M_\beta>0$ such that
$$
K_R(e)\ge\eta,\qquad K_R(w)\ge \eta\quad (w\in W, R\ge R_0),
$$
and
$$
\sup_{R\ge R_0}\sum_{g\in G} e^{\beta |g|}K_R(g)\le M_\beta.
$$
Fix $\kappa>0$ and set $R_q:=\max\{R_0,\kappa\log(q+1)\}$ for $q\in\NN$. Then there exists a constant $C>0$, depending only on $(G,W,\eta,\beta,M_\beta,\kappa)$, such that for every $h\in W$ and every $q\in\NN$,

\begin{equation}\label{eq:app-diag-generator-bounds}
\|\partial_h K_{R_q}^{(q)}\|_1\le C q^{-1/2},
\qquad
\sum_{g\in G}|g|_W|\partial_h K_{R_q}^{(q)}(g)|\le C.
\end{equation}
where $\partial_h\nu(y):=\nu(h^{-1}y)-\nu(y)$.
\end{proposition}

\begin{theorem}[$\HF_1$ classification for smooth centered measures on nilpotent groups]\label{thm:unconditional-HF1-closure}
Let $G$ be a finitely generated nilpotent group, and let $\mu$ be an adapted, smooth, and Abelian-centered probability measure on $G$ (no symmetry or finite support assumed). Then
$$
\HF_1(G,\mu)=\LHF(G,\mu)=P^1(G).
$$
In particular, every $\mu$-harmonic function of at most linear growth is a globally Lipschitz affine character.
\end{theorem}

\begin{proof}
By \Cref{thm:HF_1-is-linear-poly}, $P^1(G)=\LHF(G,\mu)\subseteq\HF_1(G,\mu)$. It remains to prove $\HF_1(G,\mu)\subseteq\LHF(G,\mu)$. For $z\in G$, define 
$$
D_z f(x):=f(zx)-f(x) \qquad(x\in G)
$$

Suppose we have already proved that $D_zf$ is bounded for each $z\in G_s$. We now argue by induction on the nilpotency class $s$ of $G$. The base case $s=1$ is when $G$ is Abelian, and $D_c f(x):=f(cx)-f(x)$ is bounded and $\mu$-harmonic for every $c\in G$. Then the Choquet-Deny property makes each $D_c f$ constant, which implies $f\in\LHF(G,\mu)$.

Assume now that $s\ge2$ and that the theorem is true for all nilpotent groups of nilpotency class $<s$. Then, $G_s\subseteq Z(G)\cap [G,G]$. Since $D_z f$ is bounded and $\mu$-harmonic for every $z\in G_s$, \cite[Theorem 1]{FHTV19} gives $D_z f\equiv c_z$ for some $c_z\in\CC$. By \Cref{lem:deep-power-distortion}, $|z^n|=O(n^{1/s})$ as $n\to\infty$ for every $z\in G_s$. If $c_z\ne0$, then $|f(z^n x)-f(x)|=n|c_z|$, contradicting the linear-growth bound $|f(w)|\le A+B|w|$ together with $|z^n x|\le |x|+O(n^{1/s})$. Thus $c_z=0$ for every $z\in G_s$, and $f$ descends to a function $\tilde f:G/G_s\to\CC$.

Set $\bar G:=G/G_s$, let $\pi:G\to\bar G$ be the quotient map, and put $\bar\mu:=\pi_*\mu$. The measure $\bar\mu$ is adapted because the image of a generating set for $G$ under $\pi$ generates $\bar G$. It is smooth because, for quotient word metrics, $|\pi(g)|_{\bar G}\le |g|_G$, so any exponential moment of $\mu$ pushes forward to an exponential moment of $\bar\mu$. It is Abelian-centered because $G_s\subseteq [G,G]$, hence
$$
\bar G_{\mathrm{ab}}=(G/G_s)_{\mathrm{ab}}\cong G/[G,G]=G_{\mathrm{ab}},
$$
and the Abelian drift of $\bar\mu$ is the image of the Abelian drift of $\mu$, which is zero. Moreover $\tilde f$ is $\bar\mu$-harmonic, since $f=\tilde f\circ\pi$ and $f$ is $\mu$-harmonic. Finally, $\tilde f$ has at most linear growth on $\bar G$: every $\bar x\in\bar G$ has a lift $x\in G$ with $\pi(x)=\bar x$ and $|x|_G\le |\bar x|_{\bar G}$, and hence $|\tilde f(\bar x)|=|f(x)|\le A+B|\bar x|_{\bar G}$. Therefore $\tilde f\in\HF_1(\bar G,\bar\mu)$. By the induction hypothesis, $\tilde f\in P^1(\bar G)$, and consequently $f\in P^1(G)$.

It therefore suffices to establish the boundedness of $D_z f$ for each $z\in G_s$.

\smallskip\noindent Define the lazy measure $L:=\tfrac12\delta_e+\tfrac12\mu$. Since $f$ is $\mu$-harmonic, it is also $L$-harmonic, and hence $L^{(m)}$-harmonic for every $m\in\NN$. Fix $m_0\in\NN$ (whose value is to be determined later), and set $K:=L^{(m_0)}$. Since $f$ is $K$-harmonic, $z$ is central, and $|f(w)|\le A+B|w|$, the same computation as before gives, for every $q\in\NN$,
\begin{equation}\label{eq:Dz-kernel-rep}
D_z f(x)=\sum_{u\in G}f(xu)\Delta_q^z(K)(u),
\end{equation}
where $\Delta_q^z(K)(y):=K^{(q)}(yz^{-1})-K^{(q)}(y)$, and hence
\begin{equation}\label{eq:Dz-bound}
|D_z f(x)|
\le(A+B|x|)\|\Delta_q^z(K)\|_1+B\sum_u|u||\Delta_q^z(K)(u)|.
\end{equation}
Thus $D_z f$ is bounded provided:
\begin{align}
\|\Delta_q^z(K)\|_1&\to0\qquad(q\to\infty),\label{eq:tv-goal}\\
\sup_{q\in\NN}\sum_u|u||\Delta_q^z(K)(u)|<\infty,\label{eq:mom-goal}
\end{align}
choosing $q=q(x)$ so that $(A+B|x|)\|\Delta_q^z(K)\|_1\le1$. We establish \eqref{eq:tv-goal}-\eqref{eq:mom-goal} in four steps.

\smallskip\noindent If $V:=G_{\mathrm{ab}}\otimes_\ZZ\RR=0$ then $G$ is finite and $\LHF(G,\mu)=\HF_1(G,\mu)=\CC$, so the claim is trivial; assume henceforth $\dim V\ge1$. Since $\mu$ is adapted and Abelian-centered, \Cref{lem:finite-recentering-set} gives a finite subset $T_0\subset\supp(\mu)\setminus\{e\}$ with $0\in\operatorname{int}(\operatorname{conv}\{[t]:t\in T_0\})$. Next choose a finite subset $T_1\subset\supp(\mu)\setminus\{e\}$ that semigroup-generates $G$: for each element $w$ of a fixed finite symmetric generating set $W$ of $G$, choose a word $w=t_{w,1}\cdots t_{w,r_w}$ with $t_{w,j}\in\supp(\mu)$, and take $T_1$ to be the set of non-identity letters appearing. Set $T:=T_0\cup T_1\subset\supp(\mu)\setminus\{e\}$. Then $T$ semigroup-generates $G$, and $0\in\operatorname{int}(\operatorname{conv}\{[t]:t\in T\})$.

For $R>0$, let $B_R=\{x\in G\mid |x|\le R \}$ and $\mu_R:=\mu\mathbf{1}_{B_R}$. Since $\mu$ is smooth, there exists $\zeta >0$ such that $C_1:=\sum_{x\in G}\mu(x)e^{\zeta|x|}<\infty$. Now,
$$
\delta_R=\sum_{|x|>R}\mu(x)
\le e^{-\zeta R}\sum_{x\in G}e^{\zeta |x|}\mu(x)
= C_1 e^{-\zeta R}.
$$
Moreover, 
$$
\sum_{|x|>R}|x|\mu(x)
\le e^{-\frac{\zeta R}{2}}\sum_{|x|>R}|x|e^{\frac{\zeta |x|}{2}}\mu(x)
\le \frac{2}{\zeta}e^{-\frac{\zeta R}{2}}\sum_{|x|>R }e^{\zeta|x|}\mu(x)
\le \frac{2}{\zeta}C_1 e^{-\frac{\zeta R}{2}}.
$$
Hence the Abelian drift tail satisfies $\|m_R\|\le\sum_{|x|>R}|x|\mu(x)\le C_2 e^{-\frac{\zeta R}{2}}$, where $m_R:=\mathbf{m}_{\mathrm{ab}}(\mu_R)=-\sum_{|x|>R}\mu(x)[x]$ and $C_2>0$ depends on $C_1$, $\zeta$ and the generating set of $G$. By \Cref{lem:small-affine-correction}, there is $C_0>0$ such that for every $R>0$ there exist coefficients $(c_t(R))_{t\in T}$ satisfying
$$
\sum_{t\in T}c_t(R)[t]=-m_R,\qquad
\sum_{t\in T}c_t(R)=0,\qquad
\max_{t\in T}|c_t(R)|\le C_0\|m_R\|\le C_0C_2 e^{-\frac{\zeta R}{2}}.
$$
Choose $R_0$ so large that $T\subset B_{R_0}$ and $C_0C_2 e^{-\frac{\zeta R_0}{2}}<\frac{1}{2}\min_{t\in T}\mu(t)$. For $R\ge R_0$, define
$$
\widetilde\mu_R(x)
:=\mu_R(x)+\sum_{t\in T}c_t(R)\delta_t(x)+\delta_R\delta_e(x).
$$
Then $\widetilde\mu_R$ has total mass $1$ and $\mathbf{m}_{\mathrm{ab}}(\widetilde\mu_R)=0$. Since $T\subset B_{R_0}\subset B_R$, we have $\mu_R(t)=\mu(t)$ for each $t\in T$, and therefore $\widetilde\mu_R(t)=\mu(t)+c_t\ge\frac12\mu(t)>0$. Since $e\notin T$, also $\widetilde\mu_R(e)=\mu_R(e)+\delta_R\ge 0$ and $\widetilde\mu_R(x)=0$ for all $x\notin B_R$. Hence $\widetilde\mu_R$ is an adapted, finitely supported, Abelian-centered probability measure. For the total variation distance, note that $\mu-\widetilde\mu_R$ is supported on $\{|x|>R\}\cup T\cup\{e\}$, so
$$
\|\mu-\widetilde\mu_R\|_1
\le \sum_{|x|>R}\mu(x)+\sum_{t\in T}|c_t|+\delta_R
\le \delta_R+|T|C_0C_2 e^{-\frac{\zeta R}{2}}+\delta_R
\le C_3e^{-\frac{\zeta R}{2}}
$$
for some $C_3>0$ depending on $C_0, C_1, C_2$ and $T$. Set $\rho_R:=\frac12\|\mu-\widetilde\mu_R\|_1\le \frac12 C_3e^{-\frac{\zeta R}{2}}$.

\smallskip\noindent Define $L_R:=\frac12\delta_e+\frac12\widetilde\mu_R$ and $K_R:=L_R^{(m_0)}$. Because $T$ semigroup-generates $G$, we can now choose $m_0$ such that each $w\in W$ is a word of length at most $m_0$ in the elements of $T$. Since $L_R(e)\ge\tfrac12$ and $L_R(t)\ge\mu(t)/4$ for every $t\in T$ and every $R\ge R_0$, there exists $\eps>0$ such that $K_R(w)\ge\eps$ for all $w\in W$ and all $R\ge R_0$.

Fix $\beta\in(0,\zeta/2)$. Using $e\notin T$ and the bound on the coefficients $c_t$,
\begin{align*}
\sum_{g\in G}e^{\beta|g|}\widetilde\mu_R(g)
&\le\sum_{|g|\le R}e^{\beta|g|}\mu(g)
  +\delta_R
  +\sum_{t\in T}|c_t|e^{\beta|t|}\\
&\le\sum_{g\in G}e^{\beta|g|}\mu(g)
  +C_1 e^{-\frac{\zeta R}{2}}
  +C_0C_2 e^{-\frac{\zeta R}{2}}\sum_{t\in T}e^{\beta|t|}
\le M_{\beta,T}
\end{align*}
for a constant $M_{\beta,T}>0$ independent of $R$. Hence
$$
\sum_{g\in G}e^{\beta|g|}K_R(g)
\le\Big(\sum_{g\in G}e^{\beta|g|}L_R(g)\Big)^{m_0}\le\big(\tfrac12+\tfrac12 M_{\beta,T}\big)^{m_0}
=:M<\infty,
$$
uniformly in $R\ge R_0$.

Fix a parameter $\kappa>0$ to be chosen later, and define 
$$
R_q:=\max\{R_0,\kappa\log(q+1)\}\qquad (q\in \NN).
$$
\Cref{lem:uniform-dungey-family} therefore gives, for each $h\in W$ and all $q\in\NN$,
\begin{equation}\label{eq:Dungey-gen-bound}
\|\partial_h K_{R_q}^{(q)}\|_1\le\frac{C_4}{\sqrt{q}},
\qquad
\sum_{y\in G}|y||\partial_h K_{R_q}^{(q)}(y)|\le C_4,
\end{equation}
with $C_4$ depending only on $(G,W,\eps,\beta,M,\kappa)$.

To pass from generator differences to $\Delta_q^z(K_{R_q})$ with $z\in G_s$, write $z=s_1\cdots s_\ell$ with $s_i\in W$ and $\ell=|z|_W$. Since $z\in Z(G)$, $\Delta_q^z(K_{R_q})(y)=K_{R_q}^{(q)}(z^{-1}y)-K_{R_q}^{(q)}(y)=\partial_z K_{R_q}^{(q)}(y)$. Telescoping gives
$$
\partial_{z}K_{R_q}^{(q)}
=\sum_{j=1}^{\ell}L_{s_1}\cdots L_{s_{j-1}}\big(\partial_{s_j}K_{R_q}^{(q)}\big),
$$
where $L_g\nu(y):=\nu(g^{-1}y)$. Each left-translation preserves the $\ell^1$-norm, and $\sum_y|y||L_g\nu(y)|=\sum_y|gy||\nu(y)|\le\sum_y|y||\nu(y)|+|g|\|\nu\|_1$. Hence
\begin{equation}\label{eq:uniform-Dungey}
\|\Delta_q^z(K_{R_q})\|_1
\le\frac{\ell C_4}{\sqrt{q}},
\qquad
\sum_{y}|y||\Delta_q^z(K_{R_q})(y)|
\le\ell C_4+\frac{\ell^2 C_4}{\sqrt{q}},
\end{equation}
for every $q\in\NN$.

\smallskip\noindent Define $\nu_q:=\widetilde\mu_{R_q}$. The total variation distance between $L$ and $L_{R_q}$ satisfies $\frac{1}{2}\|L-L_{R_q}\|_1=\frac{1}{2}\rho_{R_q}\le\frac{1}{4} C_3(q+1)^{-\frac{\zeta\kappa}{2}}$. Hence 
$$
\frac{1}{2}\|K-K_{R_q}\|_1\le \frac{m_0}{2}\| L-L_{R_q}\|\le\frac{m_0}{4}C_3(q+1)^{-\frac{\zeta\kappa}{2}}.
$$

For each step $i$, draw $(\xi_i,\eta_i)$ with marginals $\xi_i\sim K$, $\eta_i\sim K_{R_q}$, and $\PP(\xi_i\ne\eta_i)=\tfrac12\|K-K_{R_q}\|_1$. Let $X_q=\xi_1\cdots\xi_q$ and $Y_q=\eta_1\cdots\eta_q$. Set
$$
p_q(y):=K^{(q)}(y)=\PP(X_q=y),
\qquad
\widetilde p_q(y):=K_{R_q}^{(q)}(y)=\PP(Y_q=y).
$$
By the standard coupling inequality,
$$
\tfrac12\sum_y\big|\PP(X_q=y)-\PP(Y_q=y)\big|
\le\PP(X_q\ne Y_q)
\le q\cdot\tfrac{m_0}{2}C_3(q+1)^{-\frac{\zeta\kappa}{2}}
\le\tfrac{m_0 C_3}{2}q^{1-\frac{\zeta\kappa}{2}}.
$$
For a function $\nu:G\to\CC$ and $g\in G$, define $(\tau_g\nu)(y):=\nu(yg^{-1})$. Observe that
$$
\Delta_q^z(K)=\left[\tau_z \big(K^{(q)}-K^{(q)}_{R_q}\big)\right]+\Delta_q^z(K_{R_q})+\left[K^{(q)}_{R_q}-K^{(q)}\right].
$$
Hence, by the triangle inequality and the fact that $\tau_z$ preserves the $\ell^1$-norm,
\begin{align*}
\|\Delta_q^z(K)\|_1
&\le \|\tau_z(p_q-\widetilde p_q)\|_1
   +\|\Delta_q^z(K_{R_q})\|_1
   +\|\widetilde p_q-p_q\|_1 \\
&= \|\Delta_q^z(K_{R_q})\|_1
   +2\|p_q-\widetilde p_q\|_1 \\
&= \|\Delta_q^z(K_{R_q})\|_1
   +2\sum_y\big|\PP(X_q=y)-\PP(Y_q=y)\big|\\
&\le \|\Delta_q^z(K_{R_q})\|_1 + 2m_0C_3q^{1-\frac{\zeta\kappa}{2}}
\le \frac{\ell C_4}{\sqrt q}+2m_0C_3q^{1-\frac{\zeta\kappa}{2}}.
\end{align*}
Choosing $\kappa>2/\zeta$ makes the right-hand side tend to $0$ as $q\to\infty$, establishing \eqref{eq:tv-goal}.

\smallskip\noindent Now,
\begin{align*}
\sum_y|y||\Delta_q^z(K)(y)|
&\le \sum_y |y|\big|\Delta_q^z(K_{R_q})(y)\big|
 + \sum_y |y|\big|\tau_z(K^{(q)}-K^{(q)}_{R_q})(y)\big|
 + \sum_y |y|\big|K^{(q)}(y)-K^{(q)}_{R_q}(y)\big|.
\end{align*}
By change of variables and using $|yz|\le |y|+|z|$, we get
$$
\sum_y |y|\big|\tau_z(K^{(q)}-K^{(q)}_{R_q})(y)\big|
\le \sum_y |y|\big|K^{(q)}(y)-K^{(q)}_{R_q}(y)\big|
   +|z|\big\|K^{(q)}-K^{(q)}_{R_q}\big\|_1.
$$
Hence,
\begin{align*}
\sum_y|y||\Delta_q^z(K)(y)|
&\le \sum_y|y|\big|\Delta_q^z(K_{R_q})(y)\big|
  +2\sum_y|y|\big|K^{(q)}(y)-K^{(q)}_{R_q}(y)\big|
  +|z|\big\|K^{(q)}-K^{(q)}_{R_q}\big\|_1.
\end{align*}
Now for each $y$,
$$
\big|K^{(q)}(y)-K^{(q)}_{R_q}(y)\big|
\le \PP(X_q=y, X_q\ne Y_q)+\PP(Y_q=y, X_q\ne Y_q),
$$
so
$$
\sum_y |y|\big|K^{(q)}(y)-K^{(q)}_{R_q}(y)\big|
\le \EE[|X_q|\mathbf1_{\{X_q\ne Y_q\}}]
 + \EE[|Y_q|\mathbf1_{\{X_q\ne Y_q\}}].
$$
Also $\big\|K^{(q)}-K^{(q)}_{R_q}\big\|_1\le 2\PP(X_q\ne Y_q)$. Therefore
\begin{align*}
\sum_y|y||\Delta_q^z(K)(y)|
&\le \sum_y|y||\Delta_q^z(K_{R_q})(y)|
  +2\big(\EE[|X_q|\mathbf1_{\{X_q\ne Y_q\}}]+\EE[|Y_q|\mathbf1_{\{X_q\ne Y_q\}}]\big)\\
&\qquad +2|z|\PP(X_q\ne Y_q).
\end{align*}

By Cauchy--Schwarz, $\EE[|X_q|\mathbf1_{\{X_q\ne Y_q\}}]\le\EE[|X_q|^2]^{1/2}\PP(X_q\ne Y_q)^{1/2}$. Subadditivity gives $|X_q|\le\sum_{i=1}^q|\xi_i|$, so $\EE[|X_q|^2]\le q^2\EE[|\xi_1|^2]=q^2 M_2$, where $M_2=\sum_{x\in G}|x|^2K(x)<\infty$ (since $K$ has exponential tails). Hence, $\EE[|X_q|\mathbf1_{\{X_q\ne Y_q\}}]\le\sqrt{M_2m_0 C_3/2}q^{3/2-\zeta\kappa/4}$. The analogous bound holds for $Y_q$: since $|Y_q|\le\sum_{i=1}^q|\eta_i|$ with $\eta_i\sim K_{R_q}$, and $\EE[|\eta_1|^2]=\sum_g|g|^2 K_{R_q}(g)\le 2\beta^{-2}M=:M_2'$ uniformly in $q$, the same Cauchy--Schwarz estimate yields $\EE[|Y_q|\mathbf1_{\{X_q\ne Y_q\}}]\le\sqrt{M_2' m_0 C_3/2}q^{3/2-\zeta\kappa/4}$. Choosing $\kappa\ge 6/\zeta$ gives
$$
\sup_{q\in\NN}\sum_y|y||\Delta_q^z(K)(y)|<\infty,
$$
establishing \eqref{eq:mom-goal}.

\end{proof}

\begin{corollary}[$\HF_1=\LHF$ on polynomial-growth groups]\label{cor:HF_1-is-LHF}
Let $G$ be a finitely generated group of polynomial growth and let $\mu$ be adapted, smooth and Abelian-centered. Then
$$
\HF_1(G,\mu)=\LHF(G,\mu).
$$
\end{corollary}

\begin{proof}
Let $N\le G$ be a finite-index torsion-free nilpotent subgroup, with hitting law $\mu_N$, which is again adapted, smooth and Abelian-centered. \Cref{thm:unconditional-HF1-closure} therefore applies on $N$ and gives $\HF_1(N,\mu_N)=\LHF(N,\mu_N)$. Now, $\HF_1(G,\mu)\cong\HF_1(N,\mu_N)$ by \cite[Proposition 3.4]{MY16} and $\LHF(G,\mu)\cong\LHF(N,\mu_N)$ by \Cref{bij:LHF}, hence $\HF_1(G,\mu)\cong\LHF(G,\mu)$. Since $\dim \LHF(G,\mu)=\dim\LHF(N,\mu_N)<\infty$ by \Cref{thm:dim-LHF}, we have $\HF_1(G,\mu)=\LHF(G,\mu)$.
\end{proof}

\subsection{Virtual first cohomology and cohomological transport}

For virtually nilpotent groups, ordinary first cohomology $H^1(G;\RR)$ can be too small for our purposes. Indeed, if $N\trianglelefteq G$ is a finite-index nilpotent subgroup, then the finite quotient $G/N$ acts by conjugation on $H^1(N;\RR)$, and restriction identifies $H^1(G;\RR)$ with the invariant subspace $H^1(N;\RR)^{G/N}$. This invariant subspace can be proper. The appropriate canonical object is therefore the following virtual first cohomology.

\begin{lemma}[Finite-index subgroups and commutators]
\label{lem:finite-index-commutators}
Let $G$ be a finitely generated nilpotent group and let $H\leq G$ be a finite-index subgroup. Then
\begin{equation}\label{eq:finite-commutator-kernel}
(H\cap [G,G])/[H,H]
\end{equation}
is finite.
\end{lemma}
\begin{proof}
Since $[H,H]\leq H$ and $[H,H]\leq [G,G]$ (because $H\leq G$), we have $[H,H]\leq H\cap[G,G]\leq [G,G]$, and hence $[H\cap[G,G]:[H,H]]\leq [[G,G]:[H,H]]$. It therefore suffices to show that $[[G,G]:[H,H]]<\infty$.

If $G$ is Abelian (nilpotency class $c=1$), then $[G,G]=\{e\}$ and the conclusion is immediate. Assume henceforth that $c\geq2$.

We prove the auxiliary claim that $[\gamma_k(G):\gamma_k(H)]<\infty$ for every $k\geq2$, by descending induction on $k$. Here $\gamma_{k+1}(G)=[G,\gamma_k(G)]$, and $\gamma_{c+1}(G)=\{e\}$.

Choose generators $x_1,\dots,x_r$ for $G$. Since $H$ has finite index, there is an integer $m\geq1$ such that $x_i^m\in H$ for every $i$.

For each $k\geq2$, the Abelian group $\gamma_k(G)/\gamma_{k+1}(G)$ is generated by the images of the left-normed simple commutators $[x_{i_1},\dots,x_{i_k}]$ of weight $k$. Standard commutator calculus in the associated graded group gives
\begin{equation}\label{eq:power-commutator-graded}
[x_{i_1}^m,\dots,x_{i_k}^m]
\equiv
[x_{i_1},\dots,x_{i_k}]^{m^k}
\pmod{\gamma_{k+1}(G)}.
\end{equation}
The left-hand side lies in $\gamma_k(H)$. Hence the image of $\gamma_k(H)$ in $\gamma_k(G)/\gamma_{k+1}(G)$ contains the $m^k$-multiples of a finite set of generators, and therefore has finite index.

\smallskip\noindent\emph{Base case ($k=c$).} Since $\gamma_{c+1}(G)=\{e\}$, so $\gamma_c(G)/\gamma_{c+1}(G)=\gamma_c(G)$. The argument above shows directly that $\gamma_c(H)$ contains $m^c$-th powers of every generator of the finitely generated Abelian group $\gamma_c(G)$, so $[\gamma_c(G):\gamma_c(H)]<\infty$.

\smallskip\noindent\emph{Induction step ($k+1\to k$, for $2\leq k<c$).} Assume $[\gamma_{k+1}(G):\gamma_{k+1}(H)]<\infty$. The argument above gives
$$
[\gamma_k(G):\gamma_k(H)\gamma_{k+1}(G)]=[\gamma_k(G)/\gamma_{k+1}(G):\gamma_k(H)\gamma_{k+1}(G)/\gamma_{k+1}(G)]<\infty.
$$
For the remaining factor, the second isomorphism theorem (applicable because $\gamma_{k+1}(G)\trianglelefteq\gamma_k(G)$) gives
$$
\gamma_k(H)\gamma_{k+1}(G)\big/\gamma_k(H)\cong
\gamma_{k+1}(G)\big/\bigl(\gamma_k(H)\cap\gamma_{k+1}(G)\bigr).
$$
Since $\gamma_{k+1}(H)\leq\gamma_k(H)\cap\gamma_{k+1}(G)$, the right-hand side has order at most $[\gamma_{k+1}(G):\gamma_{k+1}(H)]<\infty$ by the induction hypothesis. Combining the two finite indices:
$$
[\gamma_k(G):\gamma_k(H)]
=[\gamma_k(G):\gamma_k(H)\gamma_{k+1}(G)]
[\gamma_k(H)\gamma_{k+1}(G):\gamma_k(H)]<\infty.
$$

Taking $k=2$ gives $[[G,G]:[H,H]]<\infty$, which completes the proof.
\end{proof}

\begin{lemma}[Restriction across finite index]\label{lem:finite-index-H1-iso}
Let $N$ be a finitely generated nilpotent group and let $N'\le N$ be a finite-index subgroup. Then restriction induces an isomorphism
$$
\operatorname{res}_{N'}^N:H^1(N;\mathbb F)\longrightarrow H^1(N';\mathbb F),
\qquad \mathbb F\in\{\RR,\CC\}.
$$
Equivalently, every homomorphism $N'\to\mathbb F$ extends uniquely to a homomorphism $N\to\mathbb F$.
\end{lemma}

\begin{proof}
We first prove that the restriction
$$
\Hom(A,\mathbb F)\longrightarrow \Hom(A',\mathbb F)
$$
is an isomorphism for a finite-index subgroup $A'\le A$ of a finitely generated Abelian group $A$. Note that there exists a positive integer $k$ such that $kx\in A'$ for all $x\in A$. Injectivity is immediate: if $\phi$ vanishes on $A'$ and $b\in A$, choose $m\in \NN$ with $mb\in A'$; then $m\phi(b)=\phi(mb)=0$, hence $\phi(b)=0$. For surjectivity, given $\psi\in\Hom(A',\mathbb F)$ set
$$
\widetilde\psi(b):=\frac{1}{k}\psi(kb) \qquad (b\in A).
$$
The resulting map $\widetilde\psi:A\to\mathbb F$ is a homomorphism extending $\psi$.

Write $B=N_{\rm ab}=N/[N,N]$, and let $B'=N'/(N'\cap [N,N])$. Let $C=(N'\cap [N,N])/[N',N']$, which is a finite subgroup of $N'_{\mathrm{ab}}$ by \Cref{lem:finite-index-commutators}. Then, $B'\cong N'_{\mathrm{ab}}/C$ via the isomorphism $\lambda_1:B'\to N'_{\mathrm{ab}}/C$ given by
$$
\lambda_1(x N'\cap [N,N]):=x[N',N']  \pmod C \qquad (x\in N')
$$
Also, $B'\cong D$, where $D=\{x[N,N]\mid x\in N' \}$ is a finite-index subgroup of $B$ and the isomorphism $\lambda_2:D\to B'$ is given by
$$
\lambda_2(x[N,N]):=x(N'\cap [N,N]) \qquad (x\in N').
$$
For $f\in \Hom(N',\mathbb F)$, let $\widetilde f:N'_{\mathrm{ab}}\to\mathbb F$ be the homomorphism through which $f$ factors (as $\mathbb F$ is Abelian), and $F:N'_{\mathrm{ab}}/C\to \mathbb F$ the map it induces on the quotient (as $C$ is finite and $\mathbb F$ is torsion-free). Define $g:N\to\mathbb F$ by 
$$
g(x)=\tilde g(x[N,N]) \qquad (x\in N)
$$ 
where $\widetilde g$ is the extension of $F\circ\lambda_1\circ\lambda_2$ to $B$. Then, $g|_{N'}=f$. The fact that this extension is unique can be proved exactly as we proved the injectivity of the restriction map from $\Hom(A,\mathbb F)$ to $\Hom(A',\mathbb F)$ at the beginning of this proof. 
\end{proof}

\begin{definition}[Virtual first cohomology]\label{def:virtual-H1}
Let $G$ be a finitely generated group of polynomial growth. Let $\mathcal N(G)$ be the collection of finite-index torsion-free nilpotent subgroups of $G$. For $\mathbb F\in\{\RR,\CC\}$ set
$$
\mathrm{VH}^1(G;\mathbb F)
:=\varinjlim_{N\in\mathcal N(G)} H^1(N;\mathbb F),
$$
where the transition maps are restrictions to further finite-index subgroups. Equivalently, an element is represented by a pair $(N,\alpha)$ with $N\in\mathcal N(G)$ and $\alpha\in H^1(N;\mathbb F)$, and $(N,\alpha)\sim (N',\alpha')$ if the restrictions of $\alpha$ and $\alpha'$ to some finite-index subgroup of $N\cap N'$ agree. By \Cref{lem:finite-index-H1-iso}, the transition maps are isomorphisms; in particular this is a finite-dimensional vector space.
\end{definition}

\begin{remark}[Higher virtual cohomology]\label{rem:higher-VH}
The construction of \Cref{def:virtual-H1} makes sense in every degree: for $n\ge0$ and $\mathbb F\in\{\RR,\CC\}$ set $\mathrm{VH}^n(G;\mathbb F):=\varinjlim_{N\in\mathcal N(G)} H^n(N;\mathbb F)$, with restrictions as transition maps. However, we make no use of $\mathrm{VH}^n$ for $n\ge2$ in this paper, as the normed and harmonic structure developed here is special to degree one. 
\end{remark}

\begin{theorem}[Canonical virtual affine isomorphism]\label{thm:virtual-affine-boundary}
Let $G$ be a finitely generated group of polynomial growth and let $\mu$ be an $\SAS$ measure on $G$. There is a canonical isomorphism
$$
\mathcal A_{G,\mu}:\LHF(G,\mu)\longrightarrow \CC\oplus
\mathrm{VH}^1(G;\CC).
$$
More explicitly, choose $N\in\mathcal N(G)$ and let $\mu_N$ be the hitting law on $N$. If
$$
\mathrm{Res}_N^G f(n)=c+\varphi_N(n)\qquad(n\in N)
$$
is the decomposition given by \Cref{thm:HF_1-is-linear-poly} for $(N,\mu_N)$, then
$$
\mathcal A_{G,\mu}(f)=(c,[N,\varphi_N]).
$$
This definition is independent of the choice of $N$.
\end{theorem}

\begin{proof}
Since $\mu$ is $\SAS$, the hitting law $\mu_N$ is again $\SAS$. Hence \Cref{thm:HF_1-is-linear-poly} applies on $N$. If $L\le N$ is another finite-index torsion-free nilpotent subgroup, then the restriction of $\mathrm{Res}_N^G f$ to $L$ is $\mathrm{Res}_L^G f$, and $\varphi_N$ restricts to the complex-valued homomorphism on $L$ obtained from the decomposition of $\mathrm{Res}_L^G f$ . Passing to a common finite-index subgroup therefore shows independence of $N$.

Conversely, if $(c,[N,\varphi_N])\in \CC\oplus\mathrm{VH}^1(G;\CC)$, then $c+\varphi_N$ is $\mu_N$-harmonic, and \Cref{thm:C} induces it uniquely to a function in $\LHF(G,\mu)$. If the same equivalence class is represented on another subgroup, the two induced functions on $G$ have the same restriction to a common finite-index subgroup, hence agree by the injectivity of the restriction map in \Cref{thm:C}. This gives the inverse map and proves the claim.
\end{proof}

\section{Stable norms and virtual first cohomology}\label{sec:stable-norms}

We identify the Lipschitz seminorm on affine characters with the dual stable norm and show that its ambient version is well defined on virtual first cohomology.

Fix a finite symmetric generating set $S$ of $G$. Let $\mathrm{Lip}(G)$ be the space of real/complex-valued Lipschitz functions on $G$, endowed with the seminorm
$$
\|f\|_{\mathrm{Lip},S} := \|\nabla_S f\|_\infty = \max_{s\in S} \sup_{x\in G} |\partial^{s}f(x)|,
\qquad
\partial^{s}f(x)=f(x s)-f(x).
$$
For $\phi\in\Hom(G_{\mathrm{ab}},\CC)$, set
$$
\|\phi\|_{S}:=\max_{s\in S}|\phi([s]_G)|.
$$
We first identify this generator norm with the dual stable norm of the word metric.

\begin{definition}[Stable norm of the word metric]\label{def:stable-norm}
Let $G$ be a finitely generated nilpotent group with finite symmetric generating set $S$, word length $|\cdot|_S$, and canonical projection $\pi:G\to G_{\mathrm{ab}}\otimes_\ZZ\RR$. For $a\in\pi(G)\subset G_{\mathrm{ab}}\otimes_\ZZ\RR$ set
$$
\ell_S(a):=\lim_{k\to\infty}\frac1k
\inf\bigl\{|n|_S : n\in G, \pi(n)=ka\bigr\},
$$
the limit existing by Fekete's subadditive lemma. Extended by homogeneity and continuity, $\ell_S$ is a norm on $G_{\mathrm{ab}}\otimes_\ZZ\RR$, the \emph{stable norm} of the Abelianized word metric. We write $\ell_S^\ast$ for its dual norm on $\Hom(G_{\mathrm{ab}},\RR)$, and on $\Hom(G_{\mathrm{ab}},\CC)$ after complexification. From the definition $\ell_S(\pi(n))\le|n|_S$, i.e. $\pi$ is $1$-Lipschitz for $\ell_S$.
\end{definition}

\begin{remark}\label{rem:stable-norm-provenance}
The stable norm on real homology goes back to Federer \cite{Federer74}; see also Gromov \cite[Chapter 4]{Gromov2007} and \cite{Burago92,BuragoIvanovKleiner1997}. The feature used here is its exact finite-index compatibility, proved in \Cref{lem:ambient-stable-fi-isometry}.
\end{remark}

\begin{lemma}[The Lipschitz seminorm is the dual stable norm]\label{lem:lip-is-dual-stable}
Let $G$ be finitely generated nilpotent with finite symmetric generating set $S$. The generator norm $\|\cdot\|_S$ is the dual stable norm:
$$
\|\phi\|_S=\ell_S^\ast(\phi)\qquad(\phi\in\Hom(G_{\mathrm{ab}},\CC)).
$$
Consequently, for adapted Abelian-centered $\mu$ with a finite first moment and $f=c+\varphi$ as in \Cref{cor:norm_ident_vnilp}, one has $\|\nabla_S f\|_\infty=\ell_S^\ast(\varphi)$.
\end{lemma}

\begin{proof}
Let $P:=\operatorname{conv}\{\pi(s):s\in S\}$, a symmetric convex set since $S$ is symmetric, and let $w$ be its Minkowski functional. We show $\ell_S=w$.

\emph{$w\le\ell_S$.} Let $a\in\pi(G)$. A geodesic word $n=s_{i_1}\cdots s_{i_m}$ with $m=|n|_S$ projects to $\pi(n)=m\sum_j\frac{1}{m}\pi(s_{i_j})\in mP$, so $w(\pi(n))\le m=|n|_S$. Hence for any $n$ with $\pi(n)=ka$ we have $w(ka)\le|n|_S$; dividing by $k$ and taking the infimum and limit gives $w(a)\le\ell_S(a)$. By homogenity and continuity this extends to all $a\in G_{\mathrm{ab}}\otimes_\ZZ \RR$.

\emph{$\ell_S\le w$.} By homogeneity it suffices to treat $a\in G_{\mathrm{ab}}\otimes_\ZZ \RR$ with $w(a)=1$, i.e. $a\in\partial P$. We first treat $a=\sum_i\lambda_i\pi(s_i)$ with $\lambda_i\in\mathbb{Q}_{\ge0}$, $\sum_i\lambda_i=1$. Let $k$ be the least common (positive) multiple of the denominators of $\lambda_i$ and $n_k:=\prod_i s_i^{k\lambda_i}$. Then $\pi(n_k)=ka$ and $|n_k|_S\le k$, so $\ell_S(a)\le1=w(a)$. As the points of $\partial P$ with rational convex combination coefficients are dense and $\ell_S,w$ are continuous, $\ell_S\le w$ on all of $\partial P$, hence everywhere by homogeneity.

Therefore $\ell_S=w$, and dualizing the polytope $P$,
$$
\ell_S^\ast(\phi)=\sup_{v\in P}|\phi(v)|=\max_{s\in S}|\phi([s])|=\|\phi\|_S.
$$
The final identity is \Cref{cor:norm_ident_vnilp}.
\end{proof}

The stable norm above is attached to a single nilpotent group through its own generators. To norm the \emph{virtual} object $\mathrm{VH}^1(G;\CC)$ of a polynomial-growth group, we use the ambient word metric of $G$ itself, which makes the norm canonical rather than subgroup-dependent.

\begin{definition}[Ambient stable norm on virtual first cohomology]\label{def:ambient-stable-vH1}
Let $G$ be a finitely generated group of polynomial growth with word metric $d_G$, and let $N\le G$ be a finite-index torsion-free nilpotent subgroup with projection $\pi_N:N\to N_{\mathrm{ab}}\otimes_\ZZ\RR$. For $a\in\pi_N(N)$ set
$$
\ell_{G,N}(a):=\lim_{k\to\infty}\frac1k
\inf\bigl\{d_G(e,n) : n\in N, \pi_N(n)=ka\bigr\},
$$
extended by homogeneity and continuity to a norm on $N_{\mathrm{ab}}\otimes_\ZZ\RR$, the \emph{ambient stable norm} induced by $d_G$. Its dual on $H^1(N;\CC)=\Hom(N_{\mathrm{ab}},\CC)$ is written $\ell_{G,N}^\ast$.
\end{definition}

The fact that $\ell_{G,N}$ is a genuine norm, not merely a seminorm, follows because $N$ has finite index in $G$, so $d_G|_N$ is bi-Lipschitz to the word metric of $N$, and the Abelianization $N\to N_{\mathrm{ab}}\otimes_\ZZ\RR$ is Lipschitz. Thus $d_G(e,n)\ge c\|\pi_N(n)\|$ for some $c>0$, giving $\ell_{G,N}(a)\ge c|a|$.

\begin{lemma}[Finite-index compatibility of the ambient stable norm]\label{lem:ambient-stable-fi-isometry}
Let $L\le N\le G$ be finite-index torsion-free nilpotent subgroups of a polynomial-growth group $G$. Under the canonical identification $N_{\mathrm{ab}}\otimes_\ZZ\RR\cong L_{\mathrm{ab}}\otimes_\ZZ\RR$ induced by inclusion, the ambient stable norms agree: $\ell_{G,N}=\ell_{G,L}$. Equivalently, the restriction isomorphism $\operatorname{res}_L^N:H^1(N;\CC)\to H^1(L;\CC)$ of \Cref{lem:finite-index-H1-iso} is an \emph{isometry} for the ambient dual stable norms. Consequently
$$
\bigl\|[N,\alpha]\bigr\|_{\ell_G^\ast}:=\ell_{G,N}^\ast(\alpha)
$$
is a well-defined norm on $\mathrm{VH}^1(G;\CC)$, depending only on the word metric $d_G$.
\end{lemma}

\begin{proof}
Identify $V:=N_{\mathrm{ab}}\otimes_\ZZ\RR\cong L_{\mathrm{ab}}\otimes_\ZZ\RR$ through the inclusion $L\hookrightarrow N$, an isomorphism after $\otimes_\ZZ\RR$ since $[N:L]<\infty$. Both norms use the same ambient metric $d_G$, only the index set of the infimum differs. Fix $a\in V$.

If $n\in L$ has $\pi_L(n)=ka$ then also $\pi_N(n)=ka$, so the constraint set for $\ell_{G,L}$ lies inside that for $\ell_{G,N}$. Taking infimum and the limit gives $\ell_{G,N}(a)\le\ell_{G,L}(a)$.

Conversely, as $L$ has finite index in $N$, there exists a normal subgroup $K$ of $N$ contained in $L$ with $q:=[N:K]<\infty$. By Lagrange's theorem in the finite \emph{group} $N/K$ we have $n^{q}\in K\subseteq L$ for every $n\in N$. Choose $n_k\in N$ with $\pi_N(n_k)=ka$ and $d_G(e,n_k)\le k\ell_{G,N}(a)+\eps_k$, with $\lim_{k\to\infty}\eps_k/k=0$. Then $n_k^{q}\in L$ with $\pi_L(n_k^{q})=q\pi_N(n_k)=qka$ under the identification $V$, and by left-invariance and the triangle inequality $d_G(e,n_k^{q})\le qd_G(e,n_k)$. Hence, computing $\ell_{G,L}(a)$ along the subsequence $m=qk$,
$$
\ell_{G,L}(a)\le\lim_{k\to\infty}\frac{1}{qk}d_G(e,n_k^{q})
\le\lim_{k\to\infty}\frac{d_G(e,n_k)}{k}\le\ell_{G,N}(a).
$$
Thus $\ell_{G,N}=\ell_{G,L}$, and applying the definition of dual norm immediately gives $\operatorname{res}_L^N:H^1(N;\CC)\to H^1(L;\CC)$ is an isometry.

Finally, any two finite-index torsion-free nilpotent subgroups are compared through their intersection (again finite-index torsion-free nilpotent), so the induced norm on $\mathrm{VH}^1(G;\CC)$ is well-defined.
\end{proof}

\section{Finite-index induction and Lipschitz seminorms}\label{sec:finite-index}

Throughout this section we use the right random walk $(X_t)$ and right-difference conventions: $X_{t+1}=X_t \xi_{t+1}$ where $(\xi_t)_t$ are i.i.d. random variables having common law $\mu$. 

\subsection{Finite-index induction restriction for Lipschitz harmonic functions}

\begin{definition}[Hitting measure]\label{def:Schreier}
Let $H\le G$ have finite index $m=[G:H]$. Define the \emph{hitting times}
$$
\tau:=\inf\{t\ge0: X_t\in H\},\qquad \tau^+:=\inf\{t\ge1: X_t\in H\}.
$$
The \emph{hitting measure} (first return distribution) on $H$ is the probability measure
\begin{equation}\label{eq:hitting-law}
\mu_{H}(y):=\PP_e\left(X_{\tau^+}=y\right)\qquad (y\in H).
\end{equation}
Let $(X_t^x)$ denote the walk starting at $x$. For the right walk, $X_t^h = h X_t^e$ (for the same realization of increments). Since $hH=H$ for $h\in H$, the stopping time $\tau^+$ is the same starting from $e$ or $h$. Thus, for every $h\in H$ and $y\in H$,
\begin{equation}\label{eq:shift-invariance-H}
\PP_h\left(X_{\tau^+}=hy\right)=\PP_e(h X_{\tau^+} = hy) = \PP_e(X_{\tau^+} = y) = \mu_{H}(y).
\end{equation}
\end{definition}

\begin{lemma}\label{lem:finite-index-hitting-time}
Let $G$ be a finitely generated group and $H\le G$ a subgroup of finite index $[G:H]<\infty$. Let $\mu$ be an adapted probability measure on $G$ and $(X_n)_{n\ge0}$ be the (right) $\mu$-random walk on $G$. Define the first hitting time of $H$ by
$$
\tau_H := \inf\{n\ge 0 : X_n\in H\}.
$$
Then
$$
\sup_{x\in G} \EE_x[\tau_H] < \infty.
$$
\end{lemma}

\begin{proof}
Since $H$ has finite index in $G$, there exists a finite-index normal subgroup $K$ of $G$ contained in $H$. Define the first hitting time of $K$ by
$$
\tau_K := \inf\{n\ge0 : X_n\in K\}.
$$
Since $K\subseteq H$, we have $\tau_H \le \tau_K$ almost surely, and hence
$$
\EE_x[\tau_H] \le \EE_x[\tau_K]\qquad\forall x\in G.
$$
Thus it suffices to prove that
\begin{equation}\label{eq:tauK-uniform}
\sup_{x\in G} \EE_x[\tau_K] < \infty.
\end{equation}

Let $Q := G/K$, which is a finite group because $K$ has finite index. Define the projected walk
$$
Y_n := KX_n\in Q.
$$
Then, $(Y_n)_{n\ge0}$ is a (right) random walk on the finite group $Q$ with step distribution $\bar \mu$ given by
$$
\bar \mu(Kg) := \sum_{k \in K}\mu(kg) \quad(g \in G).
$$
Since $\mu$ is adapted we also have that $\bar \mu$ is adapted, which implies that $(Y_n)$ is irreducible. Moreover,
$$
X_n\in K  \quad\iff\quad Y_n=e_Q:=K,
$$
so the hitting time $\tau_K$ agrees with the hitting time of the identity in $Q$:
$$
\tau_K = \inf\{n\ge0 : Y_n=e_Q\}.
$$

For an irreducible random walk on $Q$, the expected hitting time of any fixed state is finite from every starting state, and the maximum over starting states is finite. In particular, define
$$
\tau_{e_Q} := \inf\{n\ge0 : Y_n=e_Q\}.
$$
Then there exists a constant
$$
C := \max_{q\in Q} \EE_q[\tau_{e_Q}] < \infty,
$$
where $\EE_q$ denotes expectation for the chain $(Y_n)$ started at $Y_0=q$. Therefore,
$$
\EE_x[\tau_K] = \EE_q[\tau_{e_Q}] \le C
$$
for every $x\in G$, and so
$$
\sup_{x\in G} \EE_x[\tau_K] \le C < \infty,
$$
which proves \eqref{eq:tauK-uniform}. Hence, we have
$$
\sup_{x\in G} \EE_x[\tau_H]
  \le \sup_{x\in G} \EE_x[\tau_K]
  \le C < \infty.
$$ 
\end{proof}

\begin{theorem}[Induction-restriction for $\LHF$]\label{bij:LHF}
Let $H\le G$ be of finite index and let $\mu$ be adapted and smooth on $G$ (no symmetry assumed). Let $\mu_{H}$ be the hitting law from \Cref{def:Schreier}. Then the \emph{restriction} and \emph{induction} maps
$$
\mathrm{Res}_{H}^{G}:\LHF(G,\mu)\to \LHF(H,\mu_{H}),\qquad
\mathrm{Ind}_{H}^{G}: \LHF(H,\mu_{H})\to \LHF(G,\mu),
$$
given by
$$
\mathrm{Res}_{H}^{G}(f):=f|_{H},\qquad
\mathrm{Ind}_{H}^{G}(\tilde f)(x):=\EE_x\big[\tilde f(X_\tau)\big],
$$
are inverse linear isomorphisms. Quantitatively:

\emph{(Lipschitz control)} Fix finite symmetric generating sets $S_G,S_H$ of $G$ and $H$ respectively with corresponding word length functions $|\cdot|_{S_H}, |\cdot|_{S_G}$. Let $A\ge 1$ be such that $|h|_{S_H} \le A |h|_{S_G}$ for all $h\in H$. Let $m_1$ be the first moment of $\mu$ w.r.t. $|\cdot|_{S_G}$, $F$ be a set of right coset representatives of $H$ in $G$ containing $e_G$, $D=\max_{g\in F} |g|_{S_G}$ and $T=\sup_{x\in G}\EE_x[\tau]$. Define $C_{H,G} := \max_{s\in S_H} |s|_{S_G}$. Then
\begin{align}
\|\nabla_{S_H}(\mathrm{Res}_{H}^{G} f)\|_\infty &\le C_{H,G} \|\nabla_{S_G} f\|_\infty\qquad(\forall f\in \LHF(G,\mu)),\label{eq:res-Lip}\\
\|\nabla_{S_G}(\mathrm{Ind}_{H}^{G} \tilde f)\|_\infty &\le C_\ast \|\nabla_{S_H}\tilde f\|_\infty\qquad(\forall \tilde f\in \LHF(H,\mu_{H})),\label{eq:ind-Lip}
\end{align}
where
$$
C_\ast:= A \left( (4D+1) + 2 m_1 T \right).
$$

\emph{(Functoriality)}
\begin{enumerate}
\item[(F1)] (Nested subgroups) If $K\le H\le G$ are finite index, then
$$
\mathrm{Res}_{K}^{H}\circ \mathrm{Res}_{H}^{G}=\mathrm{Res}_{K}^{G},\qquad
\mathrm{Ind}_{H}^{G}\circ \mathrm{Ind}_{K}^{H}=\mathrm{Ind}_{K}^{G}.
$$
\item[(F2)] ($H$-equivariance) For $h\in H$, left translation $L_h f(x):=f(h^{-1}x)$ satisfies
$$
\mathrm{Res}_{H}^{G}(L_h f)=L_h(\mathrm{Res}_{H}^{G} f),\qquad
\mathrm{Ind}_{H}^{G}(L_h \tilde f)=L_h(\mathrm{Ind}_{H}^{G} \tilde f).
$$
\item[(F3)] (Conjugation naturality) For $g\in G$ put $H^g:=g^{-1}H g$ and let $\mu^g$ be the pushforward of $\mu$ by $x\mapsto gxg^{-1}$. Under the canonical identifications induced by conjugation, $\mathrm{Res}_{H^g}^{G}$ and $\mathrm{Ind}_{H^g}^{G}$ (w.r.t. $\mu^g$) correspond to $\mathrm{Res}_{H}^{G}$ and $\mathrm{Ind}_{H}^{G}$ (w.r.t. $\mu$).
\end{enumerate}
\end{theorem}

\begin{proof}
We use the notation established in the theorem statement. Let $(X_t)_{t\ge 0}$ be the right random walk $X_{t+1}=X_t\xi_{t+1}$, where $(\xi_t)_{t\ge 1}$ are i.i.d. with law $\mu$. We use $|\cdot|$ to denote $|\cdot|_{S_G}$.

\smallskip
\emph{Restriction.} Let $f\in\LHF(G,\mu)$. For $s\in S_H$ and $h\in H$,
$$
|\partial^{s}(f|_{H})(h)|=|f(hs)-f(h)| \le |s| \|\nabla_{S_G} f\|_\infty \le C_{H,G} \|\nabla_{S_G} f\|_\infty.
$$
Taking the supremum over $s\in S_H$ and $h \in H$ yields \eqref{eq:res-Lip}. By \cite[Proposition 3.4]{MY16}, $f|_H$ is $\mu_H$ harmonic and hence $f|_H \in \LHF(H, \mu_H)$.

\emph{Induction.} Let $\tilde f\in\LHF(H,\mu_{H})$ and define $f(x):=\EE_x[\tilde f(X_\tau)]$. By \cite[Proposition 3.4]{MY16}, $f$ is $\mu$-harmonic. If $x\in H$, $\tau=0$ almost surely, so $f(x)=\tilde f(x)$ (i.e., $f|_{H}=\tilde f$).

For the Lipschitz bound \eqref{eq:ind-Lip}, set $L_H:=\|\nabla_{S_H}\tilde f\|_\infty$. Then, we have $|\tilde{f}(h_1) - \tilde{f}(h_2)| \le L_H |h_1^{-1}h_2|_{S_H} \le (A L_H) |h_1^{-1}h_2|_{S_G}$ for all $h_1, h_2 \in H$. Let $K := A L_H$.

Let $h \in H$ and $g \in F$. Then,
\begin{align*}
|f(hg) - f(h)| &= |\EE_{hg}[\tilde{f}(X_\tau)] - \tilde{f}(h)| \quad (\text{since } f|_H = \tilde{f}) \\
&\le \EE_{hg}[|\tilde{f}(X_\tau) - \tilde{f}(h)|] \\
&\le K \EE_{hg}[|h^{-1} X_\tau|].
\end{align*}
Let the walk start at $X_0=hg$. Then,
$$
|h^{-1} X_\tau| =  \le |g| + \sum_{i=1}^t |\xi_i| \le D + \sum_{i=1}^\tau |\xi_i|.
$$
Applying Wald's identity, we get
\begin{align*}
\EE_{hg}[|h^{-1} X_\tau|] &\le D + \EE_{hg}\Big[\sum_{i=1}^\tau |\xi_i|\Big] = D + m_1 \EE_{hg}[\tau] \\
&\le D + m_1 T.
\end{align*}
Thus, $|f(hg) - f(h)| \le K(D + m_1 T)$.

Now consider the general case. Let $x \in G$ and $s \in S_G$. Write $x=h_1g_1$ and $xs=h_2g_2$ with $h_i \in H, g_i \in F$. Then,
$$
|h_1^{-1}h_2|_{S_G} = |g_1 s g_2^{-1}|_{S_G} \le |g_1|_{S_G}+|s|_{S_G}+|g_2^{-1}|_{S_G} \le D+1+D = 2D+1.
$$
By the triangle inequality:
\begin{align*}
|f(x) - f(xs)| &= |f(h_1g_1) - f(h_2g_2)| \\
&\le |f(h_1g_1) - f(h_1)| + |\tilde{f}(h_1) - \tilde{f}(h_2)| + |f(h_2) - f(h_2g_2)| \\
&\le K(D + m_1 T) + K|h_1^{-1}h_2|_{S_G} + K(D + m_1 T) \\
&\le K(2D + 2m_1 T) + K(2D+1) \\
&= K\left( (4D+1) + 2 m_1 T \right).
\end{align*}
Substituting $K=A L_H$, we obtain the bound \eqref{eq:ind-Lip} with the constant $C_\ast$.

\emph{Functoriality.} (F1) For $K\le H\le G$, let $\tau_H$ and $\tau_K$ be the hitting times to $H$ and $K$ respectively (note $\tau_H\le \tau_K$). The tower property of conditional expectations yield
\begin{align*}
\mathrm{Ind}_{K}^{G}\tilde f(x) &=\EE_x\big[\tilde f(X_{\tau_K})\big]
=\EE_x\Big[\EE_x \big[\tilde f(X_{\tau_K}) \big| \mathcal F_{\tau_H}\big]\Big].
\end{align*}
By strong Markov property, the inner expectation is $\mathrm{Ind}_{K}^{H}\tilde f\big(X_{\tau_H}\big)$ by the definition of induction on $H$. Thus
$$
\mathrm{Ind}_{K}^{G}\tilde f(x)=\EE_x\Big[\mathrm{Ind}_{K}^{H}\tilde f\big(X_{\tau_H}\big)\Big]
=\mathrm{Ind}_{H}^{G}\big(\mathrm{Ind}_{K}^{H}\tilde f\big)(x).
$$
Restriction functoriality is obvious.

(F2) For $h\in H$ and $x\in G$. Let $X_t^y$ denote the walk starting at $y$. The left translation of the path satisfies $h^{-1}X_t^{x} = h^{-1}(x \xi_1\dots\xi_t) = X_t^{h^{-1}x}$. Since $h\in H$, $X_t^x\in H \iff h^{-1}X_t^x\in H$, so the hitting times $\tau$ are the same for both walks. Thus
\begin{align*}
\mathrm{Ind}_{H}^{G}(L_h\tilde f)(x)
&=\EE_x[(L_h\tilde f)(X_\tau^x)] \\
&= \EE_x[\tilde f(h^{-1}X_\tau^x)] \\
&= \EE_{h^{-1}x}[\tilde f(X_\tau)] \\
&=(\mathrm{Ind}_{H}^{G}\tilde f)(h^{-1}x) = L_h\big(\mathrm{Ind}_{H}^{G}\tilde f\big)(x).
\end{align*}
The corresponding statement for restriction is immediate.

(F3) Let $\alpha_g(x):=gxg^{-1}$. If $(X)_t$ is the right $\mu$-walk, then $Y_t:=\alpha_g(X_t)$ is the right $\mu^g$-walk. Moreover $X_t \in H$ iff $Y_t \in H^g$, so the hitting times $\tau$ coincide. The constructions are therefore natural under the identifications induced by $\alpha_g$.
\end{proof}

\begin{corollary}[Sublinear growth $\Rightarrow$ constant]\label{cor:sublinear-constant}
Let $G$ be a finitely generated group with polynomial growth and $\mu$ be adapted and smooth. If $f\in\LHF(G,\mu)$ satisfies $|f(x)|=o(|x|)$ along the word metric, then $f$ is constant.
\end{corollary}

\begin{proof}
First assume $G$ is nilpotent. By \Cref{thm:HF_1-is-linear-poly}, there exist $c\in\CC$ and $\bar\phi\in\Hom(G_{\mathrm{ab}},\CC)$ such that
$$
f(x)=c+\bar\phi([x])\qquad(x\in G),
$$
and we set $\phi:=\bar\phi\circ \pi_{\mathrm{ab}}:G\to\CC$. Then
$$
|\phi(x)| = |f(x)-c| \le |f(x)| + |c|,
$$
so the assumption $|f(x)|=o(|x|)$ implies $|\phi(x)| = o(|x|)$ as $|x|\to\infty$.

If $[g]$ is torsion in $G_{\mathrm{ab}}$, then $\phi(g)=0$ because $(\CC,+)$ is torsion-free. Now suppose $[g]$ has infinite order. Then $g$ has infinite order in $G$, so the elements $g^n$ are pairwise distinct. Since balls in the word metric are finite, this implies $|g^n|\to\infty$ as $n\to\infty$.

For the fixed finite symmetric generating set $S$ defining $|\cdot|$, subadditivity gives
$$
|g^n| \le n|g|\qquad(\forall n\ge1).
$$
Because $|\phi(x)| = o(|x|)$ and $|g^n|\to\infty$, we have
$$
\frac{|\phi(g^n)|}{|g^n|} \longrightarrow 0\qquad(n\to\infty).
$$
On the other hand, $\phi$ is a homomorphism, so $|\phi(g^n)| = n|\phi(g)|$. Thus
$$
\frac{|\phi(g^n)|}{|g^n|}
 = \frac{n|\phi(g)|}{|g^n|}
 \ge \frac{n|\phi(g)|}{n|g|}
 = \frac{|\phi(g)|}{|g|}.
$$
If $\phi(g)\neq 0$, the right-hand side is a positive constant, which contradicts the fact that $\frac{|\phi(g^n)|}{|g^n|}\to 0$. Hence $\phi(g)=0$ for every $g$ with $[g]$ of infinite order.

Combining the torsion and infinite-order cases, we conclude $\phi\equiv 0$, hence $\bar\phi=0$, and therefore $f(x)=c$ is constant.

Now, assume $G$ to be virtually nilpotent and let $N$ be a finite-index nilpotent subgroup of $G$. Let $\mu_N$ denote the hitting measure on $N$ corresponding to $\mu$. Then, by \Cref{bij:LHF} $f|_N$ is $\mu_N$-harmonic and satisfies $|f(x)|=o(|x|_N)$ for $x\in N$ ($|\cdot|_N$ denotes word length with respect to some symmetric generating set of $N$). Hence, $f|_N$ is constant, so $f=\mathrm{Ind}_N^G(f|_N)$ is also constant.
\end{proof}

\begin{corollary}[Stability under change of measure]\label{cor:measure-stability}
If $G$ has polynomial growth and $\mu,\nu$ are $\SAS$ measures on $G$, then 
$$
\LHF(G,\mu) \cong \LHF(G, \nu).
$$
\end{corollary}

\begin{proof}
Let $N$ be a finite-index nilpotent subgroup of $G$. Since $\mu$ and $\nu$ are $\SAS$, the hitting measures $\mu_N,\nu_N$ on $N$ are also $\SAS$; in particular they are Abelian-centered by symmetry. Consequently, by \Cref{cor:norm_ident_vnilp}  and \Cref{bij:LHF} we have
$$
\LHF(G,\mu) \cong \LHF(N, \mu_N)=\Hom(N_{\mathrm{ab}},\CC)\oplus\CC=\LHF(N,\nu_N)\cong\LHF(G,\nu).
$$
\end{proof}
\begin{theorem}[Virtual normed identification]\label{thm:virtual-normed}
Let $G$ be a finitely generated group of polynomial growth, let $\mu$ be an $\SAS$ measure on $G$, let $S_G$ be a finite symmetric generating set of $G$, and let $N\in\mathcal N(G)$ have finite symmetric generating set $S_N$. Then the canonical isomorphism of \Cref{thm:virtual-affine-boundary} induces a linear isomorphism
$$
\mathcal A_{G,\mu}:\LHF(G,\mu)/\CC
\xrightarrow{ \cong }\mathrm{VH}^1(G;\CC),
\qquad
[f]\longmapsto[N,\varphi_N].
$$
This map is bi-Lipschitz for the norm induced by $\|\nabla_{S_G}\cdot\|_\infty$ on the domain and the ambient dual stable norm on the codomain: there exist
$$
C_1=C_1(G,N,S_G,S_N)>0,
\qquad
C_2=C_2(G,N,S_G,S_N,\mu)>0
$$
such that
$$
C_1^{-1}\big\|\mathcal A_{G,\mu}[f]\big\|_{\ell_G^\ast}
\le \|\nabla_{S_G}f\|_\infty
\le C_2\big\|\mathcal A_{G,\mu}[f]\big\|_{\ell_G^\ast}
\qquad(f\in\LHF(G,\mu)).
$$
If $G$ is nilpotent, with or without torsion, then $\mathcal A_{G,\mu}$ is an isometry.
\end{theorem}

\begin{proof}
By \Cref{thm:virtual-affine-boundary},
$\mathcal A_{G,\mu}[f]=[N,\varphi_N]$, where
$\varphi_N\in\Hom(N_{\mathrm{ab}},\CC)$ is the homomorphism part of the decomposition of $f|_N$ on $(N,\mu_N)$. Passing to the quotient by constants gives the asserted linear isomorphism, and
$$
\big\|\mathcal A_{G,\mu}[f]\big\|_{\ell_G^\ast}
=\ell_{G,N}^\ast(\varphi_N)
$$
by \Cref{lem:ambient-stable-fi-isometry}.

\emph{(i) Finite-index Lipschitz comparison.}
Applying \Cref{bij:LHF} with $H=N$, we have
$$
\|\nabla_{S_N}(f|_N)\|_\infty
\le C_{N,G}\|\nabla_{S_G}f\|_\infty.
$$
Since $\mathrm{Ind}_N^G\circ\mathrm{Res}_N^G=\mathrm{id}$,
\eqref{eq:ind-Lip} also gives
$$
\|\nabla_{S_G}f\|_\infty
\le C_\ast\|\nabla_{S_N}(f|_N)\|_\infty.
$$
Hence
$$
C_\ast^{-1}\|\nabla_{S_G}f\|_\infty
\le\|\nabla_{S_N}(f|_N)\|_\infty
\le C_{N,G}\|\nabla_{S_G}f\|_\infty.
$$

\emph{(ii) Nilpotent isometry on $N$.}
The hitting measure $\mu_N$ is again $\SAS$. Therefore,
\Cref{thm:coho-rigidity,lem:lip-is-dual-stable} applied to
$(N,\mu_N)$ gives
$$
\|\nabla_{S_N}(f|_N)\|_\infty
=\ell_{S_N}^\ast(\varphi_N).
$$

\emph{(iii) Intrinsic and ambient stable norms.}
Since $N$ has finite index in $G$, the restricted metric $d_G|_N$ and the word metric associated with $S_N$ are bi-Lipschitz. Consequently, there exist $D_1,D_2>0$, depending only on
$G,N,S_G,S_N$, such that
$$
D_1\ell_{G,N}^\ast(\varphi_N)
\le\ell_{S_N}^\ast(\varphi_N)
\le D_2\ell_{G,N}^\ast(\varphi_N).
$$

Combining these comparisons gives
\begin{align*}
\|\nabla_{S_G}f\|_\infty
&\le C_\ast D_2\ell_{G,N}^\ast(\varphi_N),\\
\|\nabla_{S_G}f\|_\infty
&\ge C_{N,G}^{-1}D_1\ell_{G,N}^\ast(\varphi_N).
\end{align*}
Thus the asserted bounds hold with
$$
C_1:=\frac{C_{N,G}}{D_1},
\qquad
C_2:=C_\ast D_2.
$$
Here $C_{N,G},D_1,D_2$ depend only on
$G,N,S_G,S_N$, whereas $C_\ast$ may also depend on $\mu$.

When $G$ is torsion-free nilpotent, take $N=G$ and $S_N=S_G$. Then restriction and induction are the identity maps, the intrinsic and ambient stable norms coincide, and \Cref{thm:coho-rigidity} gives
$$
\big\|\mathcal A_{G,\mu}[f]\big\|_{\ell_G^\ast}
=\|\nabla_{S_G}f\|_\infty.
$$

Suppose now that $G$ is nilpotent with torsion. By
\Cref{lem:finite-index-H1-iso}, inclusion induces a canonical identification
$$
N_{\mathrm{ab}}\otimes_\ZZ\RR
\cong G_{\mathrm{ab}}\otimes_\ZZ\RR.
$$
Under this identification, the power argument in the proof of
\Cref{lem:ambient-stable-fi-isometry} gives
$$
\ell_{G,N}=\ell_{S_G}.
$$
Moreover, \Cref{thm:coho-rigidity} gives
$f(x)=c+\varphi([x]_G)$, with
$\varphi_N=\varphi|_N$ and
$$
\|\nabla_{S_G}f\|_\infty=\ell_{S_G}^\ast(\varphi).
$$
It follows that
$$
\big\|\mathcal A_{G,\mu}[f]\big\|_{\ell_G^\ast}
=\ell_{G,N}^\ast(\varphi_N)
=\ell_{S_G}^\ast(\varphi)
=\|\nabla_{S_G}f\|_\infty.
$$
\end{proof}

The lower comparison is sharpened to the optimal constant $C_1=1$ in
\Cref{prop:asymptotic-isometry}. The dependence of $C_2$ on $\mu$ is shown to be unavoidable in \Cref{exmp:dihedral-inflation}.

\subsection{Sharpness: contractivity, gradient inflation, and the asymptotic isometry}\label{subsec:raw-sharpness}

\begin{definition}[Asymptotic Lipschitz seminorm]\label{def:asymptotic-lipschitz}
For $f:G\to\CC$ set
$$
\|f\|_{\mathrm{as}}:=\limsup_{d_G(x,y)\to\infty}\frac{|f(x)-f(y)|}{d_G(x,y)}.
$$
Bounded perturbations of $f$ leave $\|f\|_{\mathrm{as}}$ unchanged, and $\|f\|_{\mathrm{as}}\le\|\nabla_{S_G}f\|_\infty$ always.
\end{definition}

\begin{proposition}[Contractivity and asymptotic isometry]\label{prop:asymptotic-isometry}
Let $G$, $\mu$, $N$, $S_G$ be as in \Cref{thm:virtual-normed}. Then for every $f\in\LHF(G,\mu)$,
$$
\big\|\mathcal A_{G,\mu}[f]\big\|_{\ell_G^\ast}=\|f\|_{\mathrm{as}}\le\|\nabla_{S_G}f\|_\infty.
$$
In particular $\mathcal A_{G,\mu}$ is an isometry of $(\LHF(G,\mu)/\CC,\|\cdot\|_{\mathrm{as}})$ onto $(\mathrm{VH}^1(G;\CC),\|\cdot\|_{\ell_G^\ast})$ for every finitely generated group of polynomial growth and every $\SAS$ measure, and it is $1$-Lipschitz from the Lipschitz seminorm. Moreover, in \Cref{thm:virtual-normed} one may take $C_1=1$.
\end{proposition}

\begin{proof}
Write $\mathcal A_{G,\mu}[f]=[N,\varphi_N]$, so that $f(n)=f(e)+\varphi_N([n])$ for $n\in N$ (\Cref{thm:virtual-affine-boundary}), and set $L:=\ell_{G,N}^\ast(\varphi_N)=\|\mathcal A_{G,\mu}[f]\|_{\ell_G^\ast}$ (\Cref{lem:ambient-stable-fi-isometry}).

\emph{Lower bound $\|f\|_{\mathrm{as}}\ge L$.} Fix $a\in\pi_N(N)$, $a\neq0$. By \Cref{def:ambient-stable-vH1} there are $n_k\in N$ with $\pi_N(n_k)=ka$ and $d_G(e,n_k)\le k(\ell_{G,N}(a)+\eps_k)$ with $\lim_{k\to\infty}\eps_k=0$. Since $n_k$ are pairwise distinct and balls of $d_G$ are finite, so $d_G(e,n_k)\to\infty$ as $k\to\infty$. Since $f(n_k)-f(e)=\varphi_N(ka)=k\varphi_N(a)$,
$$
\|f\|_{\mathrm{as}} \ge \limsup_{k\to\infty}\frac{k|\varphi_N(a)|}{k(\ell_{G,N}(a)+\eps_k)} = \frac{|\varphi_N(a)|}{\ell_{G,N}(a)}.
$$
Now, for any $a\in N_{\mathrm{ab}}\otimes_\ZZ \RR\setminus\{0\}$, there exists $a_k\in \pi(N)\setminus\{0\}$ and $m_k\in\ZZ$ such that $\frac{1}{m_k}a_k\to a$ as $k\to\infty$. Therefore, $$ \frac{|\varphi_N(a)|}{\ell_{G,N}(a)}=\lim_{k\to\infty}\frac{|\varphi_N(\frac{1}{m_k}a_k)|}{\ell_{G,N}(\frac{1}{m_k}a_k)}=\lim_{k\to\infty}\frac{|\varphi_N(a_k)|}{\ell_{G,N}(a_k)}\le \|f\|_{\mathrm{as}}.
$$ 
\emph{Upper bound $\|f\|_{\mathrm{as}}\le L$.} For $m\in N$, 
$$
\inf\bigl\{d_G(e,n) : n\in N, \pi_N(n)=k\pi_N(m)\bigr\}\le d_G(e,m^k)\le kd_G(e,m),
$$ 
which implies $\ell_{G,N}(\pi_N(m))\le d_G(e,m)$.
Hence for $x,y\in N$,
$$
|f(x)-f(y)|=|\varphi_N(\pi_N(x^{-1}y))|\le L\ell_{G,N}(\pi_N(x^{-1}y))\le Ld_G(x,y).
$$
Since $[G:N]<\infty$, $N$ is $C$-coarsely dense in $(G,d_G)$ for some $C>0$. Given $x,y\in G$ pick $x',y'\in N$ with $d_G(x,x')\le c$ and $d_G(y,y')\le c$; then
$$
|f(x)-f(y)|\le|f(x')-f(y')|+2c\|\nabla_{S_G}f\|_\infty\le Ld_G(x,y)+2c(L+\|\nabla_{S_G}f\|_\infty),
$$
and dividing by $d_G(x,y)\to\infty$ gives $\|f\|_{\mathrm{as}}\le L$. The inequality $\|f\|_{\mathrm{as}}\le\|\nabla_{S_G}f\|_\infty$ is immediate from the definition, and combining everything gives the contractivity $\|\mathcal A_{G,\mu}[f]\|_{\ell_G^\ast}\le\|\nabla_{S_G}f\|_\infty$, that is, $C_1=1$.
\end{proof}

\begin{lemma}[Exact coset-affine structure]\label{lem:coset-affine}
Let $G$, $\mu$, $N$, $S_G$ be as in \Cref{thm:virtual-normed}. Let $Q:=N\backslash G$ be the finite set of right cosets with right $G$-action $Nx\cdot g:=Nxg$, and fix a section $\sigma:Q\to G$ with $\sigma(Ne)=e$, so that every $x\in G$ is uniquely $x=n\sigma(q)$ with $n\in N$, $q\in Q$. Define $a(q,g)\in N$ by $\sigma(q)g=a(q,g)\sigma(q\cdot g)$ for $g\in G$. Then every $f\in\LHF(G,\mu)$ has the exact form
\begin{equation}\label{eq:coset-affine-form}
f(n\sigma(q))=f(e)+\varphi_N([n])+u_\mu(q)\qquad(n\in N, q\in Q),
\end{equation}
where $\varphi_N$ is the homomorphism part of $f|_N$ and $u_\mu:Q\to\CC$, normalized by $u_\mu(Ne)=0$, is the unique solution of the equation
\begin{equation}\label{eq:Fredholm-alternative}
u_\mu(q)=\sum_{g\in G}\mu(g)\bigl(\varphi_N([a(q,g)])+u_\mu(q\cdot g)\bigr)\qquad(q\in Q).
\end{equation}
Consequently the Lipschitz seminorm is the exact finite quantity
$$
\|\nabla_{S_G}f\|_\infty=\max_{q\in Q, s\in S_G}\bigl|\varphi_N([a(q,s)])+u_\mu(q\cdot s)-u_\mu(q)\bigr|.
$$
\end{lemma}

\begin{proof}
The series $\sum_{g\in G}\mu(g)\varphi_N([a(q,g)])$ converges absolutely for every $q\in Q$ because $f$ is Lipschitz and $\mu$ is smooth. Let $(X_t)$ denote the $\mu$-random walk on $G$. Then, the projected random walk $(NX_t)$ on the finite set $Q$ is irreducible because $\mu$ is adapted, and doubly stochastic because each right translation is a bijection of $Q$, so the uniform measure is stationary. The section identity $\sigma(q\cdot g)g^{-1}=a(q,g)^{-1}\sigma(q)$ reads $a(q\cdot g,g^{-1})=a(q,g)^{-1}$. Reindexing $q\mapsto q\cdot g$ and using the symmetry of $\mu$ therefore gives
$$
\sum_{q\in Q}\sum_{g\in G}\mu(g)\varphi_N([a(q,g)])=-\sum_{q\in Q}\sum_{g\in G}\mu(g)\varphi_N([a(q,g)]),
$$
so $ \sum_{q\in Q}F(q)=0$, where $F(q)=\sum_{g\in G}\mu(g)\varphi_N([a(q,g)])$, ($q\in Q$). If $P$ denotes the transition kernel of the projected random walk, then the equation 
$$
(I-P)u=F
$$
has a solution $u_\mu$ by Fredholm's alternative, which is unique up to an additive constant by irreducibility. Forcing $u_\mu(Ne)=0$ makes the solution unique.

Define $\Phi(n\sigma(q)):=f(e)+\varphi_N([n])+u_\mu(q)$. Then,
\begin{equation}\label{eq:gradient-increment}
\Phi(n\sigma(q)g)-\Phi(n\sigma(q))=\varphi_N([a(q,g)])+u_\mu(q\cdot g)-u_\mu(q)
\end{equation}
which is independent of $n$. Taking $g=s\in S_G$ and using that fact that $Q$ is finite shows $\Phi$ is Lipschitz, and averaging the \eqref{eq:gradient-increment} against $\mu$ shows that \eqref{eq:Fredholm-alternative} is precisely the $\mu$-harmonicity of $\Phi$. The difference $W:=f-\Phi$ is then Lipschitz and $\mu$-harmonic, and it is bounded: 
$$
W(n\sigma(q))=\bigl(f(n\sigma(q))-f(n)\bigr)-u_\mu(q),\qquad |f(n\sigma(q))-f(n)|\le\|\nabla_{S_G}f\|_\infty\max_{q'}|\sigma(q')|_{S_G}.
$$
By \Cref{cor:sublinear-constant}, $W$ is constant, and $W(e)=0$ proves \eqref{eq:coset-affine-form}. The gradient formula is immediate from \eqref{eq:gradient-increment}.
\end{proof}

\begin{exmp}[Dependence of the norm comparison on the measure]\label{exmp:dihedral-inflation}
Let $G=D_\infty=\langle r,t\mid t^2=e, trt=r^{-1}\rangle$ with $S_G=\{r,r^{-1},t\}$ and $N=\langle r\rangle\cong\ZZ$, of index two, with $S_N=\{r,r^{-1}\}$. We use the section $\sigma(N)=e$, $\sigma(Nt)=t$. Every element is uniquely $r^nt^{\varepsilon}$ with $\varepsilon\in\{0,1\}$, and $d_G(e,r^n)=|n|$: the letter $t$ contributes no $r$-exponent, so a word for $r^n$ contains at least $|n|$ letters $r^{\pm1}$. In particular, $\ell_{G,N}(1)=1$ under $N_{\mathrm{ab}}\otimes_\ZZ\RR\cong\RR$, $[r]_N\mapsto1$. For an integer $M\ge2$, set
$$
\mu_M=\frac14\delta_r+\frac14\delta_{r^{-1}}+\frac12\delta_{r^Mt}.
$$
Then $\mu_M$ is symmetric since $r^Mt$ is an involution, adapted since $t=r^{-M}(r^Mt)$, and finitely supported. For $\alpha\in\CC$, define
$$
f_\alpha(r^n)=\alpha n,\qquad f_\alpha(r^nt)=\alpha(n-M).
$$
It is easy to check that $f_\alpha$ is $\mu_M$-harmonic, $f_\alpha\in\LHF(G,\mu_M)$, and it is the coset-affine function of \Cref{lem:coset-affine} with $\varphi_\alpha([r]_N)=\alpha$ and
$$
u_{\mu_M}(N)=0,\qquad u_{\mu_M}(Nt)=-M\alpha.
$$
Observe that
$$
\big\|\mathcal A_{G,\mu_M}[f_\alpha]\big\|_{\ell_G^\ast}
=|\alpha|=\|f_\alpha\|_{\mathrm{as}},
\qquad
\|\nabla_{S_G}f_\alpha\|_\infty=M|\alpha|.
$$
Hence the optimal constant $C_2$ of \Cref{thm:virtual-normed} for $(G,N,S_G,S_N,\mu_M)$ is at least $M$. Thus, with $G$, $N$, and the word metrics fixed, no bound on $C_2$ in terms of $(G,N,S_G,S_N)$ alone is possible, and the identification in \Cref{prop:asymptotic-isometry} is not an isometry for the induced Lipschitz norm even for virtually cyclic groups.

The failure of isometry does not require the measure to vanish on a generator. Take $S_G'=\{r,r^{-1},t,r^Mt\}$ and let
$$
\mu=p_r(\delta_r+\delta_{r^{-1}})+p_t\delta_t+p_M\delta_{r^Mt},
$$
where
$$
p_r,p_t,p_M>0,\qquad 2p_r+p_t+p_M=1.
$$
Then $\mu$ is an $\SAS$ measure with $\supp(\mu)=S_G'$. The coset-affine harmonic function $h_\alpha$, normalized by $h_\alpha(e)=0$ and satisfying $\varphi_N([r]_N)=\alpha$, has
$$
u_\mu(Nt)=-A\alpha,\qquad A=\frac{p_MM}{p_t+p_M},
$$
and
$$
\|\nabla_{S_G'}h_\alpha\|_\infty
=|\alpha|\max\{1,A,M-A\}.
$$
For the metric $d_{S_G'}$,
$$
\ell_{G,N}(1)=\frac{2}{M}.
$$
Indeed, $(r^Mt)t=r^M$ gives one inequality. Conversely, if a word representing $r^n$ contains $k$ letters $r^{\pm1}$ and $2j$ reflection letters, then
$$
|n|\le k+jM\le\frac{M}{2}(k+2j),
$$
which gives the reverse inequality. Hence
$$
\big\|\mathcal A_{G,\mu}[h_\alpha]\big\|_{\ell_G^\ast}
=\frac{M}{2}|\alpha|.
$$
For $\alpha\ne0$,
$$
\big\|\mathcal A_{G,\mu}[h_\alpha]\big\|_{\ell_G^\ast}
=\|\nabla_{S_G'}h_\alpha\|_\infty
$$
precisely when $A=M/2$.
\end{exmp}

\section{Family-uniform convolution-gradient estimates}\label{sec:uniform-dungey}

The purpose of this section is to prove \Cref{lem:uniform-dungey-family}, which
supplies the diagonal Dungey-type gradient bounds used in the proof of
\Cref{thm:unconditional-HF1-closure}. The argument is motivated by the results
in \cite{Dungey2008}. Throughout this section, $G$ is a finitely generated
nilpotent group with polynomial growth degree $D$,
$W\subseteq G$ is a fixed finite symmetric generating set containing $e$ and $|\cdot|_W$ denotes the word length function on $G$ corresponding to $W$.
Let
$$
\pi=(\pi_1,\dots,\pi_d):G\to \mathbb Z^d
$$
be the projection onto the free part of the Abelianization of $G$, with kernel $G_1$.
Choose lifts $x_1,\dots,x_d\in G$ of the standard basis vectors of $\ZZ^d$, and fix a finite
symmetric generating set $U$ of $G_1$.
For $f:G\to\CC$, define
$$
\|\nabla_W f\|_2^2:=\sum_{w\in W}\|\partial_w f\|_2^2, \qquad \partial_g:=L(g)-I,
$$
where $(L(g)f)(x):=f(g^{-1}x)$.
Then
\begin{equation}\label{eq:app-tele-bound}
\|\partial_g f\|_2\le |g|_W\|\nabla_W f\|_2 \qquad (g\in G),
\end{equation}
by writing a geodesic word for $g$ and telescoping.

As in \cite{Dungey2008}, let $D_2$ be the set of all bounded operators on $\ell^2(G)$ which are finite complex linear combination of the terms
$$
L(a)\partial_b\partial_c\qquad(a,b,c\in G).
$$
By \cite[Lemma 2.3]{Dungey2008}, $\partial_u\in D_2$ for every $u\in G_1$.
Since $U$ is finite, we fix once and for all representations
\begin{equation}\label{eq:app-fixed-U-D2}
\partial_u=\sum_{\ell=1}^{N_u} \gamma_{u,\ell}
L(a_{u,\ell})\partial_{b_{u,\ell}}\partial_{c_{u,\ell}}
\qquad (u\in U),
\end{equation}
with $a_{u,\ell},b_{u,\ell},c_{u,\ell}\in G$ and $\gamma_{u,\ell}\in\CC$.

For $s\in W$, write
$q_j(s):=\pi_j(s)\in\ZZ$ and $v_s:=\left(\prod_{j=1}^d x_j^{-q_j(s)}\right)s\in G_1$, and put
$$
E_s:=\partial_s-\sum_{j=1}^d q_j(s)\partial_{x_j}.
$$
By \cite[Lemmas 2.2 and 2.3]{Dungey2008} $E_s \in D_2$ for every $s\in W$. So we may fix once and for all
finite $D_2$ representations
\begin{equation}\label{eq:app-fixed-W-error-D2}
E_s=\sum_{\ell=1}^{N_s}\eta_{s,\ell}
L(\alpha_{s,\ell})\partial_{\beta_{s,\ell}}\partial_{\gamma_{s,\ell}}
\qquad(s\in W),
\end{equation}
with all coefficients and group elements drawn from a finite set depending only
on $(G,W)$.

A probability measure $K$ on $G$ is said to be \emph{centered} if
\begin{equation}\label{eq:dungey-centered}
\sum_{g\in G}K(g) \pi_j(g)=0
\qquad(j=1,\ldots,d).
\end{equation}
Note that if $K$ has finite first moment and is Abelian-centered, then $K$ is centered in the sense of \eqref{eq:dungey-centered}.

Let $(K_R)_{R\ge R_0}$ be a family of centered, finitely supported probability
measures on $G$, and write
$$
T_R f:=K_R*f.
$$
Throughout, we assume the standing ellipticity and exponential-moment
hypotheses
\begin{equation}\label{eq:app-standing-hyp}
\inf_{R\ge R_0}K_R(e)\ge\eta,\qquad
\inf_{R\ge R_0}\min_{w\in W}K_R(w)\ge\eta,\qquad
\mathfrak M_\beta:=\sup_{R\ge R_0}\sum_{g\in G}e^{\beta|g|_W)}K_R(g)<\infty,
\end{equation}
for fixed constants $\eta>0$, $\beta>0$.

\begin{lemma}[Quantitative $D_2$ decomposition]\label{lem:quantitative-D2-app}
There exist constants $C>0$ and $m_0\in\NN$, depending only on $(G,W)$,
such that for every $g\in G$ there is an operator $\mathcal R_g$ on $\ell^2(G)$ satisfying
\begin{equation}\label{eq:app-D2-decomp}
\partial_g=\sum_{j=1}^d \pi_j(g)\partial_{x_j}+\mathcal R_g.
\end{equation}
Moreover, $\mathcal R_g$ admits a finite $D_2$ expansion
\begin{equation}\label{eq:app-controlled-D2-expansion}
\mathcal R_g=\sum_{\nu\in I_g}\Gamma_\nu
L(A_\nu)\partial_{B_\nu}\partial_{C_\nu}
\end{equation}
with linearly controlled translation lengths
\begin{equation}\label{eq:app-controlled-D2-linear-length}
|A_\nu|_W+|B_\nu|_W+|C_\nu|_W\le C(1+|g|_W)
\qquad(\nu\in I_g)
\end{equation}
and polynomially controlled coefficient size
\begin{equation}\label{eq:app-controlled-D2-size}
\sum_{\nu\in I_g}|\Gamma_\nu|
\bigl(1+|A_\nu|_W+|B_\nu|_W+|C_\nu|_W\bigr)^4
\le C(1+|g|_W)^{m_0}.
\end{equation}
In particular,
\begin{equation}\label{eq:app-D2-bilinear}
|\langle\mathcal R_g f_1,f_2\rangle|
\le
C(1+|g|_W)^{m_0}
\|\nabla_W f_1\|_2\|\nabla_W f_2\|_2
\end{equation}
for all $f_1,f_2\in \ell^2(G)$.
Consequently, if $K$ is a centered, finitely supported probability measure on $G$
with $Tf:=K*f$, then
\begin{equation}\label{eq:app-centered-laplacian}
I-T=-\sum_{g\in G}K(g)\mathcal R_g
\end{equation}
and
\begin{equation}\label{eq:app-centered-bilinear}
|\langle(I-T)f_1,f_2\rangle|
\le
CM_{m_0}(K)\|\nabla_W f_1\|_2\|\nabla_W f_2\|_2,
\end{equation}
where
$$
M_{m_0}(K):=\sum_{g\in G}(1+|g|_W)^{m_0}K(g).
$$
\end{lemma}

\begin{proof}
The case $g=e$ is trivial, so assume $g\ne e$. Choose a $W$-geodesic word
$$
g=s_1 s_2\cdots s_n,\qquad s_i\in W,\qquad n=|g|_W,
$$
and set $P_i:=s_1\cdots s_{i-1}$, with $P_1=e$. Repeated use of
$\partial_{ab}=\partial_a+L(a)\partial_b$ gives
\begin{equation}\label{eq:app-tele-d-g}
\partial_g=\sum_{i=1}^n L(P_i)\partial_{s_i}.
\end{equation}
Substituting $\partial_{s_i}=\sum_j q_j(s_i)\partial_{x_j}+E_{s_i}$ and using
$\sum_i q_j(s_i)=\pi_j(g)$, we define
\begin{align}
\mathcal R_g
:=\partial_g-\sum_{j=1}^d\pi_j(g)\partial_{x_j}
&=\sum_{i=1}^n L(P_i)E_{s_i}
  +\sum_{i=1}^n\sum_{j=1}^d q_j(s_i)(L(P_i)-I)\partial_{x_j}.
\label{eq:app-Rg-decomposition}
\end{align}
This is the desired operator $\mathcal R_g$, satisfying \eqref{eq:app-D2-decomp}.

\smallskip\noindent
Insert the fixed representations \eqref{eq:app-fixed-W-error-D2} into the
first sum of \eqref{eq:app-Rg-decomposition}. Let $C_1\ge1$ be a
constant, depending only on $(G,W)$, such that
$$
|\alpha_{s,\ell}|_W+|\beta_{s,\ell}|_W+|\gamma_{s,\ell}|_W \le C_1
$$
for every summand occurring in the fixed representations
\eqref{eq:app-fixed-W-error-D2}, and such that $|x_j|_W\le C_1$
for every $1\le j\le d$. Also let $C_2\ge1$ bound the finitely
many numbers $|\eta_{s,\ell}|$ and $|q_j(s)|$ with $s\in W$ and
$1\le j\le d$.

The first sum produces terms indexed by
$$
I_g^{(1)}:=\{(i,\ell):1\le i\le n, 1\le \ell\le N_{s_i}\}
$$
of the form
$$
\eta_{s_i,\ell}L(P_i\alpha_{s_i,\ell}) \partial_{\beta_{s_i,\ell}}\partial_{\gamma_{s_i,\ell}}.
$$
For these terms,
$$
|P_i\alpha_{s_i,\ell}|_W+|\beta_{s_i,\ell}|_W+|\gamma_{s_i,\ell}|_W\le |P_i|_W+C_1\le C_1(1+n).
$$
Their number is
$$
|I_g^{(1)}|=\sum_{i=1}^n N_{s_i}\le N_W n, \qquad N_W:=\max_{s\in W}N_s<\infty.
$$

For the second sum of \eqref{eq:app-Rg-decomposition}, expand
$$
L(P_i)-I=\partial_{P_i}=\sum_{r=1}^{i-1}L(P_r)\partial_{s_r},
$$
which gives
$$
(L(P_i)-I)\partial_{x_j}=\sum_{r=1}^{i-1}L(P_r)\partial_{s_r}\partial_{x_j}.
$$
These terms are indexed by
$$
I_g^{(2)}:=\{(i,j,r):1\le i\le n, 1\le j\le d, 1\le r\le i-1\},
$$
and have the form
$$
q_j(s_i)L(P_r)\partial_{s_r}\partial_{x_j}.
$$
For these terms,
$$
|P_r|_W+|s_r|_W+|x_j|_W\le (n-1)+1+C_1 \le C_1(1+n),
$$
and their number is
$$
|I_g^{(2)}|=d\sum_{i=1}^n(i-1)=\frac{d n(n-1)}2.
$$
Thus the total number of displayed $D_2$ summands is $O(n^2)$.
Taking $I_g$ to be the disjoint union of these two index families, with the parameters
and coefficients displayed above, gives \eqref{eq:app-controlled-D2-expansion}. 
Moreover,
$$
\begin{aligned}
\sum_{\nu\in I_g}|\Gamma_\nu|
\bigl(1+|A_\nu|_W+|B_\nu|_W+|C_\nu|_W\bigr)^4
&\le C_2\bigl(|I_g^{(1)}|+|I_g^{(2)}|\bigr)
   \bigl(1+n+C_1\bigr)^4  \\
&\le C(1+n)^2(1+n)^4
\le C(1+n)^6.
\end{aligned}
$$
for some $C>0$ depending only on $(G,W)$. The length estimate \eqref{eq:app-controlled-D2-linear-length} follows after increasing $C$, if required. Since $n=|g|_W$, \eqref{eq:app-controlled-D2-size} holds with $m_0=6$.

\smallskip\noindent
For a single term $A=L(a)\partial_b\partial_c$ in
\eqref{eq:app-controlled-D2-expansion}, we have
$$
\langle Af_1,f_2\rangle=\langle \partial_c f_1,\partial_{b^{-1}}L(a^{-1})f_2\rangle.
$$
Writing $b^{-1}$ as a $W$-geodesic word and using
$\partial_w L(a^{-1})=L(a^{-1})\partial_{awa^{-1}}$ for $w\in W$ together with
the trivial bound $|awa^{-1}|_W\le2|a|_W+1$ and \eqref{eq:app-tele-bound}, we obtain
$$
\|\partial_{b^{-1}}L(a^{-1})f_2\|_2\le|b|_W(2|a|_W+1)\|\nabla_W f_2\|_2.
$$
Combining this with \eqref{eq:app-tele-bound} applied to $\partial_c f_1$, and Cauchy-Schwarz inequality gives
\begin{equation}\label{eq:app-D2-term-with-conj}
|\langle Af_1,f_2\rangle|\le |c|_W|b|_W(2|a|_W+1)
\|\nabla_W f_1\|_2\|\nabla_W f_2\|_2.
\end{equation}
Therefore the $\nu$-th summand, including its coefficient $\Gamma_\nu$,
contributes at most
$$
2|\Gamma_\nu|\bigl(1+|A_\nu|_W+|B_\nu|_W+|C_\nu|_W\bigr)^3 \|\nabla_W f_1\|_2\|\nabla_W f_2\|_2.
$$
Since $1+|A_\nu|_W+|B_\nu|_W+|C_\nu|_W\ge1$, summing over the controlled
expansion and using \eqref{eq:app-controlled-D2-size} proves
\eqref{eq:app-D2-bilinear}.

\smallskip\noindent
If $K$ is centered then $\sum_g K(g)\pi_j(g)=0$ for $j=1,\dots,d$, so
averaging \eqref{eq:app-D2-decomp} against $K$ gives
\eqref{eq:app-centered-laplacian}, and \eqref{eq:app-centered-bilinear} follows
immediately from \eqref{eq:app-D2-bilinear}.
\end{proof}

\begin{corollary}[Uniform sectoriality of the family]\label{cor:uniform-sectorial-family-app}
Under the standing hypotheses \eqref{eq:app-standing-hyp}, there exist constants
$c_0,c_1>0$, depending only on $(G,W,\eta,\beta,\mathfrak M_\beta)$, such that
\begin{equation}\label{eq:app-family-form-comparison}
c_0\|\nabla_W f\|_2^2
\le \Re\langle(I-T_R)f,f\rangle
\le c_1\|\nabla_W f\|_2^2
\qquad (R\ge R_0, f\in \ell^2(G)),
\end{equation}
and consequently
\begin{equation}\label{eq:app-family-sectorial}
\bigl|\langle(I-T_R)f,f\bigr\rangle|
\le \frac{c_1}{c_0}\Re\langle(I-T_R)f,f\rangle
\qquad (R\ge R_0, f\in \ell^2(G)).
\end{equation}
\end{corollary}

\begin{proof}
The exponential-moment hypothesis in \eqref{eq:app-standing-hyp} implies
$\sup_{R\ge R_0}M_{m_0}(K_R)<\infty$. Hence \Cref{lem:quantitative-D2-app}
(with $K=K_R$) gives
$$
\bigl|\langle(I-T_R)f,f\rangle\bigr|\le c_1\|\nabla_W f\|_2^2
$$
for some $c_1$ depending only on $(G,W,\eta,\beta,\mathfrak M_\beta)$.
On the other hand, 
$$
\Re\langle(I-T_R)f,f\rangle=\frac12\sum_{g\in G} K_R(g)\|\partial_g f\|_2^2 \ge \frac{\eta}{2}\|\nabla_W f\|_2^2,
$$
because $ K_R(w)\ge\eta$ for every $w\in W$. This proves
\eqref{eq:app-family-form-comparison},
and \eqref{eq:app-family-sectorial} follows immediately.
\end{proof}

For $\psi:G\to\RR$ and $\lambda\in\RR$, define
$$
T_{R,\lambda,\psi}:=e^{\lambda\psi}T_Re^{-\lambda\psi}.
$$

\begin{lemma}[Uniform form perturbation under a Lipschitz weight]\label{lem:uniform-perturb-app}
Under the standing hypotheses \eqref{eq:app-standing-hyp}, let $\psi:G\to\RR$ satisfy
$$
|\partial_w\psi|\le 1\qquad(w\in W).
$$
Then there exist $\lambda_0>0$ and $C_{\mathrm{pert}}>0$, depending only on
$(G,W,\eta,\beta,\mathfrak M_\beta)$, such that for every $R\ge R_0$, every $|\lambda|\le\lambda_0$,
every $\eps>0$, and all $f_1,f_2\in\ell^2(G)$,
\begin{equation}\label{eq:app-uniform-quad-perturb}
\bigl|\langle(T_{R,\lambda,\psi}-T_R)f_1,f_2\bigr\rangle|
\le \eps\bigl(\|\nabla_W f_1\|_2^2+\|\nabla_W f_2\|_2^2\bigr)
+C_{\mathrm{pert}}(1+\eps^{-1})\lambda^2\bigl(\|f_1\|_2^2+\|f_2\|_2^2\bigr).
\end{equation}
Consequently,
\begin{equation}\label{eq:app-uniform-quad-perturb-form}
\Re\langle(I-T_{R,\lambda,\psi})f,f\rangle
\ge \frac{c_0}{2}\|\nabla_W f\|_2^2
-C_{\mathrm{pert}}\lambda^2\|f\|_2^2
\qquad (R\ge R_0, |\lambda|\le\lambda_0, f\in \ell^2(G)).
\end{equation}
\end{lemma}

\begin{proof}
By \eqref{eq:app-centered-laplacian} applied to $K=K_R$,
$$
T_{R,\lambda,\psi}-T_R =\sum_{g\in G}K_R(g)\bigl(e^{\lambda\psi}\mathcal R_g e^{-\lambda\psi}-\mathcal R_g\bigr).
$$
It therefore suffices to estimate a generic term $A=L(a)\partial_b\partial_c$
in the expansion of $\mathcal R_g$ from
\Cref{lem:quantitative-D2-app}.

For $d\in G$, define the multiplication operator $M_d(\lambda)\text{ on }\ell^2(G),$
$$
 (M_d(\lambda)f)(x):=\bigl(e^{-\lambda\partial_d\psi(x)}-1\bigr)f(x),
$$
so that $M_d(\lambda)$ is multiplication by the function
$e^{-\lambda\partial_d\psi}-1$. Since $|\partial_w\psi|\le 1$ for $w\in W$,
telescoping along a geodesic for $d$ gives $|\partial_d\psi|\le|d|_W$ pointwise,
hence
\begin{equation}\label{eq:app-Md-bound}
\|M_d(\lambda)\|_{2\to2}\le |\lambda||d|_We^{|\lambda||d|_W}.
\end{equation}
Direct computation yields the basic identities
$$
e^{\lambda\psi}L(a)e^{-\lambda\psi}=(1+M_a(\lambda))L(a), \qquad e^{\lambda\psi}\partial_d e^{-\lambda\psi}=\partial_d+M_d(\lambda)L(d).
$$
Hence
\begin{align}
e^{\lambda\psi}A e^{-\lambda\psi}-A
&=M_a(\lambda)L(a)\partial_b\partial_c\notag\\
&\quad+(1+M_a(\lambda))L(a)M_b(\lambda)L(b)\partial_c\notag\\
&\quad+(1+M_a(\lambda))L(a)\partial_b M_c(\lambda)L(c)\notag\\
&\quad+(1+M_a(\lambda))L(a)M_b(\lambda)L(b)M_c(\lambda)L(c).
\label{eq:app-full-twist-expansion}
\end{align}

\smallskip\noindent
Consider the second term, $(1+M_a(\lambda))L(a)M_b(\lambda)L(b)\partial_c$.
Using \eqref{eq:app-tele-bound}, we get
$$
\bigl|\langle(1+M_a(\lambda))L(a)M_b(\lambda)L(b)\partial_c f_1,f_2\rangle\bigr| \le |c|_W\|M_b(\lambda)\|_\infty(1+\|M_a(\lambda)\|_\infty)\|\nabla_W f_1\|_2\|f_2\|_2.
$$
Combined with \eqref{eq:app-Md-bound} and Young's inequality
$st\le \eps s^2+(4\eps)^{-1}t^2$, this is bounded by
$$
\eps|c|_W^2\|\nabla_W f_1\|_2^2 +C\eps^{-1}\lambda^2(1+|a|_W)^2|b|_W^2 e^{C|\lambda|(|a|_W+|b|_W)}\|f_2\|_2^2.
$$
For the third term $(1+M_a(\lambda))L(a)\partial_b M_c(\lambda)L(c)$, put
$m_a(\lambda):=(1+M_a(\lambda))^*=e^{-\lambda\partial_a\psi}$ (recall $\psi$
is real-valued). Then,
$$
\langle (1+M_a(\lambda))L(a)\partial_b M_c(\lambda)L(c)f_1,f_2\rangle =\langle M_c(\lambda)L(c)f_1,\partial_{b^{-1}}L(a^{-1})\bigl(m_a(\lambda)f_2\bigr)\rangle.
$$
Writing $b^{-1}$ as a $W$-geodesic word and using
$\partial_w L(a^{-1})=L(a^{-1})\partial_{awa^{-1}}$ together with
$|awa^{-1}|_W\le 2|a|_W+1$, the product rule
$\partial_g(\varphi h)=(\partial_g\varphi)L(g)h+\varphi\partial_g h$ for a
multiplier $\varphi$ and the bound
$$
\|\partial_h e^{-\lambda\partial_a\psi}\|_\infty \le C|\lambda||a|_We^{C|\lambda||a|_W} \qquad(|h|_W\le 2|a|_W+1),
$$
which follows from $\|\partial_a\psi\|_\infty\le|a|_W$, hence
$|\partial_h\partial_a\psi|\le 2|a|_W$, combined with the mean-value inequality
$|e^{-\lambda u}-e^{-\lambda v}|\le|\lambda||u-v|e^{|\lambda|\max(|u|,|v|)}$,
yield
$$
\|\partial_{b^{-1}}L(a^{-1})\bigl(m_a(\lambda)f_2\bigr)\|_2 \le C|b|_W(1+|a|_W)e^{C|\lambda||a|_W}\|\nabla_W f_2\|_2 +C|\lambda||a|_W(1+|a|_W)|b|_We^{C|\lambda||a|_W}\|f_2\|_2.
$$
Combining with \eqref{eq:app-Md-bound} for $M_c(\lambda)$ and Young's inequality,
$st\le \eps s^2+(4\eps)^{-1}t^2$, the third term is
bounded by
$$
\eps|b|_W^2(1+|a|_W)^2\|\nabla_W f_2\|_2^2 +C(1+\eps^{-1})\lambda^2(1+|a|_W+|b|_W+|c|_W)^4 e^{C|\lambda|(|a|_W+|c|_W)}(\|f_1\|_2^2+\|f_2\|_2^2).
$$
The fourth term is bounded trivially by the product of three sup-norms:
$$
C\|M_a(\lambda)+1\|_\infty\|M_b(\lambda)\|_\infty\|M_c(\lambda)\|_\infty\|f_1\|_2\|f_2\|_2 \le C\lambda^2|b|_W|c|_We^{C|\lambda|(|a|_W+|b|_W+|c|_W)}(\|f_1\|_2^2+\|f_2\|_2^2).
$$

\smallskip\noindent
Observe that
$$
\langle M_a(\lambda)L(a)\partial_b\partial_c f_1,f_2\rangle =\langle\partial_c f_1,\partial_{b^{-1}}L(a^{-1})(M_a(\lambda)^\ast f_2)\rangle.
$$
Since $\partial_{b^{-1}}$ is bounded on $\ell^2$ with operator norm $\le 2$ and $L(a^{-1})$ is an
isometry, while $M_a(\lambda)^\ast$ is a bounded multiplier,
$$
\bigl\|\partial_{b^{-1}}L(a^{-1})(M_a(\lambda)^\ast f_2)\bigr\|_2 \le 2\|M_a(\lambda)\|_\infty\|f_2\|_2.
$$
Combining with \eqref{eq:app-tele-bound} and Young's inequality,
\begin{align*}
|\langle M_a(\lambda)L(a)\partial_b\partial_c f_1,f_2\rangle|
&\le 2|c|_W\|M_a(\lambda)\|_\infty\|\nabla_W f_1\|_2\|f_2\|_2\\
&\le \eps|c|_W^2\|\nabla_W f_1\|_2^2
+C\eps^{-1}\lambda^2|a|_W^2 e^{C|\lambda||a|_W}\|f_2\|_2^2.
\end{align*}

\smallskip\noindent
Assembling the four term bounds for a single $A=L(a)\partial_b\partial_c$ yields
\begin{multline*}
\bigl|\langle(e^{\lambda\psi}A e^{-\lambda\psi}-A)f_1,f_2\rangle\bigr|
\le \eps(1+|a|_W+|b|_W+|c|_W)^4\bigl(\|\nabla_W f_1\|_2^2+\|\nabla_W f_2\|_2^2\bigr)\\
+C(1+\eps^{-1})\lambda^2(1+|a|_W+|b|_W+|c|_W)^4e^{C|\lambda|(|a|_W+|b|_W+|c|_W)}
(\|f_1\|_2^2+\|f_2\|_2^2).
\end{multline*}
Applying this term-by-term to the expansion
\eqref{eq:app-controlled-D2-expansion} of $\mathcal R_g$ from
\Cref{lem:quantitative-D2-app}, and using \eqref{eq:app-controlled-D2-linear-length} together with \eqref{eq:app-controlled-D2-size}, we obtain, for some fixed integer $m_1$,
\begin{multline*}
\bigl|\langle(e^{\lambda\psi}\mathcal R_g e^{-\lambda\psi}-\mathcal R_g)f_1,f_2\rangle\bigr|
\le \eps(1+|g|_W)^{m_1}\bigl(\|\nabla_W f_1\|_2^2+\|\nabla_W f_2\|_2^2\bigr)\\
+C(1+\eps^{-1})\lambda^2(1+|g|_W)^{m_1}e^{C|\lambda||g|_W}
(\|f_1\|_2^2+\|f_2\|_2^2).
\end{multline*}
Choose $\lambda_0\in(0,\beta/(2C)]$ small enough that the uniform exponential moment
$$
\sup_{R\ge R_0}\sum_{g\in G}(1+|g|_W)^{m_1}e^{C\lambda_0|g|_W}K_R(g)<\infty.
$$
Choose $M_1\ge 1$ so that $\sup_{R\ge R_0}\sum_{g\in G}(1+|g|_W)^{m_1}e^{C\lambda_0|g|_W}K_R(g)<M_1$. Averaging against $K_R$ and replacing $\eps$ by $\eps/M_1$ yields
\eqref{eq:app-uniform-quad-perturb}.

Finally, combining \eqref{eq:app-uniform-quad-perturb} with
\eqref{eq:app-family-form-comparison} and choosing $\eps=c_0/4$ gives
\eqref{eq:app-uniform-quad-perturb-form}.
\end{proof}

\begin{lemma}[A quantitative criterion]\label{lem:quantitative-nevanlinna-ritt-app}
Let $R$ be a bounded operator on a complex Banach space. Suppose that, for
some constants $M,K>0$,
$$
\sup_{q\ge0}\|R^q\|\le M, \qquad \sup_{|z|>1}|z-1|\|(zI-R)^{-1}\|\le K.
$$
Then there is a constant $C_{\mathrm{NR}}=C_{\mathrm{NR}}(M,K)>0$ such that
$$
\|(I-R)R^q\|\le C_{\mathrm{NR}}q^{-1}\qquad(q\in\NN).
$$
\end{lemma}

\begin{proof}
By \cite[Theorem 1]{Lyubich1999}, we have
\begin{equation}\label{eq:classical-ritt-criterion-app}
\sup_{q\in\NN}q\|(I-R)R^q\|<\infty .
\end{equation}
It remains only to justify that the constant in the conclusion may be chosen
as a function of the two constants $M$ and $K$.  Suppose this is not true.  Then, for every $j\in\NN$, there
exist complex Banach spaces $X_j$, bounded operators $R_j$ on $X_j$, and
integers $q_j\ge1$, such that
$$
\sup_{q\ge0}\|R_j^q\|\le M, \qquad \sup_{|z|>1}|z-1|\|(zI-R_j)^{-1}\|\le K,
$$
but
$$
q_j\|(I-R_j)R_j^{q_j}\|\ge j .
$$
Let $X=(\bigoplus_j X_j)_{\ell^\infty}$ and let
$R=\bigoplus_j R_j$.  Then $R$ is bounded, $\sup_{q\ge0}\|R^q\|\le M$, and,
for $|z|>1$,
$$
(zI-R)^{-1}=\bigoplus_j (zI-R_j)^{-1}, \qquad |z-1|\|(zI-R)^{-1}\|\le K .
$$
Hence, by \cite[Theorem 1]{Lyubich1999},
$\sup_{q\in\NN}q\|(I-R)R^q\|<\infty$.  But for $q=q_j$, the norm of the
$j$-th coordinate gives
$$
q_j\|(I-R)R^{q_j}\|\ge q_j\|(I-R_j)R_j^{q_j}\|\ge j,
$$
a contradiction.  Therefore a finite constant
$C_{\mathrm{NR}}=C_{\mathrm{NR}}(M,K)$ exists, and the estimate follows.
\end{proof}

\begin{lemma}[Shifted-sectorial resolvent bound]\label{lem:blunck-replacement-app}
Let $S$ be a bounded operator on a Hilbert space and let $\sigma\in[0,\sigma_0]$.
For $\delta\in(\pi/2,\pi)$ put
$$
\Lambda_\delta:=\{z\in\CC\setminus\{0\}: |\arg z|<\delta\}.
$$
Put $A:=I-S+\sigma I$ and $\widetilde S:=(1+\sigma)^{-1}S$. Suppose that, for constants
$M>0$, $\gamma<1$, $C_{\rm res}>0$, and $C_{\rm ap}>0$,
\begin{equation}\label{eq:shifted-ritt-power-hyp}
\| S^q\|\le M e^{\gamma\sigma q}\qquad(q\ge0),
\end{equation}
\begin{equation}\label{eq:shifted-ritt-res-hyp}
\| \alpha(A+\alpha I)^{-1}\|\le C_{\rm res}
\qquad(\alpha\in\Lambda_\delta),
\end{equation}
and, with
$$
\Omega_\delta:=\{z\in\CC: |z|>1\text{ and } z-1\notin\Lambda_\delta\},
$$
\begin{equation}\label{eq:shifted-ritt-away-hyp}
\sup_{z\in\Omega_\delta}|z-1|
\|(zI-\widetilde S)^{-1}\|\le C_{\rm ap}.
\end{equation}
Then, after decreasing $\sigma_0$ if necessary, there are constants $C,\Omega>0$, depending only on
$M,\gamma,C_{\rm res},C_{\rm ap},\delta,\sigma_0$, such that
\begin{equation}\label{eq:shifted-ritt-conclusion}
\lVert (I-S)S^q\rVert\le C(q^{-1}+\sigma)e^{\Omega\sigma q}
\qquad(q\in\NN).
\end{equation}
\end{lemma}

\begin{proof}
We first prove that $\widetilde S$ is a Ritt operator with uniform constants. Since
$$
I-\widetilde S=(1+\sigma)^{-1}(I-S+\sigma I)=(1+\sigma)^{-1}A,
$$
we have, for every $z$ with $|z|>1$ and $z-1\in\Lambda_\delta$,
$$
zI-\widetilde S=(1+\sigma)^{-1}\bigl(A+(1+\sigma)(z-1)I\bigr).
$$
Thus, setting $\alpha=(1+\sigma)(z-1)$ and using
\eqref{eq:shifted-ritt-res-hyp},
$$
|z-1|\lVert(zI-\widetilde S)^{-1}\rVert =|\alpha|\lVert(A+\alpha I)^{-1}\rVert \le C_{\rm res}.
$$
For $z\in\Omega_\delta$ the same bound, with $C_{\rm ap}$ in place of
$C_{\rm res}$, is exactly \eqref{eq:shifted-ritt-away-hyp}. Hence
\begin{equation}\label{eq:shifted-ritt-full-resolvent}
\sup_{|z|>1}|z-1|\lVert(zI-\widetilde S)^{-1}\rVert
\le \max\{C_{\rm res},C_{\rm ap}\}.
\end{equation}
Moreover,
$$
\lVert\widetilde S^q\rVert\le M(1+\sigma)^{-q}e^{\gamma\sigma q}.
$$
Since $\gamma<1$, decreasing $\sigma_0$ gives
$\log(1+\sigma)\ge ((1+\gamma)/2)\sigma$ for $0\le\sigma\le\sigma_0$, and hence
$\sup_q\lVert\widetilde S^q\rVert\le M$.
Applying \Cref{lem:quantitative-nevanlinna-ritt-app} to
$R=\widetilde S$, using the power bound above and \eqref{eq:shifted-ritt-full-resolvent}, gives
$$
\lVert(I-\widetilde S)\widetilde S^q\rVert\le C q^{-1} \qquad(q\in\NN),
$$
where $C$ depends only on the displayed constants. Finally,
$$
(I-S)S^q =(1+\sigma)^q\bigl((1+\sigma)(I-\widetilde S)\widetilde S^q-\sigma\widetilde S^q\bigr),
$$
and therefore
\begin{align*}
\lVert(I-S)S^q\rVert
&\le (1+\sigma)^q\bigl((1+\sigma)Cq^{-1}+\sigma M\bigr)  \\
&\le C'(q^{-1}+\sigma)e^{\Omega\sigma q}.
\end{align*}
This is \eqref{eq:shifted-ritt-conclusion}.
\end{proof}

\begin{proposition}[Uniform Davies weighted bounds]\label{prop:uniform-davies-family-app}
Under the hypotheses of \Cref{lem:uniform-perturb-app}, there exist constants $\lambda_0>0$, $C>0$, and $\omega>0$, depending only on $(G,W,\eta,\beta,\mathfrak M_\beta)$, such that for every $R\ge R_0$, every $\psi:G\to\RR$ satisfying $\|\partial_w\psi\|_\infty\le1$ for all $w\in W$, every $|\lambda|\le\lambda_0$, every $q\in\NN$, and every $h\in W$,
\begin{equation}\label{eq:app-uniform-weighted-2to2}
\|e^{\lambda\psi}T_R^q e^{-\lambda\psi}\|_{2\to2}
\le e^{\omega\lambda^2 q},
\end{equation}
and
\begin{equation}\label{eq:app-uniform-weighted-diff-2to2}
\|e^{\lambda\psi}\partial_h T_R^q e^{-\lambda\psi}\|_{2\to2}
\le C q^{-1/2}e^{\omega\lambda^2 q}.
\end{equation}
\end{proposition}

\begin{proof}
This is the family-uniform analogue of the weighted spatial regularity estimate
in \cite[Theorem 1.11]{Dungey2008}. Write
$$
S:=T_{R,\lambda,\psi}=e^{\lambda\psi}T_Re^{-\lambda\psi}, \qquad T:=T_R.
$$
The hypothesis $\|\partial_w\psi\|_\infty\le1$ for $w\in W$ implies $\|\partial_g\psi\|_\infty\le |g|_W$ for every $g\in G$, by telescoping along a $W$-geodesic word for $g$.
Thus, after decreasing $\lambda_0$ if necessary so that $\lambda_0\le\beta/2$,
\begin{align}
\|S-T\|_{2\to2}
&\le \sum_{g\in G}K_R(g)\|e^{-\lambda\partial_g\psi}-1\|_\infty \notag\\
&\le |\lambda|\sum_{g\in G}K_R(g)|g|_We^{|\lambda||g|_W} \notag\\
&\le C_{\beta,\mathfrak M}|\lambda|,
\label{eq:app-one-step-S-minus-T}
\end{align}
where the last constant depends only on $\beta$ and $\mathfrak M_\beta$.  Indeed,
for $0\le t<\infty$ and $|\lambda|\le\beta/2$, the quantity
$t e^{|\lambda|t}e^{-\beta t}$ is uniformly bounded.

We shall also use
\begin{equation}\label{eq:app-grad-T-uniform}
\|\nabla_W Tf\|_2\le C\|\nabla_W f\|_2.
\end{equation}
Indeed, for $w\in W$,
$$
\partial_w Tf =\sum_{g\in G}K_R(g)\bigl(L(wg)-L(g)\bigr)f =\sum_{g\in G}K_R(g)L(g)\partial_{g^{-1}wg}f,
$$
because $L(g)\partial_{g^{-1}wg}=L(wg)-L(g)$. Hence \eqref{eq:app-tele-bound}
gives
$$
\|\partial_w Tf\|_2 \le \sum_{g\in G}K_R(g)|g^{-1}wg|_W\|\nabla_W f\|_2 \le C\|\nabla_W f\|_2,
$$
uniformly in $w\in W$, since $|g^{-1}wg|_W\le 2|g|_W+1$ and the first
moments of the family $K_R$ are uniformly bounded by
\eqref{eq:app-standing-hyp}.

\smallskip\noindent
Set
$$
Q_R(f):=\|f\|_2^2-\|Tf\|_2^2, \qquad Q_{R,\lambda,\psi}(f):=\|f\|_2^2-\|Sf\|_2^2.
$$
Since $T$ is convolution by a probability measure, Young's convolution inequality gives
$\|Tf\|_2\le\|f\|_2$, so $Q_R(f)\ge0$. Let $\widehat K_R:=K_R^**K_R$, where
$K_R^*(g):=K_R(g^{-1})$. Then $T^*T$ is convolution by $\widehat K_R$, and $\widehat K_R(w)\ge K_R(e)K_R(w)\ge
\eta^2$ for every $w\in W$. Therefore
\begin{equation}\label{eq:app-QR-lower}
Q_R(f)=\tfrac12\sum_{g\in G}\widehat K_R(g)\|\partial_g f\|_2^2
\ge c_Q\|\nabla_W f\|_2^2,
\qquad c_Q:=\eta^2/2.
\end{equation}
Moreover,
\begin{align*}
Q_R(f)-Q_{R,\lambda,\psi}(f)
&=\|Sf\|_2^2-\|Tf\|_2^2 \\
&=2\Re\langle(S-T)f,Tf\rangle+\|(S-T)f\|_2^2.
\end{align*}
Applying \eqref{eq:app-uniform-quad-perturb} with $f_1=f$, $f_2=Tf$, together
with \eqref{eq:app-grad-T-uniform}, $\|Tf\|_2\le\|f\|_2$, and
\eqref{eq:app-one-step-S-minus-T}, we obtain, after absorbing the constant
factor from \eqref{eq:app-grad-T-uniform} into $\eps$,
\begin{equation}\label{eq:app-Q-perturb}
|Q_R(f)-Q_{R,\lambda,\psi}(f)|
\le \eps\|\nabla_W f\|_2^2
+C'(1+\eps^{-1})\lambda^2\|f\|_2^2.
\end{equation}
where $C'>0$ depends only on $(G,W,\eta,\beta,\mathfrak M_\beta)$. Choosing $\eps=c_Q/2$ in \eqref{eq:app-Q-perturb} and using
\eqref{eq:app-QR-lower},
$$
Q_{R,\lambda,\psi}(f) \ge \tfrac{c_Q}{2}\|\nabla_W f\|_2^2-C'\left( 1+\frac{2}{c_Q}\right)\lambda^2\|f\|_2^2,
$$
so that $\|Sf\|_2^2\le \left(1+C'\left( 1+\frac{2}{c_Q}\right)\lambda^2\right)\|f\|_2^2$. Hence,
\begin{equation}\label{eq:app-S-power-bound}
\|S^q\|_{2\to2}\le \|S\|_{2\to2}^q\le e^{\omega_2\lambda^2 q}\qquad(q\in\NN),
\end{equation}
for a constant $\omega_2>0$ depending only on $(G,W,\eta,\beta,\mathfrak M_\beta)$. After the
final enlargement of $\omega_2$ below, this gives
\eqref{eq:app-uniform-weighted-2to2}.

\smallskip\noindent
Choose $B>\max\{ C_{\mathrm{pert}}+1, C_{\mathrm{pert}}+c_0/2\}$ large enough that $\gamma:=\omega_2/B<1$, and put
$$
A:=I-S+B\lambda^2 I.
$$
By \eqref{eq:app-uniform-quad-perturb-form},
\begin{equation}\label{eq:app-shifted-sector-real}
\Re\langle Af,f\rangle
\ge \tfrac{c_0}{2}\|\nabla_W f\|_2^2+(B-C_{\mathrm{pert}})\lambda^2\|f\|_2^2
\ge \frac{c_0}{2}\bigl(\|\nabla_W f\|_2^2+\lambda^2\|f\|_2^2\bigr).
\end{equation}
On the other hand, the corresponding upper bound follows directly from the
sectorial estimate for $I-T$ and the quadratic perturbation estimate.  Since
$$
A=(I-T)-(S-T)+B\lambda^2 I,
$$
we have, for every $f\in\ell^2(G)$,
\begin{align*}
|\langle Af,f\rangle|
&\le |\langle(I-T)f,f\rangle|+|\langle(S-T)f,f\rangle|
   +B\lambda^2\|f\|_2^2 .
\end{align*}
The first term is controlled by \eqref{eq:app-family-sectorial} and the upper
bound in \eqref{eq:app-family-form-comparison}:
$$
|\langle(I-T)f,f\rangle| \le \frac{c_1}{c_0}\Re\langle(I-T)f,f\rangle \le \frac{c_1^2}{c_0}\|\nabla_W f\|_2^2 .
$$
For the second term, apply \eqref{eq:app-uniform-quad-perturb} with
$f_1=f_2=f$ and then fix, say, $\eps=1$:
$$
|\langle(S-T)f,f\rangle| \le 2\|\nabla_W f\|_2^2+4C_{\mathrm{pert}}\lambda^2\|f\|_2^2 .
$$
Consequently
\begin{equation}\label{eq:app-shifted-sector-upper}
|\langle Af,f\rangle|
\le C_A\bigl(\|\nabla_W f\|_2^2+\lambda^2\|f\|_2^2\bigr),
\end{equation}
where one may take
$$
C_A:=\max\left\{\frac{c_1^2}{c_0}+2,4C_{\mathrm{pert}}+B\right\}.
$$
Combining \eqref{eq:app-shifted-sector-upper} with
\eqref{eq:app-shifted-sector-real} gives
\begin{equation}\label{eq:app-shifted-sectorial-form}
|\langle Af,f\rangle|\le C_{\mathrm{sec}}\Re\langle Af,f\rangle
\qquad(f\in\ell^2(G)),
\end{equation}
with $C_{\mathrm{sec}}:=2C_A/c_0$.
We now derive \eqref{eq:shifted-ritt-res-hyp} in \Cref{lem:blunck-replacement-app} directly from \eqref{eq:app-shifted-sectorial-form}. Let
$W(A):=\{\langle Af,f\rangle:f\in\ell^2(G),\|f\|_2=1\}$ denote the numerical range of $A$.
By \eqref{eq:app-shifted-sectorial-form}, $W(A)$ is contained in a closed sector
$\Sigma_\theta:=\{0\}\cup\{z\in\CC\setminus\{0\}:|\arg z|\le\theta\}$ for some
$\theta\in[0,\pi/2)$ depending only on $C_{\mathrm{sec}}$ (one may take $\theta\in[0,\pi/2)$ such that
$\tan\theta=\sqrt{C_{\mathrm{sec}}^2-1}$). Choose $\delta\in(\pi/2,\pi-\theta)$
and let $\Lambda_\delta:=\{z\in\CC\setminus\{0\}:|\arg z|<\delta\}$. Elementary
planar geometry then gives a constant $c_{\theta,\delta}>0$ such that
$$
|z+\alpha|\ge c_{\theta,\delta}|\alpha| \qquad(z\in\Sigma_\theta, \alpha\in\Lambda_\delta).
$$
Hence, for every $u\in\ell^2(G)$ and $\alpha\in\Lambda_\delta$ we have
$$
\|(A+\alpha I)u\|_2\|u\|_2 \ge|\langle(A+\alpha I) u,u\rangle| \ge c_{\theta,\delta}|\alpha|\|u\|_2^2,
$$
which gives the lower bound $\|(A+\alpha I)u\|_2\ge c_{\theta,\delta}|\alpha|\|u\|_2$
for all $u\in\ell^2(G)$ and $\alpha\in\Lambda_\delta$. The same estimate applied to $A^*$ (whose numerical
range is the complex conjugate of $W(A)$) shows that $A+\alpha I$ is bounded
below and has dense range, hence is invertible. Consequently there is a
constant $C_{\mathrm{res}}>0$, depending only on $C_{\mathrm{sec}}$, such that
\begin{equation}\label{eq:app-shifted-resolvent}
\|\alpha(A+\alpha I)^{-1}\|_{2\to2}\le C_{\mathrm{res}}
\qquad(\alpha\in\Lambda_\delta).
\end{equation}
Equivalently,
\begin{equation}\label{eq:app-blunck-resolvent-hypothesis}
\|\alpha(I-S+B\lambda^2 I+\alpha I)^{-1}\|_{2\to2}\le C_{\mathrm{res}}
\qquad(\alpha\in\Lambda_\delta).
\end{equation}
Put $\sigma:=B\lambda^2$ and $\widetilde S:=(1+\sigma)^{-1}S$. By
\eqref{eq:app-S-power-bound},
$$
\|S^q\|_{2\to2}\le e^{\gamma\sigma q}\qquad(q\ge0).
$$
It remains to verify \eqref{eq:shifted-ritt-away-hyp} in
\Cref{lem:blunck-replacement-app}. Let
$$
\Omega_\delta:=\{z\in\CC: |z|>1\text{ and }z-1\notin\Lambda_\delta\}.
$$
For $p_R:=K_R(e)$ we have $p_R\ge\eta$.  Put
$$
D_{p_R}:=p_R+(1-p_R)\overline{\mathbb D}.
$$
If $p_R=1$, then
$T=I$ and $D_{p_R}=\{1\}$, so
$$
\|(zI-T)^{-1}\|_{2\to2}=|z-1|^{-1}=\operatorname{dist}(z,D_{p_R})^{-1} \qquad(z\ne1).
$$
Assume now that $p_R<1$ and write
$$
T=p_R I+(1-p_R)P_R,
$$
where
$$
P_R:=\frac{T-p_RI}{1-p_R}
$$
is convolution by the probability measure
$(K_R-p_R\delta_e)/(1-p_R)$; hence $\|P_R\|_{2\to2}\le1$ by Young's
inequality.  If $z\notin D_{p_R}$, then $|z-p_R|>1-p_R$, and therefore
$$
zI-T=(z-p_R)\left(I-\frac{1-p_R}{z-p_R}P_R\right).
$$
The last factor is invertible and we have
\begin{align*}
\|(zI-T)^{-1}\|_{2\to2}
&\le |z-p_R|^{-1}\sum_{n=0}^{\infty}
   \left(\frac{1-p_R}{|z-p_R|}\right)^n  \\
&=\bigl(|z-p_R|-(1-p_R)\bigr)^{-1}
 =\operatorname{dist}(z,D_{p_R})^{-1}.
\end{align*}
Since $p_R\ge\eta$, the disk $D_{p_R}$ is contained in the larger disk
$$
D_\eta:=\eta+(1-\eta)\overline{\mathbb D}.
$$
Indeed, the distance between the centres is $p_R-\eta$, which is exactly the
excess of the radius of $D_\eta$ over the radius of $D_{p_R}$.
Consequently,
\begin{equation}\label{eq:app-T-away-resolvent}
\|(zI-T)^{-1}\|_{2\to2}\le \operatorname{dist}(z,D_\eta)^{-1}
\qquad(z\notin D_\eta).
\end{equation}
Write $z=1+re^{i\theta}$ with $r=|z-1|$ and $\theta\in[-\pi,\pi]$. If
$z\in\Omega_\delta$, then $|\theta|\ge\delta$, so
$t:=-\cos\theta\ge t_\delta:=-\cos\delta>0$.  The condition $|z|>1$ gives
$r^2+2r\cos\theta>0$, hence $r>2t\ge2t_\delta$.  Thus
$|z-1|\ge r_{\rm ap}$ with $r_{\rm ap}:=2t_\delta$.

Put $a:=1-\eta$. Since $D_\eta$ is the disk with centre $\eta$ and radius $a$,
$$
\operatorname{dist}(z,D_\eta)=|z-\eta|-a=|re^{i\theta}+a|-a.
$$
Indeed, $r>2t\ge 2at$ implies $|re^{i\theta}+a|>a$, so $z$ is outside the
closed disk $D_\eta$.  For $r\ge4$ this is at least $r-2a\ge r/2$. On the
compact set
$$
\{(r,\theta): |\theta|\in[\delta,\pi], 2(-\cos\theta)\le r\le4\}
$$
the continuous function
$$
(r,\theta)\mapsto \frac{|re^{i\theta}+a|-a}{r}
$$
has a positive minimum, which is attained at $r=2(-\cos\theta)$. Combining this lower bound
with the estimate for $r\ge4$ gives a constant $c_{\rm ap}>0$, depending only
on $\eta$ and $\delta$, such that
\begin{equation}\label{eq:app-away-geometry}
\operatorname{dist}(z,D_\eta)\ge c_{\rm ap}|z-1|,
\qquad |z-1|\ge r_{\rm ap}
\qquad(z\in\Omega_\delta).
\end{equation}
Combining \eqref{eq:app-T-away-resolvent} and \eqref{eq:app-away-geometry},
$$
|z-1|\|(zI-T)^{-1}\|_{2\to2}\le c_{\rm ap}^{-1} \qquad(z\in\Omega_\delta).
$$
Moreover, by \eqref{eq:app-one-step-S-minus-T} and $\|T\|_{2\to2}\le1$,
$$
\|\widetilde S-T\|_{2\to2} \le \frac{1}{1+\sigma}\|S-T\|_{2\to2}   +\frac{\sigma}{1+\sigma}\|T\|_{2\to2} \le C|\lambda|+B\lambda^2.
$$
Using \eqref{eq:app-T-away-resolvent}-\eqref{eq:app-away-geometry}, $\|(zI-T)^{-1}\|_{2\to2}$ is uniformly
bounded on $\Omega_\delta$ by $(c_{\rm ap}r_{\rm ap})^{-1}$.  After decreasing
$\lambda_0$ so that
$$
C|\lambda|+B\lambda^2\le \frac12 c_{\rm ap}r_{\rm ap} \qquad(|\lambda|\le\lambda_0),
$$
and writing $(zI-\widetilde S)^{-1}=\left[I-(zI-T)^{-1}(\widetilde S-T)\right]^{-1}(zI-T)^{-1}$ we get
\begin{equation}\label{eq:app-S-away-resolvent}
\sup_{z\in\Omega_\delta}|z-1|
\|(zI-\widetilde S)^{-1}\|_{2\to2}\le C_{\rm ap}',
\end{equation}
with $C_{\rm ap}'$ depending only on $(G,W,\eta,\beta,\mathfrak M_\beta)$.

The estimate \eqref{eq:app-blunck-resolvent-hypothesis} is precisely
\eqref{eq:shifted-ritt-res-hyp} for $A=I-S+\sigma I$, and
\eqref{eq:app-S-away-resolvent} is \eqref{eq:shifted-ritt-away-hyp}. After decreasing
$\lambda_0$ so that $\sigma\le\sigma_0$,
\Cref{lem:blunck-replacement-app} gives
\begin{equation}\label{eq:app-shifted-blunck-output}
\|(I-S)S^q\|_{2\to2}
\le C(q^{-1}+\lambda^2)e^{\omega_3\lambda^2 q}
\qquad(q\ge1),
\end{equation}
with constants uniform in $R$, $\psi$, $\lambda$, and $q$. 

\smallskip\noindent
Apply \eqref{eq:app-uniform-quad-perturb-form} to $u=S^qf$. Using
\eqref{eq:app-shifted-blunck-output}, \eqref{eq:app-S-power-bound}, and
Cauchy--Schwarz,
\begin{align*}
\tfrac{c_0}{2}\|\nabla_W S^qf\|_2^2
&\le \Re\bigl((I-S)S^qf,S^qf\bigr)+C_{\mathrm{pert}}\lambda^2\|S^qf\|_2^2\\
&\le \|(I-S)S^qf\|_2\|S^qf\|_2+C_{\mathrm{pert}}\lambda^2\|S^qf\|_2^2\\
&\le C(q^{-1}+\lambda^2)e^{\omega_4\lambda^2 q}\|f\|_2^2,
\end{align*}
for a constant $\omega_4$ depending only on the standing data. Hence,
using $\sqrt{a+b}\le\sqrt a+\sqrt b$,
\begin{equation}\label{eq:app-grad-Sq-bound}
\|\nabla_W S^q\|_{2\to2}
\le C(q^{-1/2}+|\lambda|)e^{\omega_4\lambda^2 q}.
\end{equation}
Next,
$$
e^{\lambda\psi}\partial_h T_R^q e^{-\lambda\psi} =e^{\lambda\psi}\partial_h e^{-\lambda\psi}S^q =\bigl(\partial_h+M_h(\lambda)L(h)\bigr)S^q, \qquad M_h(\lambda):=e^{-\lambda\partial_h\psi}-1.
$$
For $h\in W$ we have $\|M_h(\lambda)\|_\infty\le C|\lambda|$. Combining this
with \eqref{eq:app-grad-Sq-bound} and \eqref{eq:app-S-power-bound} gives
$$
\|e^{\lambda\psi}\partial_h T_R^q e^{-\lambda\psi}\|_{2\to2} \le C(q^{-1/2}+|\lambda|)e^{\omega_5\lambda^2 q}.
$$
Finally, for any fixed $\delta_0>0$ and all $q\in\NN$, the elementary bound
$t\le\sqrt{e/(2\delta_0)}e^{\delta_0 t^2}$ applied to $t=|\lambda|\sqrt q$
gives
$$
|\lambda|\le C_{\delta_0}q^{-1/2}e^{\delta_0\lambda^2 q},
$$
so the $|\lambda|$ term is absorbed into the $q^{-1/2}$ term by enlarging the
final exponent. Taking $\omega\ge\max\{\omega_2,\omega_5+\delta_0\}$ proves
\eqref{eq:app-uniform-weighted-2to2} and
\eqref{eq:app-uniform-weighted-diff-2to2}.
\end{proof}

Define $\rho:G\to \NN\cup \{0\}$ by $\rho(g)=|g|_W$, $g\in G$.

\begin{lemma}[Weighted $L^2$ bound]\label{lem:weighted-L2-app}
There exists $C>0$, depending only on
$(G,W,\eta,\beta,\mathfrak M_\beta)$, such that for every $R\ge R_0$ and every $|\lambda|\le \beta/2$,
\begin{equation}\label{eq:app-one-step-weighted-L2}
\|e^{\lambda\rho}K_R\|_2\le C.
\end{equation}
\end{lemma}

\begin{proof}
By \eqref{eq:app-standing-hyp}, the quantity $\mathfrak M_\beta$ is finite. Since $K_R$ takes values in $[0,1]$,
for every $R\ge R_0$ and every $|\lambda|\le\beta/2$,
$$
\|e^{\lambda\rho}K_R\|_2^2 =\sum_{g\in G}e^{2\lambda|g|_W}K_R(g)^2 \le \sum_{g\in G}e^{\beta|g|_W}K_R(g)\le \mathfrak M_\beta,
$$
and the lemma follows with $C:=\sqrt{\mathfrak M_\beta}$.
\end{proof}

For $\lambda\in\RR$, let $U_\lambda$ denote multiplication by $e^{\lambda\rho}$.
This operator is usually unbounded on $\ell^2(G)$.  In the averaging lemma below we use it only on finitely supported vectors.

\begin{lemma}[Weighted finite-support convolution bounds]\label{lem:weighted-finite-support-app}
Let $f:G\to\CC$ be finitely supported, and write
$(L(f)h)(x):=\sum_g f(g)h(g^{-1}x)=(f*h)(x)$ for the left convolution operator.
For $r\ge1$, let
$$
B_r:=\{g\in G:|g|_W\le r\}, \qquad \chi_r:=|B_r|^{-1}\mathbf 1_{B_r}.
$$
Then $U_\lambda f\in\ell^2(G)$ for every $\lambda\in\RR$, and the operator
$U_{-\lambda}L(f)U_\lambda$, initially defined on finitely supported functions,
extends boundedly to $\ell^2(G)$ with
\begin{equation}\label{eq:app-weighted-conv-bounded}
\|U_{-\lambda}L(f)U_\lambda\|_{2\to2}
\le \sum_{g\in G}|f(g)|e^{|\lambda||g|_W}<\infty .
\end{equation}
Moreover,
\begin{equation}\label{eq:app-weighted-averaging}
\|U_\lambda f\|_2
\le \sup_{|g|_W\le r}\|U_\lambda(I-L(g))f\|_2
+e^{|\lambda|r}|B_r|^{-1/2}\|U_{-\lambda}L(f)U_\lambda\|_{2\to2}.
\end{equation}
\end{lemma}

\begin{proof}
For $a\in G$ and finitely supported $h$,
$$
(U_{-\lambda}L(a)U_\lambda h)(x) =e^{-\lambda|x|_W}e^{\lambda|a^{-1}x|_W}h(a^{-1}x).
$$
Since $||a^{-1}x|_W-|x|_W|\le|a|_W$, the multiplier in front of
$h(a^{-1}x)$ is bounded in absolute value by $e^{|\lambda||a|_W}$.  Thus
$U_{-\lambda}L(a)U_\lambda$ is bounded on $\ell^2(G)$ with norm at most
$e^{|\lambda||a|_W}$.  Summing over the finite support of $f$ proves
\eqref{eq:app-weighted-conv-bounded}.  The same argument applies to
$U_\lambda L(f^\star)U_{-\lambda}$, where $f^\star(x):=\overline{f(x^{-1})}$,
and on finitely supported vectors one has
\begin{equation}\label{eq:app-weighted-adjoint-identity}
\bigl(U_\lambda L(f^\star)U_{-\lambda}\bigr)^*=U_{-\lambda}L(f)U_\lambda .
\end{equation}
Indeed, this follows term-by-term from
$(U_\lambda L(a)U_{-\lambda})^*=U_{-\lambda}L(a^{-1})U_\lambda$.

Since $B_r$ is finite and $\sum_g\chi_r(g)=1$,
$$
f-\chi_r*f=\sum_{g\in B_r}\chi_r(g)(I-L(g))f.
$$
Hence
$$
\|U_\lambda(f-\chi_r*f)\|_2 \le \sup_{|g|_W\le r}\|U_\lambda(I-L(g))f\|_2 .
$$
It remains to estimate $\|U_\lambda(\chi_r*f)\|_2$.  Since $W$ is symmetric,
$|x^{-1}|_W=|x|_W$, so $\|U_\lambda a\|_2=\|U_\lambda a^\star\|_2$ for
finitely supported $a$. Also $B_r$ is symmetric, hence $\chi_r^\star=\chi_r$, and
therefore
$$
(\chi_r*f)^\star=f^\star*\chi_r.
$$
Consequently,
$$
\|U_\lambda(\chi_r*f)\|_2 =\|U_\lambda(f^\star*\chi_r)\|_2 =\|U_\lambda L(f^\star)U_{-\lambda}(U_\lambda\chi_r)\|_2.
$$
By \eqref{eq:app-weighted-adjoint-identity},
$$
\|U_\lambda L(f^\star)U_{-\lambda}\|_{2\to2} =\|U_{-\lambda}L(f)U_\lambda\|_{2\to2}.
$$
Finally,
$$
\|U_\lambda\chi_r\|_2^2 =|B_r|^{-2}\sum_{|g|_W\le r}e^{2\lambda|g|_W} \le e^{2|\lambda|r}|B_r|^{-1},
$$
so $\|U_\lambda\chi_r\|_2\le e^{|\lambda|r}|B_r|^{-1/2}$.  Combining the
last estimates proves \eqref{eq:app-weighted-averaging}.
\end{proof}

\begin{proposition}[Uniform weighted $L^2$ kernel bound]\label{prop:uniform-weighted-L2-bound-app}
There exist constants $C>0$ and
$\omega>0$, depending only on $(G,W,\eta,\beta,\mathfrak M_\beta)$, such that for every $R\ge R_0$,
every $q\in\NN$, and every $|\lambda|\le\min\{\lambda_0,\beta/2\}$,
\begin{equation}\label{eq:app-uniform-weighted-kernel-L2-transfer}
\|U_\lambda K_R^{(q)}\|_2
\le Cq^{-D/4}e^{\omega\lambda^2 q}.
\end{equation}
\end{proposition}

\begin{proof}
Set $\lambda_2:=\min\{\lambda_0,\beta/2\}$.  Increasing constants if necessary,
we fix $C'>0$ and $\omega_0>0$ such that, for all $R\ge R_0$,
$|\lambda|\le\lambda_2$, and $m\ge0$,
\begin{equation}\label{eq:app-two-sided-weighted-2to2-transfer}
\|U_\lambda T_R^mU_{-\lambda}\|_{2\to2}
\le e^{\omega_0\lambda^2m},
\qquad
\|U_{-\lambda}T_R^mU_\lambda\|_{2\to2}
\le e^{\omega_0\lambda^2m},
\end{equation}
and, for all $m\ge1$ and $h\in W$,
\begin{equation}\label{eq:app-generator-weighted-diff-transfer}
\|U_\lambda\partial_hT_R^mU_{-\lambda}\|_{2\to2}
\le C'm^{-1/2}e^{\omega_0\lambda^2m}.
\end{equation}
The second estimate in \eqref{eq:app-two-sided-weighted-2to2-transfer} is
\eqref{eq:app-uniform-weighted-2to2} with $-\lambda$ in place of $\lambda$, and
\eqref{eq:app-generator-weighted-diff-transfer} is
\eqref{eq:app-uniform-weighted-diff-2to2} with a harmless enlargement of
$\omega_0$.

By \eqref{eq:app-one-step-weighted-L2} and
\eqref{eq:app-two-sided-weighted-2to2-transfer}, there is $C_0>0$ such that
\begin{equation}\label{eq:app-crude-weighted-L2}
\|U_\lambda K_R^{(q)}\|_2
=\|U_\lambda T_R^{q-1}U_{-\lambda}(U_\lambda K_R)\|_2
\le C_0 e^{\omega_0\lambda^2 q}
\end{equation}
for all $R\ge R_0$, all $q\in\NN$, and all $|\lambda|\le\lambda_2$.

We claim that there exist constants $C_1>0$ and $\omega_1>0$ such that
\begin{equation}\label{eq:app-general-g-weighted-diff}
\|U_\lambda\partial_g T_R^n U_{-\lambda}\|_{2\to2}
\le C_1\frac{|g|_W}{\sqrt n}e^{\omega_1\lambda^2 n}
\end{equation}
whenever $R\ge R_0$, $n\ge1$, $|\lambda|\le\lambda_2$, and $|g|_W\le \sqrt n$.
If $g=e$ this is trivial.  Otherwise put $m:=|g|_W$ and write
$g=s_1\cdots s_m$ with $s_j\in W$.  Telescoping gives
$$
\partial_g=\partial_{s_1}+\sum_{j=2}^m L(s_1\cdots s_{j-1})\partial_{s_j}.
$$
Also
$$
\|U_\lambda L(a)U_{-\lambda}\|_{2\to2}\le e^{|\lambda||a|_W} \qquad(a\in G),
$$
because $||x|_W-|a^{-1}x|_W|\le |a|_W$.  Therefore, using
\eqref{eq:app-generator-weighted-diff-transfer},
\begin{align*}
\|U_\lambda\partial_g T_R^n U_{-\lambda}\|_{2\to2}
&\le C' n^{-1/2}e^{\omega_0\lambda^2 n}
\Bigl(1+\sum_{j=2}^m e^{|\lambda||s_1\cdots s_{j-1}|_W}\Bigr)\\
&\le C'm n^{-1/2}e^{|\lambda|m}e^{\omega_0\lambda^2 n}.
\end{align*}
Since $m\le\sqrt n$, the inequality $ab\le \delta a^2+b^2/(4\delta)$, with
$a=|\lambda|\sqrt n$, $b=m/\sqrt n$, and $\delta=\omega_0$, gives
$|\lambda|m\le \omega_0\lambda^2n+1/(4\omega_0)$.  Thus
\eqref{eq:app-general-g-weighted-diff} holds with
$C_1:=C'e^{1/(4\omega_0)}$ and $\omega_1:=2\omega_0$.

\smallskip\noindent
Set $C_3:=2^{D/4}C_1$.  Choose $\eps_*\in(0,1]$ so small that
$C_3\eps_*\le\frac12$, and let
$n_0:=\lceil\eps_*^{-2}\rceil$ so that $\eps_*\sqrt n\ge1$ for all $n\ge n_0$.
For $n\ge n_0$, set $r=\eps_*\sqrt n\ge1$.  Since $K_R$ is finitely supported,
$K_R^{(2n)}$ is finitely supported, so \Cref{lem:weighted-finite-support-app} applies to
$f=K_R^{(2n)}$ and gives
\begin{align*}
\|U_\lambda K_R^{(2n)}\|_2
&\le \sup_{|g|_W\le r}\|U_\lambda(I-L(g))K_R^{(2n)}\|_2
+e^{|\lambda|r}|B_r|^{-1/2}\|U_{-\lambda}T_R^{2n}U_\lambda\|_{2\to2}.
\end{align*}
For the first term, write $K_R^{(2n)}=T_R^nK_R^{(n)}$ and use
\eqref{eq:app-general-g-weighted-diff}:
$$
\sup_{|g|_W\le r}\|U_\lambda(I-L(g))K_R^{(2n)}\|_2 \le C_1\frac{r}{\sqrt n}e^{\omega_1\lambda^2 n}\|U_\lambda K_R^{(n)}\|_2.
$$
For the second term, \eqref{eq:app-two-sided-weighted-2to2-transfer} gives
$$
\|U_{-\lambda}T_R^{2n}U_\lambda\|_{2\to2}\le e^{2\omega_0\lambda^2 n}.
$$
Since $G$ has polynomial growth of degree $D$ with respect to $W$, there exists
$c_V>0$ such that $|B_r|\ge c_V r^D$ for all $r\ge1$.  Consequently,
\begin{equation}\label{eq:app-weighted-L2-recursion-raw}
\|U_\lambda K_R^{(2n)}\|_2
\le C_1\frac{r}{\sqrt n}e^{\omega_1\lambda^2 n}\|U_\lambda K_R^{(n)}\|_2
+\frac{1}{\sqrt{c_V}}\Bigl(\frac{r}{\sqrt n}\Bigr)^{-D/2}n^{-D/4}e^{|\lambda|r+2\omega_0\lambda^2 n}.
\end{equation}
Choose $\Omega>\max\{\omega_0,\omega_1\}$ and define
$$
\mathcal B_{n,\lambda,R}:= n^{D/4}e^{-\Omega\lambda^2 n}\|U_\lambda K_R^{(n)}\|_2.
$$
Multiplying \eqref{eq:app-weighted-L2-recursion-raw} by
$(2n)^{D/4}e^{-2\Omega\lambda^2 n}$ gives
$$
\mathcal B_{2n,\lambda,R} \le C_3\Bigl(\frac{r}{\sqrt n}\Bigr)\mathcal B_{n,\lambda,R} +C_4\Bigl(\frac{r}{\sqrt n}\Bigr)^{-D/2} e^{|\lambda|r-2(\Omega-\omega_0)\lambda^2 n},
$$
where $C_4:=2^{D/4}/\sqrt{c_V}$.  With $r=\eps_*\sqrt n$ and
$C_3\eps_*\le\frac12$,
$$
\mathcal B_{2n,\lambda,R} \le \frac12\mathcal B_{n,\lambda,R} +C_4\eps_*^{-D/2}e^{\eps_*|\lambda|\sqrt n-2(\Omega-\omega_0)\lambda^2 n}.
$$
The exponent on the right is uniformly bounded above because
$$
\eps_*|\lambda|\sqrt n-2(\Omega-\omega_0)\lambda^2 n \le \frac{\eps_*^2}{8(\Omega-\omega_0)}.
$$
Hence there exists $C_5>0$ such that
\begin{equation}\label{eq:app-dyadic-recurrence}
\mathcal B_{2n,\lambda,R}\le \frac12\mathcal B_{n,\lambda,R}+C_5
\end{equation}
for all $R\ge R_0$, all $n\ge n_0$, and all $|\lambda|\le\lambda_2$.

\smallskip\noindent
By \eqref{eq:app-crude-weighted-L2},
$$
\sup_{1\le n\le 2n_0}\sup_{R\ge R_0}\sup_{|\lambda|\le\lambda_2} \mathcal B_{n,\lambda,R}<\infty .
$$
Iterating \eqref{eq:app-dyadic-recurrence} gives
$$
\sup_{k\ge0}\sup_{R\ge R_0}\sup_{|\lambda|\le\lambda_2}\mathcal B_{2^k,\lambda,R}<\infty.
$$
Now let $q\in\NN$, and choose $k\ge0$ so that $2^k\le q<2^{k+1}$.
Write $m:=2^k$ and $q=m+\ell$ with $0\le \ell<m$.  Then
$$
U_\lambda K_R^{(q)} =U_\lambda T_R^\ell U_{-\lambda}(U_\lambda K_R^{(m)}),
$$
so \eqref{eq:app-two-sided-weighted-2to2-transfer} gives
$$
\|U_\lambda K_R^{(q)}\|_2 \le e^{\omega_0\lambda^2\ell}\|U_\lambda K_R^{(m)}\|_2.
$$
Therefore
$$
q^{D/4}e^{-\Omega\lambda^2 q}\|U_\lambda K_R^{(q)}\|_2 \le \Bigl(\frac{q}{m}\Bigr)^{D/4}e^{-(\Omega-\omega_0)\lambda^2\ell} \mathcal B_{m,\lambda,R} \le 2^{D/4}\mathcal B_{m,\lambda,R}.
$$
The bound on the right-hand side is uniform, and this proves
\eqref{eq:app-uniform-weighted-kernel-L2-transfer} with $\omega=\Omega$.
\end{proof}

\begin{proof}[Proof of \Cref{lem:uniform-dungey-family}]
We first prove that there exist constants $a>0$ and $C>0$, depending only on
$(G,W,\eta,\beta,\mathfrak M_\beta,\kappa)$, such that for every $q\in\NN$ and every $h\in W$,
\begin{equation}\label{eq:app-diag-weighted-kernel-L2}
\bigl\|e^{a\rho/\sqrt q}K_{R_q}^{(q)}\bigr\|_2\le C q^{-D/4}.
\end{equation}
\begin{equation}\label{eq:app-diag-weighted-diff-L2}
\bigl\|e^{a\rho/\sqrt q}\partial_h K_{R_q}^{(q)}\bigr\|_2\le C q^{-1/2-D/4},
\end{equation}
Put $\lambda_2:=\min\{\lambda_0,\beta/2\}$. Choose $a\in(0,\min\{1,\lambda_2\}/2]$, and set $\lambda_q:=a/\sqrt q$.
Then \eqref{eq:app-uniform-weighted-kernel-L2-transfer} applied to $R=R_q$ with $\lambda=\lambda_q$ gives
$$
\bigl\|e^{\lambda_q\rho}K_{R_q}^{(q)}\bigr\|_2 \le C q^{-D/4}e^{\omega\lambda_q^2 q} = C e^{\omega a^2}q^{-D/4},
$$
which is \eqref{eq:app-diag-weighted-kernel-L2} after absorbing $e^{\omega a^2}$ in $C$.

For \eqref{eq:app-diag-weighted-diff-L2}, we first address the case $q=1$.
For $h\in W$,
$$
\partial_h K_{R_1}=L(h)K_{R_1}-K_{R_1}.
$$
Since $|hy|_W\le |h|_W+|y|_W$ and $|h|_W=1$ for $h\in W$,
$$
\|e^{a\rho}L(h)K_{R_1}\|_2\le e^a\|e^{a\rho}K_{R_1}\|_2.
$$
The one-step bound \eqref{eq:app-one-step-weighted-L2}, applied with
$|a|\le\beta/2$, gives
$$
\|e^{a\rho}\partial_h K_{R_1}\|_2 \le (e^a+1)\|e^{a\rho}K_{R_1}\|_2\le C.
$$
This is the desired bound for $q=1$ after enlarging $C$.

Assume henceforth $q\ge2$ and write
$$
q=m+n,\qquad m:=\left\lceil\frac q2\right\rceil,\qquad n:=q-m.
$$
Since $T_{R_q}^q=T_{R_q}^mT_{R_q}^n$ as operators, we have $\partial_h T_{R_q}^q=(\partial_h T_{R_q}^m)T_{R_q}^n$, and hence
$$
e^{\lambda_q\rho}\partial_h K_{R_q}^{(q)} =\bigl(e^{\lambda_q\rho}\partial_h T_{R_q}^m e^{-\lambda_q\rho}\bigr) \bigl(e^{\lambda_q\rho}K_{R_q}^{(n)}\bigr).
$$
Applying \eqref{eq:app-uniform-weighted-diff-2to2} with $R=R_q$, $\lambda=\lambda_q$, $\psi=\rho$, and convolution time $m$ gives
$$
\bigl\|e^{\lambda_q\rho}\partial_h T_{R_q}^m e^{-\lambda_q\rho}\bigr\|_{2\to2} \le Cm^{-1/2}e^{\omega\lambda_q^2 m},
$$
and \eqref{eq:app-uniform-weighted-kernel-L2-transfer} with the same $R$ and $\lambda$ and convolution time $n$ gives
$$
\bigl\|e^{\lambda_q\rho}K_{R_q}^{(n)}\bigr\|_2 \le Cn^{-D/4}e^{\omega\lambda_q^2 n};
$$
where both constants are independent of $q$.
Multiplying and using $m+n=q$ together with $\lambda_q^2 q=a^2$,
$$
\bigl\|e^{\lambda_q\rho}\partial_h K_{R_q}^{(q)}\bigr\|_2 \le Cm^{-1/2}n^{-D/4}e^{\omega\lambda_q^2(m+n)} =Cm^{-1/2}n^{-D/4}e^{\omega a^2} \le C'q^{-1/2-D/4},
$$
for some constant $C'>0$. This is
\eqref{eq:app-diag-weighted-diff-L2}.
Finally, by Cauchy-Schwarz inequality,
$$
\|\partial_h K_{R_q}^{(q)}\|_1 =\sum_{g\in G}e^{-\lambda_q|g|_W}e^{\lambda_q|g|_W}|\partial_h K_{R_q}^{(q)}(g)| \le\Bigl(\sum_{g\in G}e^{-2\lambda_q|g|_W}\Bigr)^{1/2} \bigl\|e^{\lambda_q\rho}\partial_h K_{R_q}^{(q)}\bigr\|_2,
$$
and similarly
$$
\sum_{g\in G}|g|_W|\partial_h K_{R_q}^{(q)}(g)| \le\Bigl(\sum_{g\in G}|g|_W^2 e^{-2\lambda_q|g|_W}\Bigr)^{1/2} \bigl\|e^{\lambda_q\rho}\partial_h K_{R_q}^{(q)}\bigr\|_2.
$$
If $B_r=\{ g\in G\mid |g|_W\le r\}$ with $B_{-1}=\emptyset$, we have
$$
\sum_{g\in G} e^{-2\lambda_q|g|_W} =\sum_{r=0}^\infty(|B_r|-|B_{r-1}|)e^{-2\lambda_q r} =(1-e^{-2\lambda_q})\sum_{r= 0}^\infty|B_r|e^{-2\lambda_q r}.
$$
Polynomial growth of $(G,W)$ of degree $D$ gives $|B_r|\le C_V(r+1)^D$ for all
$r\ge0$, for some $C_V>0$ and since $\lambda_q\le a\le 1/2$, we have the following estimate
$$
\sum_{r=0 }^\infty(r+1)^D e^{-2\lambda_q r}\le C'\lambda_q^{-(D+1)}.
$$
Indeed, for $0<\lambda_q\le1/2$,
$$
(r+1)^D e^{-2\lambda_q r} \le C_D\lambda_q^{-D}\bigl(1+(\lambda_q r)^D\bigr)e^{-2\lambda_q r} \le C_D'\lambda_q^{-D}e^{-\lambda_q r},
$$
for some constants $C_D, C_D'>0$ depending on $D$ because $x^D e^{-x}$ is bounded on $[0,\infty)$.  Summing the resulting
geometric series gives
$$
\sum_{r=0}^\infty(r+1)^D e^{-2\lambda_q r} \le C_D'\lambda_q^{-D}\sum_{r=0}^\infty e^{-\lambda_q r} \le C'\lambda_q^{-(D+1)},
$$
where $C'=2C_D'$. Combined with $1-e^{-2\lambda_q}\le 2\lambda_q$ this produces
\begin{equation}\label{eq:app-shell-alpha-zero}
\sum_{g\in G}e^{-2\lambda_q|g|_W}\le C\lambda_q^{-D}\qquad(C=2C_VC'').
\end{equation}
For the weighted sum, we split
$e^{-2\lambda_q|g|_W}=e^{-\lambda_q|g|_W}\cdot e^{-\lambda_q|g|_W}$ and
use the bound
$\max_{x\ge 0}x^2 e^{-\lambda_q x}=4 e^{-2}\lambda_q^{-2}$:
$$
|g|_W^2 e^{-2\lambda_q|g|_W} \le 4 e^{-2}\lambda_q^{-2}e^{-\lambda_q|g|_W}.
$$
Applying \eqref{eq:app-shell-alpha-zero} with $\lambda_q/2$ in place of
$\lambda_q$ (permissible since $\lambda_q/2\le 1/4$) then gives
$\sum_g e^{-\lambda_q|g|_W}\le C\lambda_q^{-D}$, hence
$$
\sum_{g\in G}|g|_W^2 e^{-2\lambda_q|g|_W} \le C\lambda_q^{-(D+2)}.
$$
Since $\lambda_q=a/\sqrt q$,
$$
\sum_{g\in G}e^{-2\lambda_q|g|_W}\le Cq^{D/2}, \qquad \sum_{g\in G}|g|_W^2 e^{-2\lambda_q|g|_W}\le Cq^{D/2+1}.
$$
Combining these estimates with
\eqref{eq:app-diag-weighted-diff-L2} yields \eqref{eq:app-diag-generator-bounds}.
\end{proof}

\begin{proof}[Proof of \Cref{thm:B}]
After adjoining $e$ to $W$ if necessary, the proof of \Cref{lem:uniform-dungey-family} with $R_q$ replaced throughout by an arbitrary $R\ge R_0$ gives all the estimates
in \Cref{thm:B}, uniformly in $R$.
\end{proof}

\section*{Acknowledgements}
The research of the first author was partially supported by SEED Grant RD/0519-IRCCSH0-024. During the initial stages of the preparation of the manuscript the first author was a visitor at MPIM Bonn. The second author would like to thank the PMRF for partially supporting his work.
The third author was partially supported by IIT Bombay IRCC fellowship, TIFR Mumbai post-doctoral fellowship and NBHM postdoctoral fellowship (Sr. No. 0204/17/2025/R\&D-II/12398) during this work. The authors are deeply grateful to Gideon Amir, Tom Meyerovitch and Ariel Yadin for very insightful correspondence. All three authors would like to thank IIT Bombay for providing ideal working conditions.

\bibliographystyle{alpha}

\bibliography{references1}

\end{document}